\documentclass[12pt, a4paper]{article}
\usepackage{amsmath, amssymb, amscd}
\usepackage{amsthm}
\usepackage{layout}
\usepackage{bm}
\usepackage{cases}
\usepackage{color}
\usepackage{comment} %to treat comment env
\usepackage{empheq}
\usepackage{graphicx}
\usepackage{mathrsfs}
\usepackage{mathtools} %to treat multiple cases
\usepackage[all]{xy}
\usepackage{latexsym}
\usepackage{enumerate}
\usepackage{url}
\usepackage[shortlabels]{enumitem}
\setlist{nosep} %vertical spacing of items in various list environments

\allowdisplaybreaks[2]

\newtheorem{Thm}{Theorem}[section]
\newtheorem{lem}[Thm]{Lemma}
\newtheorem{Prop}[Thm]{Proposition}
\newtheorem{Cor}[Thm]{Corollary}
\newtheorem{Rem}[Thm]{Remark}
\newtheorem{ex}[Thm]{Example}
\newtheorem*{Thmn}{Main Theorem}

\theoremstyle{definition}
\newtheorem{Def}[Thm]{Definition}

\makeatletter

\@addtoreset{equation}{section}
\makeatother

\makeatletter

\makeatletter
\def\BOXSYMBOL{\RIfM@\bgroup\else$\bgroup\aftergroup$\fi
  \vcenter{\hrule\hbox{\vrule height.85em\kern.6em\vrule}\hrule}\egroup}
\makeatother
\newcommand{\BOX}{%
  \ifmmode\else\leavevmode\unskip\penalty9999\hbox{}\nobreak\hfill\fi
  \quad\hbox{\BOXSYMBOL}}

\newcommand{\eps}{\varepsilon}

\newcommand{\bbA}{\mathbb{A}}

\newcommand{\bbC}{\mathbb{C}}

\newcommand{\bbF}{\mathbb{F}}
\newcommand{\bbG}{\mathbb{G}}

\newcommand{\bbP}{\mathbb{P}}
\newcommand{\bbQ}{\mathbb{Q}}
\newcommand{\bbR}{\mathbb{R}}

\newcommand{\bbZ}{\mathbb{Z}}

\newcommand{\calC}{\mathcal{C}}

\newcommand{\calG}{\mathcal{G}}

\newcommand{\calM}{\mathcal{M}}
\newcommand{\calN}{\mathcal{N}}
\newcommand{\calO}{\mathcal{O}}

\newcommand{\calQ}{\mathcal{Q}}

\newcommand{\calS}{\mathcal{S}}

\newcommand{\calX}{\mathcal{X}}
\newcommand{\calY}{\mathcal{Y}}
\newcommand{\calZ}{\mathcal{Z}}

\newcommand{\frakF}{\mathfrak{F}}

\newcommand{\frakg}{\mathfrak{g}}

\newcommand{\frakI}{\mathfrak{I}}

\newcommand{\frakm}{\mathfrak{m}}

\newcommand{\frakP}{\mathfrak{P}}
\newcommand{\frakp}{\mathfrak{p}}

\newcommand{\frakS}{\mathfrak{S}}

\newcommand{\scrA}{\mathscr{A}}
\newcommand{\scrL}{\mathscr{L}}
\newcommand{\scrX}{\mathscr{X}}
\newcommand{\scrY}{\mathscr{Y}}
\newcommand{\scrZ}{\mathscr{Z}}

\newcommand{\Ad}{\mathop{\mathrm{Ad}}\nolimits}
\newcommand{\Arf}{\mathrm{Arf}}
\newcommand{\Art}{\mathrm{Art}}
\newcommand{\Aut}{\mathrm{Aut}}
\newcommand{\Frob}{\mathrm{Frob}}
\newcommand{\GL}{\mathrm{GL}}

\newcommand{\Gal}{\mathop{\mathrm{Gal}}\nolimits}
\newcommand{\Ker}{\mathop{\mathrm{Ker}}\nolimits}
\newcommand{\Lie}{\mathop{\mathrm{Lie}}\nolimits}
\newcommand{\diag}{\mathop{\mathrm{diag}}\nolimits}
\newcommand{\tr}{\mathop{\mathrm{tr}}\nolimits}
\newcommand{\Tr}{\mathop{\mathrm{Tr}}\nolimits}

\newcommand{\Trd}{\mathop{\mathrm{Trd}}\nolimits}
\newcommand{\Nrd}{\mathop{\mathrm{Nrd}}\nolimits}
\newcommand{\id}{\mathop{\mathrm{id}}\nolimits}
\newcommand{\cInd}{\mathop{\mathrm{c\mathchar`-Ind}}\nolimits}
\newcommand{\Ind}{\mathop{\mathrm{Ind}}\nolimits}

\newcommand{\rad}{\mathop{\mathrm{rad}}\nolimits}

\newcommand{\sgn}{\mathop{\mathrm{sgn}}\nolimits}
\newcommand{\LL}{\mathop{\mathrm{LL}}\nolimits}
\newcommand{\LJ}{\mathop{\mathrm{LJ}}\nolimits}
\newcommand{\Hom}{\mathop{\mathrm{Hom}}\nolimits}
\newcommand{\End}{\mathop{\mathrm{End}}\nolimits}
\newcommand{\Cont}{\mathop{\mathrm{Cont}}\nolimits}

\newcommand{\Spa}{\mathop{\mathrm{Spa}}\nolimits}
\newcommand{\Spec}{\mathop{\mathrm{Spec}}\nolimits}
\newcommand{\Spf}{\mathop{\mathrm{Spf}}\nolimits}
\newcommand{\Stab}{\mathop{\mathrm{Stab}}\nolimits}

\newcommand{\Nil}{\mathop{\mathrm{Nil}}\nolimits}
\newcommand{\Adic}{\mathop{\mathrm{Adic}}\nolimits}

\newcommand{\ba}{\bm{a}}
\newcommand{\bc}{\bm{c}}
\newcommand{\bd}{\bm{d}}
\newcommand{\bg}{\bm{g}}
\newcommand{\br}{\bm{r}}
\newcommand{\bt}{\bm{t}}
\newcommand{\bT}{\bm{T}}
\newcommand{\bu}{\bm{u}}
\newcommand{\bU}{\bm{U}}
\newcommand{\bw}{\bm{w}}
\newcommand{\bx}{\bm{x}}
\newcommand{\bX}{\bm{X}}
\newcommand{\by}{\bm{y}}
\newcommand{\bY}{\bm{Y}}
\newcommand{\bz}{\bm{z}}
\newcommand{\bZ}{\bm{Z}}
\newcommand{\bzeta}{\bm{\zeta}}
\newcommand{\bxi}{\bm{\xi}}

\newcommand{\lv}{\left\lvert}
\newcommand{\rv}{\right\rvert}

\newcommand{\Sets}{{\mathrm{Sets}}}
\newcommand{\LTp}{\calM _{H_0, \infty, \overline{\eta}}^{\text{ad}}}
\newcommand{\LTpone}{\calM _{\wedge H_0, \infty, \overline{\eta}}^{\text{ad}}}

\begin{document}
\title{Affinoids in the Lubin-Tate perfectoid space and special cases of the local Langlands correspondence}
\author{Kazuki Tokimoto}
\date{}
\maketitle


\begin{abstract}
Following Weinstein, Boyarchenko-Weinstein and Imai-Tsushima,
we construct a family of affinoids in the Lubin-Tate perfectoid space
and their formal models 
such that the cohomology of the reduction of each formal model realizes 
the local Langlands correspondence and the local Jacquet-Langlands correspondence
for certain representations.
In the terminology of the essentially tame local Langlands correspondence,
the representations treated here are characterized 
as being parametrized by 
minimal admissible pairs 
in which the field extensions are totally ramified.
\end{abstract}

\tableofcontents
\section*{Introduction}
 Let $K$ be a non-archimedean local field
 with finite residue field $k$.
 We write $p$ for the characteristic of $k$
 and $q$ the cardinality.
 Let $W_K$ be the Weil group of $K$
 and $D$ the central division algebra over $K$
 of invariant $1/n$.
 Denote by $\calO _K \subset K$ the valuation ring
 and by $\frakp \subset \calO _K$ the maximal ideal.
 We fix an algebraic closure $\overline{k}$ of $k$.
 Let $n\geq 1$ be an integer.
 Then 
 the Lubin-Tate spaces
 are defined to be certain deformation spaces
 of a one-dimensional formal $\calO _K$-module
 over $\overline{k}$
 of height $n$
 with level structures.
 The Lubin-Tate spaces 
 naturally form a projective system,
 called the Lubin-Tate tower,
 and
 the non-abelian Lubin-Tate theory asserts that 
 the $\ell$-adic cohomology of the Lubin-Tate tower,
 which admits a natural action of
 a large subgroup of
 $\GL _n(K)\times D^{\times}\times W_K$,
 realizes the local Langlands correspondence for $\GL _n(K)$
 and 
 the local Jacquet-Langlands correspondence
 simultaneously.
 However, as the proofs of this fact (\cite{BoyMDr}, \cite{HTsimSh})
 make heavy use of the theory of automorphic representations
 and the global geometry,
 the geometry of the Lubin-Tate spaces 
 and its relation to representations
 are not yet fully understood.
 
 Among the studies on the geometry of the Lubin-Tate spaces
 is work \cite{YoLTv} of Yoshida.
 There he constructed a semistable model 
 of the Lubin-Tate space of level $\frakp$
 and proved that 
 an affine open subscheme of the reduction
 is isomorphic to a Deligne-Lusztig variety for $\GL _n(k)$.
 The appearance of the Deligne-Lusztig variety
 reflects the fact that 
 some irreducible supercuspidal representations of $\GL _n(K)$
 can be constructed from irreducible cuspidal representations
 of $\GL _n(k)$.
 Note that this open subscheme can also be
 obtained as the reduction of an affinoid subspace
 in the Lubin-Tate space
 (see \cite{WeGred} for details in equal-characteristic).

 More recently, 
 Weinstein showed in \cite{WeSemi}
 that a certain limit space 
 of the Lubin-Tate tower
 makes sense as a perfectoid space.
 While it is no longer an ordinary finite-type analytic space,
 the Lubin-Tate perfectoid space has a simpler geometry;
 with a coordinate not available on the individual Lubin-Tate spaces,
 the defining equation is simpler 
 and the group actions can be made more explicit.
 Taking advantage of these properties,
 Weinstein \cite{WeSemi}, Boyarchenko-Weinstein \cite{BWMax}
 and Imai-Tsushima \cite{ITepitame}, \cite{ITsimpwild}
 constructed families of affinoid subspaces 
 in the Lubin-Tate perfectoid space
 and their formal models
 such that the cohomology of the reduction 
 of each formal model realizes the local Langlands 
 and Jacquet-Langlands correspondences
 for some representations. 
 The aim of this paper is to construct
 such a family of affinoids
 related to certain other representations.
 
 Let $H_0$ be the formal $\calO _K$-module over $\overline{k}$ of height $n$.
 Let $C$ be the completion of an algebraic closure of $K$.
 We denote the Lubin-Tate perfectoid space by $\calM _{H_0, \infty, \overline{\eta}}^{\text{ad}}$,
 which is an adic space over $C$.
 Let $\ell \neq p$ be a prime number.
 We fix an isomorphism $\overline{\bbQ} _{\ell}\simeq \bbC$.
 Set $G=\GL _n(K)\times D^{\times}\times W_K$.
 Here is our main theorem (recall that we make no assumption on the characteristic of $K$):
 \begin{Thmn}
  Suppose that $p$ does not divide $n$.
  Let $\nu >0$ be an integer which is coprime to $n$.
  Let $L/K$ be a totally ramified extension of degree $n$.
  Then there exist an affinoid \footnote{
  In this paper
  by an \emph{affinoid}
  we mean an open affinoid adic subspace.
  Also we often refer to the special fiber of a formal scheme $\scrA$ over $\calO _C$
  simply as the \emph{reduction} of $\scrA$.
  } 
  $\calZ _{\nu}\subset \calM _{H_0, \infty, \overline{\eta}}^{\emph{ad}}$
  and a formal model $\scrZ _{\nu}$ of $\calZ _{\nu}$
  such that the following hold.
  \begin{enumerate}[(1)]
   \item The stabilizer $\Stab _{\nu}$
   of $\calZ _{\nu}$ naturally acts on the formal model $\scrZ _{\nu}$
   and hence on the reduction $\overline{\scrZ} _{\nu}$.
   \item \label{item:mainthm}
           For an irreducible smooth representation $\pi$ of $\GL _n(K)$,
           we have
           \[
           \Hom _{\GL _n(K)}\left( \cInd _{\Stab _{\nu}}^{G} H_c^{n-1}
                                                                    \left( \overline{\scrZ} _{\nu}, \overline{\bbQ} _{\ell} \right) ((n-1)/2), 
                                                                    \pi \right)\neq 0
           \]
           if and only if the image $\tau$ of $\pi$ under the local Langlands correspondence
           is a character twist of an $n$-dimensional irreducible smooth representation
           of the form $\Ind _{L/K}\xi$
           for a character $\xi$ of $L^{\times}$
           which is non-trivial on $U_L^{\nu}$, but trivial on $U_L^{\nu +1}$.
           Moreover, if the above space is non-zero,
           it is isomorphic to $\rho \boxtimes \tau$
           as a representation of $D^{\times}\times W_K$,
           where $\rho$ is the image of $\pi$ under the local Jacquet-Langlands correspondence.
  \end{enumerate}
 \end{Thmn}
 Thus the cohomology of $\overline{\scrZ} _{\nu}$
 realizes the local Langlands correspondence and 
 the local Jacquet-Langlands correspondence
 for representations $\pi$ characterized by the condition 
 in \ref{item:mainthm}.
 
 Although Main Theorem simply asserts the existence,
 our proof is in fact very explicit;
 we construct
 $\calZ _{\nu}$, $\scrZ _{\nu}$,
 and compute $\overline{\scrZ} _{\nu}$, $\Stab _{\nu}$ and
 the cohomology.
 See, for instance,
 Subsection \ref{subsec:defofaf} for the definition of $\calZ _{\nu}$
 and Theorem \ref{Thm:reduction} for the description of $\overline{\scrZ} _{\nu}$.
 As will be further explained later in this introduction,
 the reductions $\overline{\scrZ} _{\nu}$ are of quite different flavor 
 from any of those appearing in the preceding work \cite{BWMax}, \cite{WeSemi}, \cite{ITepitame}, \cite{ITsimpwild} 
 if $\nu$ is even,
 and consequently the computation of their cohomology is based on new ideas.
 
 We remark that
 if $K$ is of equal-characteristic Main Theorem specializes to the main theorem of the previous version of this paper. 
 The computation of the reductions and the stabilizers (Section \ref{sec:affandred}) is more elaborate in this version, 
 whereas in fact the description of the reduction (Theorem \ref{Thm:reduction}), 
 and hence the arguments concerning the cohomology (Section \ref{sec:cohofred}) 
 and the representations (Section \ref{sec:real}) remain essentially the same.
 
 Main Theorem does not assert anything about
 the relation 
 between the cohomology of each reduction $\overline{\scrZ} _{\nu}$
 and that of the Lubin-Tate tower.
 Thus our result alone does not yield direct consequences
 for the non-abelian Lubin-Tate theory.
 We find it interesting nonetheless.
 Moreover, after the previous version of this paper appeared,
 Mieda \cite{MiGeomLLC} obtained a useful sufficient condition for the cohomology of the reduction of such a formal model 
 to contribute to the cohomology of the Lubin-Tate tower (see \cite{MiGeomLLC} for the precise formulation and statement)
 and verified that it holds in our situation.
 The same remarks apply to other preceding results \cite{BWMax}, \cite{WeSemi}, \cite{ITepitame}, \cite{ITsimpwild}.

 \vspace{\baselineskip} 
 Let us compare Main Theorem
 with the preceding results
 by listing the class of representations realized in the cohomology
 (cf.\ Remark \ref{Rem:minimalpairs} \ref{item:precedingresults}).
 In \cite{BWMax}, \cite{ITepitame}, \cite{WeSemi} as well as this paper,
 an irreducible smooth representation $\pi$ of $\GL _n(K)$ is realized in the cohomology of the reduction
 if and only if its local Langlands correspondent $\tau$ is of the form 
 $\Ind _{F/K}\xi$
 for a tamely ramified extension $F/K$ of degree $n$ 
 and a \emph{minimal} character $\xi$ of $F^{\times}$
 subject to the condition depending on each result.
 A minimal character has a numerical invariant called the \emph{jump} $\nu$ of $\xi$,
 which is a non-negative integer coprime to the ramification index $e$ of $F/K$
 (see Definition \ref{def:minimalpair} for the definitions of minimality and jump;
 these are essentially taken from the work \cite{BHetLLCIII} of Bushnell-Henniart).
 The extensions $F/K$ and the minimal characters $\xi$ allowed in each result are 
 characterized by the possible values of $n$, $e$, $\nu$ indicated in the following table: \footnote{
 In fact \cite{BWMax} only treats positive $\nu$. However, $\nu =0$ case is in a sense the simplest
 and seems to be well-known to experts.
 The Deligne-Lusztig variety appears in the reduction as in \cite{YoLTv}, \cite{WeGred}.
 See \cite[\S5]{MiGeomLLC} for an account in the perfectoid setting.}
 \begin{center}
 \begin{tabular}{|c|c|c|c|} \hline
         & $n$ &  $e$ & $\nu$ \\ \hline\hline
 this paper         & general & $n$   & general \\ \hline
 \cite{BWMax}    & general & $1$   & general \\ \hline
 \cite{ITepitame} & general & $n$   & $1$ \\ \hline
 \cite{WeSemi}    & $2$      & $1, 2$ & general \\ \hline 
 \end{tabular}
 \end{center}
 Thus, Main Theorem is the ``totally ramified version'' of \cite{BWMax},
 and generalizes \cite{ITepitame}
 and the ``totally ramified part'' of \cite{WeSemi}.
 We remark that \cite{ITepitame} has a companion paper
 \cite{ITsimpwild}
 and together they cover (character twists of) \emph{simple supercuspidal representations}.
 If $p$ divides $n$ as assumed in \cite{ITsimpwild},
 then $\tau$ is not of the above form and accordingly the analysis is far more subtle.
 Also, the aim of \cite{WeSemi} is related to but different from those of the other papers listed above,
 and the affinoids and the formal models are obtained in the course of proving the main result.
 
 The author learned from
 Imai and Tsushima
 that they had previously constructed
 what should be $\calZ _{2}$ and $\scrZ _{2}$
 in our notation,
 computed the reduction
 and verified the non-triviality of the middle-degree cohomology.
 Although this unannounced
 result
 preceded ours,
 our result was obtained independently.
 On a related note,
 in \cite{ITLTsurf}
 they study
 certain affinoids
 in the Lubin-Tate space of level $\frakp ^2$
 when $K$ is of equal-characteristic
 and $n=3$.
 These should be regarded as a finite-level counterpart
 specialized to $(n, \nu)=(3, 2)$ and $K$ equal-characteristic.
 
 \vspace{\baselineskip}
 The proof of Main Theorem naturally consists of the following four steps:
 \begin{enumerate}[(a)]
 \item constructing affinoids and formal models; \label{step1}
 \item computing the reduction and the stabilizer along with the induced action; \label{step2}
 \item computing the cohomology as a representation; \label{step3}
 \item comparing the resulting representation with the local Langlands correspondence and the local Jacquet-Langlands correspondence. \label{step4}
 \end{enumerate}
 Let us describe the key ideas in each steps,
 focusing mainly on the first two.
 
 \paragraph{Step (a)}
 In constructing affinoids and formal models 
 we work with a Cartesian diagram of formal schemes and the induced diagram of the adic generic fibers:
 \[
 \xymatrix{
 \calM _{H_0, \infty, \calO _C} \ar[r] \ar[d]_{\alpha _n} &\calM _{\wedge H_0, \infty, \calO _C} \ar[d]^{\alpha _1} \\
 \Nil ^{\flat, n}_{\calO _C} \ar[r]^{\Delta} &\Nil ^{\flat}_{\calO _C},
 }
 \qquad
 \xymatrix{
 \calM _{H_0, \infty, \overline{\eta}}^{\text{ad}} \ar[r] \ar[d]_{\alpha _n} &\calM _{\wedge H_0, \infty, \overline{\eta}}^{\text{ad}} \ar[d]^{\alpha _1} \\
 \Nil ^{\flat, n, \text{ad}}_{\overline{\eta}} \ar[r]^{\Delta} &\Nil ^{\flat, \text{ad}}_{\overline{\eta}}.
 }
 \]
 We refer to Subsection \ref{subsection:formalmodel} for the definitions of the
 objects and morphisms,
 and only remark that 
 \begin{itemize}
 \item $\calM _{H_0, \infty, \calO _C}$ is a formal model of the Lubin-Tate perfectoid space $\calM _{H_0, \infty, \overline{\eta}}^{\text{ad}}$,
 \item $\wedge H_0$ is the formal $\calO _K$-module over $\overline{k}$ of height one,
 \item $\Nil ^{\flat, n}_{\calO _C}=\Spf \calO _C[[X_1^{q^{-\infty}}, \dots, X_n^{q^{-\infty}}]]$ and $\Nil ^{\flat}_{\calO _C}=\Spf \calO _C[[T^{q^{-\infty}}]]$ are certain $K$-vector space objects 
 (we define $\calO _C[[X_1^{q^{-\infty}}, \dots, X_n^{q^{-\infty}}]]$ as the completion of 
 $\calO _C[X_1^{q^{-\infty}}, \dots, X_n^{q^{-\infty}}]=\varinjlim _{X_i\mapsto X_i^q} \calO _C[X_1, \dots, X_n]$
 with respect to an adic topology defined by the ideal generated by $X_1, \dots, X_n$ and 
 $\frakp$, and similarly $\calO _C[[T^{q^{-\infty}}]]$) and
 \item $\Delta$ is a multilinear and alternating morphism called the determinant morphism.
 \end{itemize}
 We take a $C$-valued point $\xi \in \calM _{H_0, \infty, \overline{\eta}}^{\text{ad}}(C)$ with CM by $L/K$ (Subsection \ref{subsec:CMpts})
 and construct an affinoid $\calX _{\nu} \subset \Nil ^{\flat, n, \text{ad}}_{\overline{\eta}}$ around $\alpha _n(\xi)$.
 The affinoid $\calZ _{\nu} \subset \calM _{H_0, \infty, \overline{\eta}}^{\text{ad}}$ is defined as the pull-back
 of $\calX _{\nu}$ in $\calM _{H_0, \infty, \overline{\eta}}^{\text{ad}}$ (Subsection \ref{subsec:defofaf}).
  
 The formal model $\scrZ _{\nu}$ of $\calZ _{\nu}$
 is defined
 as the fibered product involving suitable formal subschemes
 of $\Nil ^{\flat, n}_{\calO _C}$, $\Nil ^{\flat}_{\calO _C}$ and $\calM _{\wedge H_0, \infty, \calO _C}$.
 We first define a natural formal model $\scrX _{\nu}$ of $\calX _{\nu}$
 using the coordinate around $\alpha _n(\xi)$ used to define $\calX _{\nu}$.
 Then we study the set-theoretic image $\Delta (\calX _{\nu}) \subset \Nil ^{\flat, \text{ad}}_{\overline{\eta}}$
 to construct an affinoid $\calY _{\nu} \subset \Nil ^{\flat, \text{ad}}_{\overline{\eta}}$
 around $\Delta \circ \alpha _n(\xi)$ such that $\Delta (\calX _{\nu}) \subset \calY _{\nu}$,
 which allows us to define a natural formal model $\scrY _{\nu}$ of $\calY _{\nu}$.
 We note that $\scrX _{\nu}\simeq \Spf \calO _C\langle x_1^{q^{-\infty}}, \dots, x_n^{q^{-\infty}}\rangle$
 and $\scrY _{\nu}\simeq \Spf \calO _C\langle x^{q^{-\infty}}\rangle$ are rather simple formal schemes.
 Finally, using the Lubin-Tate theory we study $\alpha _1^{-1}(\calY _{\nu}) \subset \calM _{\wedge H_0, \infty, \overline{\eta}}^{\text{ad}}$ to define its formal model.
 
 The technical heart of the whole paper is the study of 
 $\Delta (\calX _{\nu}) \subset \Nil ^{\flat, \text{ad}}_{\overline{\eta}}$ above
 (and the related analysis of $\Delta \colon \scrX _{\nu}\to \scrY _{\nu}$ in Step \ref{step2}).
 Here we were inspired by 
 the $n=2$ case treated by Weinstein in \cite[\S5]{WeSemi}.
 For simplicity, suppose first that $K$ is of equal-characteristic.
 We often collectively write $\bX _i=(X_i^{q^{-l}})_{l\geq 0}$ and $\bT =(T^{q^{-l}})_{l\geq 0}$.\footnote{
 As should be clear from the following discussion, the use of $q$-th power compatible systems of topologically nilpotent elements,
 or equivalently valued points of $\Nil ^{\flat}$,
 is crucial in this paper.
 Accordingly, we employ subtle notation related to $\Nil ^{\flat}$.
 We do not explain such notation in this introduction and refer the reader to the main body of the paper,
 notably to \eqref{eq:nilop0}--\eqref{eq:nilop2} and Subsection \ref{subsec:nilnotation}.
 }
 Via the coordinate $\bT =(T^{q^{-l}})_{l\geq 0}$ of $\Nil ^{\flat, \text{ad}}_{\overline{\eta}}$,
 the point $\Delta \circ \alpha _n(\xi) \in \Nil ^{\flat, \text{ad}}_{\overline{\eta}}(C)$
 and the determinant morphism $\Delta$ define elements
 \[
 t^{q^{-l}}\in C, \qquad \Delta ^{q^{-l}}(\bX _1, \dots, \bX _n) \in \calO _C[[X_1^{q^{-\infty}}, \dots, X_n^{q^{-\infty}}]]
 \]
 for $l\geq 0$.
 To estimate the valuations of the functions $\Delta ^{q^{-l}}(\bX _1, \dots, \bX _n)-t^{q^{-l}}$ on $\calX _{\nu}$
 and to define the desired affinoid $\calY _{\nu}$ using the coordinate $U^{q^{-l}}=T^{q^{-l}}-t^{q^{-l}}$,
 we use a generalization of Weinstein's idea found in \cite[\S5.5, Lemma 2.14]{WeSemi}:
 when substituting $\bX _i$ with a suitable coordinate $\bY _i$ $(2\leq i\leq n)$ and $\bZ$
 around $\alpha _n(\xi) \in \calX _{\nu}$
 and expanding the functions by the additivity of the determinant morphism $\Delta$,
 certain cancellations occur and 
 we are reduced to estimate the valuations of 
 several ``main terms'' expressed in terms of $\Delta$, $\bY _i$ and $\bZ$.
 See \eqref{mainterm} for the cancellations\footnote{
 Strictly speaking, $\Delta (\bX _1, \dots, \bX _n)$ denotes the system $(\Delta ^{q^{-l}}(\bX _1, \dots, \bX _n))_{l\geq 0}$ in \eqref{mainterm}.
 }; 
 assuming $K$ is of equal-characteristic one may drop all $\Nil ^{\flat}_{H_0}$ and $\Nil ^{\flat}_{\wedge H_0}$ in the equation.
 We remark that finding the coordinate $\bY _i$ and $\bZ$ is vital in this argument
 and we achieved that by closely examining the construction of affinoids in the work \cite{ITepitame} of Imai-Tsushima 
 (see \cite[Remark 2.6]{ITepitame} for a comparison of our affinoids for $\nu =1$ and affinoids in \cite{ITepitame}). 
 
 To estimate the valuations of the main terms we apply Lemma \ref{lem:estimate}, 
 which is based on \cite[Lemma 5.7]{WeSemi}.
 Given estimates of valuations of $\bx _1, \dots, \bx _n$,
 this lemma together with Lemma \ref{lem:det} yields an approximation of $\Delta ^{q^{-l}}(\bx _1, \dots, \bx _n)$
 for $l\geq 0$.
 This allows us to obtain the required estimate 
 and define the affinoid $\calY _{\nu}$ and the formal model $\scrY _{\nu}$.
 
 Roughly speaking, when we try to eliminate the assumption on the characteristic,
 all the additions $+$ in the above discussion need to be replaced with the operations $+_{\Nil ^{\flat}_{H_0}}$,
 $+_{\Nil ^{\flat}_{\wedge H_0}}$ coming from those on the respective formal modules.
 It turns out that this causes little difficulty in practice; $\bX +_{\Nil ^{\flat}_{H_0}} \bY$ can be approximated by 
 the naive sum $\bX +\bY$, and in particular can be estimated in the same way.
 This fundamental principle is communicated to the author by Yoichi Mieda
 and is explained in more detail in Subsection \ref{subsec:lemsonnil}.
 
 \paragraph{Step (b)}
 By the construction of the formal model $\scrZ _{\nu}$, 
 the special fiber $\overline{\scrZ} _{\nu}$ is a fibered product in a natural way.
 Again the most laborious part in computing $\overline{\scrZ} _{\nu}$
 is the analysis of the morphism 
 $\overline{\scrX} _{\nu}\to \overline{\scrY} _{\nu}$
 induced by $\Delta \colon \scrX _{\nu} \to \scrY _{\nu}$.
 This goes similarly to the corresponding simpler computations in Step \ref{step1} as follows
 (cf.\ the proof of Theorem \ref{Thm:reduction}):
 Via the coordinate of $\scrY _{\nu}$
 the analysis amounts to approximating explicit functions on $\scrX _{\nu}$ related to 
 ``$\Delta ^{q^{-l}}(\bX _1, \dots, \bX _n)-t^{q^{-l}}$'' above
 (which will be more complicated in mixed-characteristic; in any case see \eqref{eq:moroffms2} for the precise form).
 By the same cancellations
 we have only to approximate the main terms
 expressed in terms of $\Delta, \bY _i$ and $\bZ$, 
 which is done by using Lemmas \ref{lem:det}, \ref{lem:estimate}.
 We note that these lemmas only provide approximations 
 with respect to a fixed valuation on the coordinate ring; 
 to turn such approximations into congruences
 we use Lemma \ref{lem:esttoapp}.
 
 In Theorem \ref{Thm:stabandaction}
 we concretely describe 
 the stabilizer $\Stab _{\nu}$ of $\calZ _{\nu}$
 in terms of the stabilizer $\calS$ of the CM point $\xi$
 and certain subgroups $U_{\frakI}^{(\nu)} \subset \GL _n(K)$, $U_D^{(\nu)}\subset D^{\times}$,
 and also explicitly compute the induced action on the special fiber $\overline{\scrZ} _{\nu}$
 (strictly speaking, we only write down the action of a subgroup of $\Stab _{\nu}$ of finite index;
 this is sufficient for our purpose. See Remark \ref{Rem:Stab} \ref{item:LandL'}).
 Given the formalism of $\Nil ^{\flat}$
 the proof of Theorem \ref{Thm:stabandaction}
 mostly reduces to 
 nearly straightforward computations
 except that we use some subtle properties of $\xi$ and $\calS$ (Subsection \ref{subsec:CMpts}).

 Before discussing Step \ref{step3}
 we note some properties
 of 
 the reductions $\overline{\scrZ} _{\nu}$.
 While only those with $\nu$ coprime to $n$ 
 are relevant to Main Theorem, 
 the affinoids $\calZ _{\nu}$
 and the formal models $\scrZ _{\nu}$
 are constructed for any $\nu >0$
 in a uniform way.
 The reductions $\overline{\scrZ} _{\nu}$
 are closely related to the perfections
 of algebraic varieties $Z_{\nu}$
 (see Theorem \ref{Thm:reduction}),
 which turn out to be periodic
 in $\nu$ with period $2n$.
 They fall into two ``series''
 depending on whether
 $\nu$ is odd or even.
 If $\nu$ is odd,
 $Z_{\nu}$ is the variety
 obtained by pulling back
 the standard Artin-Schreier covering
 $\bbA _{\overline{k}}^1\rightarrow \bbA _{\overline{k}}^1$
 by a morphism $\bbA _{\overline{k}}^{n-1}\rightarrow \bbA _{\overline{k}}^1$
 corresponding to a quadratic form which depends on $\nu$.
 If $\nu$ is even, 
 the defining equation of $Z_{\nu}$ 
 is more involved.
 However,
 it can be described 
 in terms of 
 the Lang torsor of an algebraic group $\calG _{\nu}$
 and a morphism related to a quadratic form 
 (see Subsection \ref{subsec:alggps} for more details). 
 Here the description is modeled on
 a similar description
 found in \cite[\S3.4]{BWMax},
 but the analogy is not so straightforward;
 the relevant algebraic groups are not the same
 and no quadratic forms occur in \cite{BWMax}.
 
 \paragraph{Step (c)}
 We compute the cohomology of $\overline{\scrZ} _{\nu}$ as a representation
 by combining representation-theoretic facts, which are elementary in essence,
 and some standard results in $\ell$-adic cohomolgy.
 If $\nu$ is odd, we are reduced to computing various invariants of the quadratic form
 and we efficiently do so by using some symmetry of the form 
 following Bushnell-Henniart \cite{BHetLLCII}.
 If $\nu$ is even, finite Heisenberg groups occur
 and by their well-known property
 the problem is to show that a suitable cohomology has the prescribed dimension.
 It turns out that the dimension can be computed by applying a result of Deligne in \cite{DelWeilII}.
 In both cases we also use \cite[Theorem 3.2]{DeLu} of Deligne-Lusztig
 to study actions of cyclic groups of order prime to $p$.
 
 \paragraph{Step (d)}
 This step is entirely representation-theoretic.
 The main ingredient is the essentially tame local Langlands and 
 Jacquet-Langlands correspondences
 of Bushnell-Henniart (\cite{BHetLLCI}, \cite{BHetLLCII}, \cite{BHetLLCIII}, \cite{BHetJL}).
 We include a summary of the results specialized to our situation in Subsection \ref{subsec:etc}. 
 
 \vspace{\baselineskip}
 
 We stress that although we did draw much inspiration from the preceding results
 \cite{WeSemi}, \cite{BWMax}, \cite{ITepitame}, 
 new phenomena arise and different techniques are required at several points in this paper.
 For instance, let us look at the computation of the cohomology.
 Recall that affinoids and formal models in Main Theorem are (partial) generalizations of those in \cite{ITepitame}
 and \cite{WeSemi}.
 However, $\nu$ is odd in these papers;
 in \cite{ITepitame} $\nu =1$ is assumed 
 and in \cite{WeSemi} $\nu$ has to be coprime to $e=n=2$.
 On the other hand, it turns out that
 the algebraic varieties $Z_{\nu}$, which are related to the reductions $\overline{\scrZ} _{\nu}$,
 are considerably more delicate if $\nu$ is even.
 Relying on simple yet useful ideas,
 we compute the cohomology
 by a method different from that in \cite{ITepitame}, \cite{WeSemi} and \cite{BWMax}.

\vspace{\baselineskip} 
 We describe the outline of the paper.
 
 In Section \ref{section:LTperf}, we review some basic facts on
 the Lubin-Tate perfectoid space, a formal model and the relevant group action,
 following \cite{WeSemi}, \cite{BWMax} and \cite{ITepitame}.
 We also introduce the formal scheme $\Nil ^{\flat}$ and related notation.
 
 In Section \ref{sec:affandred}, a family of affinoids and
 formal models is constructed,
 and the reductions are studied
 along with the induced actions
 of the stabilizers.
 In Subsection \ref{subsec:CMpts}
 we recall some facts on CM points from \cite{BWMax} 
 and fix a convenient choice of a CM point $\xi$ as in \cite{ITepitame}. 
 In Subsection \ref{subsec:defofaf}
 we define a family of affinoids indexed by $\nu >0$
 using a coordinate around the fixed CM point $\xi$.
 Subsection \ref{subsec:lemsonnil}
 contains a series of lemmas concerning approximation of valued points of $\Nil ^{\flat}$.
 They allow us to work with the mixed-characteristic setting 
 in almost the same way as the equal-characteristic setting.
 As explained above, the lemmas are largely based on the ideas communicated by Yoichi Mieda.
 In Subsection \ref{subsec:lemsondet}
 we discuss the approximation of the determinant morphism 
 $\Delta \colon \Nil ^{\flat, n, \text{ad}}_{\overline{\eta}} \to \Nil ^{\flat, \text{ad}}_{\overline{\eta}}$.
 We state two general lemmas (Lemmas \ref{lem:det} and \ref{lem:estimate})
 and apply them to two special cases that we need (Lemmas \ref{lem:estimate1} and \ref{lem:estimate2}).
 In Subsection \ref{subsec:redoffms}
 we define formal models and compute their reductions,
 building on the approximation of $\Delta$
 as the discussion of Steps (a) and (b) above.
 In Subsection \ref{subsec:stabandaction}
 we compute the stabilizer of the affinoids and the induced action on the reductions.
 In Subsection \ref{subsec:alggps}
 we define certain algebraic groups and 
 describe the algebraic varieties $Z_{\nu}$ 
 using the Lang torsors and certain quadratic forms.
 
 In Section \ref{sec:cohofred}
 we compute the cohomology
 of $\overline{\scrZ} _{\nu}$ together with 
 the relevant group actions.
 This is reduced to the corresponding
 computation for $Z_{\nu}$
 and is treated separately 
 for odd and even $\nu$
 as the discussion of Step (c) above.
 Subsections \ref{subsec:quadformsandcoh}, \ref{subsec:repofcycfgps}
 (resp.\ Subsection \ref{subsec:Heisenberg})
 contain key ingredients for the computations
 for the odd (resp.\ even) cases.
 
 In Section \ref{sec:real} 
 we prove Main Theorem.
 To this end,
 we apply the theory of the essentially tame local Langlands
 and Jacquet-Langlands correspondences
 developed in \cite{BHetLLCI}, \cite{BHetLLCII}, \cite{BHetLLCIII},
 \cite{BHetJL},
 as well as the results obtained in the previous sections.
 The review of the theory
 in the special cases that we need
 is given in Subsection \ref{subsec:etc}.
 In Subsection \ref{subsec:real}
 we finally complete the proof of Main Theorem.

 \paragraph{Acknowledgments}
 This paper is an improved version of the author's Ph.D. thesis
 submitted to The University of Tokyo in January 2016.  
 He wishes to express his deepest gratitude to his supervisor Takeshi Tsuji
 for his patient guidance, constant encouragement and uplifting words.
 Next he is grateful to Yoichi Mieda
 for making many comments on the previous version of this paper
 and especially for providing an argument to extend the result to the mixed-characteristic setting.
 Moreover he wants to thank Naoki Imai, Tomoki Mihara and Takahiro Tsushima
 for inspiring discussions, comments and for answering questions.
 He also wants to thank Tetsushi Ito for helpful comments.
 Furthermore he thanks Takeshi Saito for his encouragement and support.
 Finally he thanks the referee for reading the paper very carefully,
 pointing out many misprints and inaccuracies
 and suggesting a number of improvements in expositions.
   
This work was supported by the Program for Leading Graduate 
Schools, MEXT, Japan;
the Research Institute for Mathematical Sciences,
a Joint Usage/Research Center located in Kyoto University;
Iwanami Fujukai Foundation; 
and JSPS KAKENHI Grant Number 19K14503.

 \paragraph{Notation}
 For any non-archimedean valuation field $F$,
 we denote the valuation ring by $\calO _F\subset F$
 and its maximal ideal by $\frakp _F\subset \calO _F$. 
 If $R$ is a topological ring, we denote 
 by $\Nil (R)$
 the set of topologically nilpotent elements.

 For any non-archimedean local field $F$,
 we write $v_F$ for the additive valuation
 normalized so that $v_F(\varpi _F)=1$
 for any uniformizer $\varpi _F\in F$. 
 We write $U_F=U_F^0=\calO _F^{\times}$ and $U_F^i=1+\frakp _F^i$
 for $i\geq 1$.
 We denote by $W_F$ the Weil group of $F$ and
 the Artin reciprocity map $\Art _F \colon F^{\times}\rightarrow W_F^{\text{ab}}$
 is normalized so that the uniformizers are mapped to geometric Frobenius elements.
 A character $\xi$ of $F^{\times}$
 is often identified with a character of $W_F$
 via $\Art _F$. 

 We take a prime number $\ell \neq p$,
 and fix an isomorphism $\overline{\bbQ} _{\ell}\simeq \bbC$ 
 of fields.
 Smooth representations over $\overline{\bbQ} _{\ell}$
 are always identified with those over $\bbC$
 by this isomorphism.
 
 For an adic space $X$, a point $\tau \in X$ and a function $f\in \Gamma (X, \calO _X)$
 we write $\lvert f(\tau)\rvert$ for the valuation $\tau (f)$.

 We often denote a multiset by $[\cdot]$
 to distinguish it from a set $\{ \cdot \}$.

 \section{Preliminaries on Lubin-Tate perfectoid space} \label{section:LTperf}
  \subsection{Lubin-Tate perfectoid space and a formal model} \label{subsection:formalmodel}
  We summarize the relevant materials on the Lubin-Tate spaces,
  the Lubin-Tate perfectoid space and a formal model.
  Our basic references are \cite{WeSemi}, \cite{BWMax} and \cite{ITepitame}.
  In many parts,
  we closely follow their expositions.
  
  Let $K$ be a non-archimedean local field
  and $k$ its residue field. 
  Denote by $\frakp =\frakp _K$ 
  the maximal ideal of $\calO _K$.
  We take a uniformizer $\varpi $ of $K$.
  We write $q$ for the cardinality of $k$
  and $p$ for the characteristic of $k$.
  We fix an algebraic closure $\overline{K}$ of $K$
  and denote by $\overline{k}$ the residue field of $\overline{K}$.
  Let $C$ be the completion of $\overline{K}$.

  Let $n$ be a positive integer.
  Let $H _0$ be a one-dimensional formal $\calO _K$-module  
  over $\overline{k}$ of height $n$,
  which is unique up to isomorphism.
  Let $K^{\text{ur}}$ be the maximal unramified extension of $K$ in $\overline{K}$
  and $\widehat{K}^{\text{ur}}$ its topological closure in $C$.
  We denote by $\calC$ 
  the category of complete Noetherian local $\calO _{\widehat{K}^{\text{ur}}}$-algebras
  with residue field $\overline{k}$.
  Let $R\in \calC$.
  A pair $(H, \iota)$ 
  consisting of a formal $\calO _K$-module $H$ over $R$
  and an isomorphism $\iota \colon H_0\stackrel{\sim}{\rightarrow} H\otimes _R \overline{k}$
  is said to be a deformation of $H_0$ to $R$.
  For a formal $\calO _K$-module $H$ over $R$ and
  an integer $m\geq 0$,
  we mean by ``a Drinfeld level $\frakp ^m$-structure on $H$''
  what is called ``a structure of level $m$ on $H$''
  in \cite[p.\ 572 Definition]{DrEmod}.
  
  We define a functor $\calC \rightarrow \Sets$ 
  by associating
  to $R\in \calC$ 
  the set of isomorphism classes
  of triples $(H, \iota , \phi)$
  in which
  $(H, \iota)$ is 
  a deformation of $H_0$ over $R$
  and $\phi \colon ({\frakp ^{-m}}/{\calO _K})^n\rightarrow H[\frakp ^{m}](R)$ is 
  a Drinfeld level $\frakp ^m$-structure on $H$.
  This functor is representable by
  a regular local ring $R_m$ 
  of dimension $n$
  by \cite[Proposition 4.3]{DrEmod}.
  These rings $R_m$
  naturally form an inductive system
  $\{ R_m\}$.  
  We denote by $R_{\infty}=(\varinjlim R_m)^{\wedge}$
  the completion of the inductive limit
  with respect to 
  the ideal 
  generated by the maximal ideal of $R_0$.
  We put $\calM _{H_0, \infty}=\Spf R_{\infty}$.
    
  Let $(H, \iota)$ be a deformation
  of $H_0$ to $\calO _{\widehat{K}^{\text{ur}}}$
  and $A$ its coordinate ring.
  We set 
  $\widetilde{H}=\Spf (\varinjlim A)^{\wedge}$,
  where the transition maps are ring homomorphisms
  corresponding to the multiplication by 
  $\varpi$ of $H$
  and
  the completion is taken
  with respect to the ideal
  generated by a defining ideal of $A$.
  Then $\widetilde{H}$
  is a $K$-vector space object
  in the category $\Adic _{\calO _{\widehat{K}^{\text{ur}}}}$
  of complete adic $\calO _{\widehat{K}^{\text{ur}}}$-algebras.
  It is shown in \cite[Proposition 2.7]{WeSemi}
  that 
  $\widetilde{H}$, as a $K$-vector space object,
  does not depend on the choice of $(H, \iota)$ 
  and is isomorphic to $\Spf \calO _{\widehat{K}^{\text{ur}}}[[X^{q^{-\infty}}]]$
  as a formal scheme.
  Here $\calO _{\widehat{K}^{\text{ur}}}[[X^{q^{-\infty}}]]$
  is defined to be the $(\varpi , X)$-adic completion of
  $\calO _{\widehat{K}^{\text{ur}}}[X^{q^{-\infty}}]
  =\varinjlim _{X\mapsto X^q}\calO _{\widehat{K}^{\text{ur}}}[X]$. 
  In \cite[\S2.7]{WeSemi}
  a natural morphism $\widetilde{\alpha} _n\colon \calM _{H_0, \infty}\to \widetilde{H} ^n$ is constructed.

  Let $K^{\text{ab}}$ be the maximal abelian extension of $K$ in $\overline{K}$
  and $\widehat{K}^{\text{ab}}$ its topological closure in $C$.
  We denote by 
  $\wedge H_0$ 
  the formal $\calO _K$-module
  of height $1$ over $\overline{k}$
  and by $\wedge H$ 
  a deformation of $\wedge H_0$ over $\calO _{\widehat{K}^{\text{ur}}}$.
  Then from the above discussion specialized to $n=1$
  a formal scheme $\calM _{\wedge H_0, \infty}$
  and a morphism $\widetilde{\alpha} _1\colon \calM _{\wedge H_0, \infty}\to \widetilde{\wedge H}$
  are defined.
  By the Lubin-Tate theory
  $\calM _{\wedge H_0, \infty}$ and $\Spf \calO _{\widehat{K}^{\text{ab}}}$
  are isomorphic.
  
  A general theory of exterior powers of $\varpi$-divisible $\calO _K$-modules is 
  developed in \cite{HedPhD}.
  Based on the facts that the top exterior power of $H$ is $\wedge H$ (\cite[Theorem 2.10]{WeSemi})
  and that the universal alternating and multilinear homomorphism sends Drinfeld level structures of $H$
  to those of $\wedge H$ (\cite[Proposition 2.11]{WeSemi}),
  natural morphisms $\delta \colon \widetilde{H} ^n\to \widetilde{\wedge H}$
  and $\calM _{H_0, \infty}\to \calM _{\wedge H_0, \infty}$
  are constructed 
  in \cite[\S2.6, \S2.7]{WeSemi}.
  The morphism $\delta$ induces an alternating and multilinear homomorphism 
  $\widetilde{H} ^n(R)\to \widetilde{\wedge H} (R)$ for any $R\in \Adic _{\calO _{\widehat{K}^{\text{ur}}}}$
  and is called the determinant morphism.
  
  As a result we have the following commutative diagram
   \begin{equation} \label{eq:Wecart}
   \xymatrix{
   \calM _{H_0, \infty} \ar[r] \ar[d]_{\widetilde{\alpha} _n} &\calM _{\wedge H_0, \infty} \ar[d]^{\widetilde{\alpha} _1}   \\
   \widetilde{H} ^n \ar[r]^{\delta}  &\widetilde{\wedge H}  \\
   }
   \end{equation}

  \begin{Thm} (\cite[Theorem 2.17]{WeSemi}) \label{Thm:Wecart}
   The above diagram is Cartesian.
  \end{Thm}
  
  In the following,
  we describe the morphisms $\widetilde{\alpha} _1$ and $\delta$ more explicitly
  by making various choices,
  and then introduce the Lubin-Tate perfectoid space and its formal model. 
  Along the way we also define $\Nil ^{\flat}$, which plays a crucial role in Section \ref{sec:affandred}.
  Note that the explicit descriptions depend on the choices of $\varpi$ and
  $\{ t_m\} _{m\geq 1}$; we will rearrange them in Subsection \ref{subsec:CMpts}. 
  
  For a formal $\calO _K$-module $\Sigma$
  and $a\in \calO _K$,
  we denote by $+_{\Sigma}$ the addition of $\Sigma$
  and by $[a]_{\Sigma}$ the multiplication by $a$
  of $\Sigma$.
  Let $H$ be the formal $\calO _K$-module over $\calO _{\widehat{K} ^{\text{ur}}}$
  such that $\log _{H}(X)=\sum _{i\geq 0}X^{q^{in}}/{\varpi ^i}$ (cf.\ \cite[\S2.3]{BWMax}).
  It is in fact defined over $\calO _K$, but we usually regard it as an object over $\calO _{\widehat{K} ^{\text{ur}}}$.
  For $a\in \calO _C$ we denote by $\overline{a}$ its image in $\overline{k}$.
  We take $H _0$ 
  to be the reduction of $H$,
  so that
  \[
  [\varpi] _{H_0}(X)=X^{q^n},  \quad [\zeta] _{H_0}(X)=\overline{\zeta}X \text{ for $\zeta \in \mu _{q-1}(K)$}.
  \]
  We set $\calO _D=\End H_0$ and $D=\calO _D\otimes _{\calO _K}K$.
  Then $D$ is a central division algebra over $K$
  of invariant $1/n$.
  Denoting by $[a]$ 
  the action of $a\in \calO _D$,
  we define $\varphi _D \in \calO _D$
  to be the element
  such that 
  $[\varphi _D](X)=X^q$.
  Let $K_n$ be the unramified extension of $K$ of degree $n$ in $\overline{K}$.
  For $\zeta \in \mu _{q^n-1}(K_n)$
  we define $\zeta \in \calO _D$ by
  $[\zeta ](X)=\overline{\zeta}X$.
  This defines a $K$-algebra embedding 
  $K_n\rightarrow D$.
  Then $D$ is generated over $K_n$
  by $\varphi _D$
  and we have $\varphi _D^n=\varpi$, $\varphi _D \zeta=\zeta ^q\varphi _D$
  for $\zeta \in \mu _{q^n-1}(K_n)$.
  
  Let $\wedge H$ be the one-dimensional formal $\calO _K$-module
  over $\calO _{\widehat{K} ^{\text{ur}}}$ such that $\log _{\wedge H}(X)=\sum _{i\geq 0}(-1)^{i(n-1)}X^{q^i}/{\varpi ^i}$.
  We take $\wedge H_0$ to be the reduction of $\wedge H$.
  Let
  $\{ t_m\} _{m\geq 1}$
  be a system of elements of $\calO _{\overline{K}}$
  such that
  \begin{equation} \label{eq:systemt}
  t_m\in \calO _{\overline{K}}, \quad [\varpi ]_{\wedge H}(t_1)=0, \quad t_1\neq 0,
  \quad [\varpi ]_{\wedge H}(t_m)=t_{m-1} \text{ for $m\geq 2$}.
  \end{equation}
  We put $\varpi'=(-1)^{n-1}\varpi$.
  Then we have $\log _{\wedge H}(X)=\sum _{i\geq 0} X^{q^i}/{{\varpi'}^i}$
  and thus $[\varpi']_{\wedge H_0}(X)=X^q$.
  
  For $R\in \Adic _{\calO _{\widehat{K}^{\text{ur}}}}$
  we identify $H(R)$ and $\wedge H(R)$ with $\Nil (R)$,
  the set of topologically nilpotent elements of $R$,
  and in particular $\widetilde{H} (R)$ and $\widetilde{\wedge H}(R)$
  are regarded as appropriate subsets of $\Nil (R)^{\bbZ _{\geq 0}}$:
  \begin{align*}
  &\widetilde{H} (R)=\{ (\tilde{x} _m)_{m\geq 0}\in \Nil (R)^{\bbZ _{\geq 0}}\mid [\varpi]_{H}(\tilde{x} _{m+1})=\tilde{x} _m\}, \\
  &\widetilde{\wedge H} (R)=\{ (\tilde{x} _m)_{m\geq 0}\in \Nil (R)^{\bbZ _{\geq 0}}\mid [\varpi]_{\wedge H}(\tilde{x} _{m+1})=\tilde{x} _m\}.
  \end{align*}
  
  A choice of the above system $\{ t_m\} _{m\geq 1}$ amounts to
  fixing an isomorphism $\calM _{\wedge H_0, \infty}\simeq \Spf \calO _{\widehat{K}^{\text{ab}}}$,
  and one can easily describe the induced morphism 
  $\Spf \calO _{\widehat{K}^{\text{ab}}}\simeq \calM _{\wedge H_0, \infty}\stackrel{\widetilde{\alpha} _1}{\to} \widetilde{\wedge H}$
  by
  \[
  \Spf \calO _{\widehat{K} ^{\text{ab}}}(R) \to \widetilde{\wedge H}(R); \ 
  (f\colon \calO _{\widehat{K}^{\text{ab}}}\to R) \mapsto (f(t_{m+1}))_{m\geq 0}
  \] 
  for $R\in \Adic _{\calO _{\widehat{K}^{\text{ur}}}}$.
  We tacitly identify $\calM _{\wedge H_0, \infty}$ with $\Spf \calO _{\widehat{K}^{\text{ab}}}$.
 
  Let $\Nil ^{\flat}$ be the functor $\Adic _{\calO _{\widehat{K}^{\text{ur}}}} \to (\Sets)$ defined by 
  \[
  \Nil ^{\flat}(R)=\varprojlim _{x\mapsto x^q}\Nil (R),
  \]
  which is clearly representable by $\Spf \calO _{\widehat{K} ^{\text{ur}}}[[T^{q^{-\infty}}]]$.
  For an element $\bx \in \Nil ^{\flat}(R)$ 
  we usually use notation such as $\bx =(x^{q^{-l}})_{l\geq 0}$;
  this should cause no confusion.
  We normalize the isomorphism $\Spf \calO _{\widehat{K} ^{\text{ur}}}[[T^{q^{-\infty}}]]\simeq \Nil ^{\flat}$, so that 
  \[
  \Spf \calO _{\widehat{K} ^{\text{ur}}}[[T^{q^{-\infty}}]](R)\stackrel{\sim}{\to} \Nil ^{\flat}(R); \ (f\colon \calO _{\widehat{K} ^{\text{ur}}}[[T^{q^{-\infty}}]]\to R) \mapsto (f(T^{q^{-l}}))_{l\geq 0}.
  \]
  In \cite[\S2.6]{BWMax} an isomorphism $\lambda \colon \widetilde{H}\stackrel{\sim}{\to} \Nil ^{\flat}$ is given:
  \begin{equation} \label{eq:lambda}
  \lambda \colon (\tilde{x} _m)_{m\geq 0}\mapsto (\lim _{m\to \infty}\tilde{x} _m^{q^{mn-l}})_{l\geq 0}, 
  \quad \lambda ^{-1}\colon (x^{q^{-l}})_{l\geq 0}\mapsto (\lim _{l\to \infty}[\varpi ^{l-m}]_{H}(x^{q^{-ln}}))_{m\geq 0}.
  \end{equation}
  Similarly, recalling that $[\varpi']_{\wedge H_0}(X)= X^{q}$,
  we define an isomorphism $\lambda' \colon \widetilde{\wedge H}\to \Nil ^{\flat}$ as follows:
  \begin{align}
  \lambda' \colon (\tilde{x} _m)_{m\geq 0}
  \mapsto &(\lim _{m\to \infty}([(-1)^{(n-1)m}]_{\wedge H}(\tilde{x} _m))^{q^{m-l}})_{l\geq 0} \nonumber \\
  &=(\lim _{m\to \infty}(-1)^{q(n-1)m}\tilde{x} _m^{q^{m-l}})_{l\geq 0}, \label{eq:lambda'} \\ 
  {\lambda'} ^{-1}\colon (x^{q^{-l}})_{l\geq 0}
  \mapsto &(\lim _{l\to \infty}[(-1)^{(n-1)m}{\varpi'}^{l-m}]_{\wedge H}(x^{q^{-l}}))_{m\geq 0} \nonumber \\
  &=(\lim _{l\to \infty}[{\varpi}^{l-m}]_{\wedge H}((-1)^{(n-1)l}x^{q^{-l}}))_{m\geq 0}. \nonumber
  \end{align}
  We put $\alpha _1=\lambda' \circ \widetilde{\alpha} _1$.
  It follows that if we set 
  \[
  t^{q^{-l}}=\lim _{m\rightarrow \infty}(-1)^{q(n-1)(m-1)} t_m^{q^{m-1-l}}\in \widehat{K} ^{\text{ab}}
  \]
  for all $l\geq 0$,
  then $\alpha _1\colon \Spf \calO _{\widehat{K} ^{\text{ab}}} \to \Nil ^{\flat}$
  is induced by
  \begin{equation} \label{rightverticalarrow}
  \calO _{\widehat{K} ^{\text{ur}}}[[T^{q^{-\infty}}]] \to \calO _{\widehat{K} ^{\text{ab}}}; \ T^{q^{-l}}\mapsto t^{q^{-l}}.
  \end{equation}
  We denote by $v=v_K$
  the normalized valuation on $K$
  and extend it to $C$
  by continuity.
  Then $v(t)=1/(q-1)$.
  
  \begin{Rem}
  Although $\widetilde{H}$ and $\Nil ^{\flat}$
  are isomorphic and we mainly work with the latter
  in Section \ref{sec:affandred},
  we distinguish them for now;
  the former is more directly related to the Lubin-Tate spaces
  whereas the coordinate of the latter is useful in studying affinoids.
  \end{Rem}
  
  \paragraph{Operations on $\widetilde{H}$, $\widetilde{\wedge H}$ and $\Nil ^{\flat}$}
  As mentioned before, $\widetilde{H}$ and $\widetilde{\wedge H}$
  are $K$-vector space objects in $\Adic _{\calO _{\widehat{K}^{\text{ur}}}}$.
  We denote the associated operations with $\widetilde{H}$ and $\widetilde{\wedge H}$
  as subscripts;
  for $R\in \Adic _{\calO _{\widehat{K}^{\text{ur}}}}$, $a\in K$ with $a\varpi ^{m_0}\in \calO _K$ $(m_0\geq 0)$
  and $\tilde{\bx} =(\tilde{x} _m)_{m\geq 0}, \tilde{\by} =(\tilde{y} _m)_{m\geq 0}\in \widetilde{H} (R)$,
  we write
  \[
  [a]_{\widetilde{H}}(\tilde{\bx})=([a\varpi ^{m_0}]_{H}(\tilde{x} _{m+m_0}))_{m\geq 0}, \quad 
  \tilde{\bx} +_{\widetilde{H}} \tilde{\by}=(\tilde{x} _m+_H \tilde{y}_m)_{m\geq 0}  
  \]
  and similarly for $\widetilde{\wedge H}$.
  Moreover, the action of $\calO _K$ on $\widetilde{H}$ extends to
  $\calO _D=\End H_0$ by \cite[Lemma 2.5.3 (3)]{BWMax}, 
  and hence $D=\calO _D\otimes _{\calO _K}K$ acts on $\widetilde{H}$.
  We also denote this action by $[d]_{\widetilde{H}}(\bx)$ for $d\in D$ and $\bx \in \widetilde{H}(R)$.
  
  All these induce the corresponding operations on $\Nil ^{\flat}$ via $\lambda$ and $\lambda'$.
  We denote each of such operations with $\Nil ^{\flat} _{H_0}$ (resp.\ $\Nil ^{\flat}_{\wedge H_0}$) as a subscript
  (although it has nothing to do with any base change of $\Nil ^{\flat}$), 
  if the operation is induced via $\lambda$ (resp.\ $\lambda'$).
  Note that in fact it only depends on the choice of $H_0$ (resp.\ $\wedge H_0$).
  We clearly have the following.
  \begin{align} \label{eq:nilop0}
  &\bx +_{\Nil ^{\flat}_{H_0}} \by =\lambda (\lambda ^{-1}(\bx) +_{\widetilde{H}} \lambda ^{-1}(\by)), \\ 
  &\bx +_{\Nil ^{\flat}_{\wedge H_0}} \by =\lambda' ({\lambda'} ^{-1}(\bx) +_{\widetilde{\wedge H}} {\lambda'} ^{-1}(\by)), \\ \label{eq:nilop}
  &[\zeta _{n}]_{\Nil ^{\flat}_{H_0}}(\bx)=(\zeta _n^{q^{-l}}x^{q^{-l}})_{l\geq 0}, \quad
  [\varphi _D^{l_0}]_{\Nil ^{\flat}_{H_0}}(\bx)=(x^{q^{l_0-l}})_{l\geq 0}, \\ \label{eq:nilop2}
  &[\zeta]_{\Nil ^{\flat}_{\wedge H_0}}(\by)=(\zeta y^{q^{-l}})_{l\geq 0}, \quad
  [{\varpi'} ^{l_0}]_{\Nil ^{\flat}_{\wedge H_0}}(\by)=(y^{q^{l_0-l}})_{l\geq 0},
  \end{align}
  where
  $\bx =(x^{q^{-l}})_{l\geq 0}, \by =(y^{q^{-l}})_{l\geq 0}\in \Nil ^{\flat}(R)$,
  $\zeta _n\in \mu _{q^{n}-1}(K_n)$, $\zeta \in \mu _{q-1}(K)$,
  and $\zeta _n^{q^{-l}}$ are defined as the unique elements
  satisfying $(\zeta _n^{q^{-l}})_{l\geq 0} \in \Nil ^{\flat}(\calO _{\widehat{K} ^{\text{ur}}})$
  and $\zeta _n^{q^{-l}} \in \mu _{q^{n}-1}(K_n)$ for all $l\geq 1$.
  We occasionally write $\bx ^{q^{l_0}}$ for $[\varphi _D^{l_0}]_{\Nil ^{\flat}_{H_0}}(\bx)$
  and $\zeta \bx$ for $[\zeta]_{\Nil ^{\flat}_{\wedge H_0}}(\bx)$.
  Also we write $\bx \by=(x^{q^{-l}}y^{q^{-l}})_{l\geq 0}$.

  \paragraph{Description of $\delta$} 
  We set
  \begin{equation} \label{eq:S}
  S=\left\{
      (m_j)\in \bbZ ^n
      \Biggm| 
      \sum _{j=1}^n m_j=\sum _{j=1}^n (j-1),
      \
      m_j\not\equiv m_{j'} \bmod n
      \text{ (if $j\neq j'$)}
      \right\}.
  \end{equation}
  Let $R\in \Adic _{\calO _{\widehat{K}^{\text{ur}}}}$ and $(\tilde{\bx} _1, \dots , \tilde{\bx} _n)\in \widetilde{H}^n(R)$, 
  and set $\bx _i=\lambda (\tilde{\bx} _i)\in \Nil ^{\flat}$.
  It follows from \cite[the proof of Thm.\ 2.10.3]{BWMax} 
  that $\delta \colon \widetilde{H} ^n(R)\to \widetilde{\wedge H}(R)$ is described by 
  \[
  \delta (\tilde{\bx} _1, \dots, \tilde{\bx} _n)
  =(\widetilde{\wedge H})\sum _{m=(m_i)\in S} [\sgn \sigma _m]_{\widetilde{\wedge H}}({\lambda'}^{-1}(\bx_1^{q^{m_1}}\cdots \bx_n^{q^{m_n}})),
  \]
  where 
  \begin{itemize}
  \item the symbol $(\widetilde{\wedge H}) \sum$ means the summation with respect to $+_{\widetilde{\wedge H}}$, 
  \item and we define
      \begin{equation} \label{eq:sigmam}
    \sigma _m= \begin{pmatrix}
                     \overline{0}     & \overline{1}     & \cdots &\overline{n-1}    \\
                     \overline{m} _1 & \overline{m} _2 & \cdots &\overline{m} _n 
                    \end{pmatrix}  
   \end{equation}
   as the permutation of $\bbZ /{n\bbZ}$
   (here 
   $\overline{j}$ denotes the image of $j\in \bbZ$ in $\bbZ /{n\bbZ}$).
  \end{itemize}
  We also define $\Delta \colon \Nil ^{\flat, n}\to \Nil ^{\flat}$
  as the morphism induced via $\lambda$ and $\lambda'$;
  we set $\Delta =\lambda' \circ \delta \circ \lambda ^{-1}$.
  Here we simply write $\lambda$ for 
  \[
  \lambda ^n\colon \widetilde{H}^n\stackrel{\sim}{\to} \Nil ^{\flat, n}; \ (\tilde{\bx} _1, \dots , \tilde{\bx} _n)\mapsto (\lambda (\tilde{\bx} _1), \dots, \lambda (\tilde{\bx} _n)).
  \]
  As before we put $\alpha _n=\lambda \circ \widetilde{\alpha} _n$.
  Thus we now have the following commutative diagram.
   \[
   \xymatrix{
   \calM _{H_0, \infty} \ar[r] \ar[d]^{\widetilde{\alpha} _n} \ar@(dl, ul)[dd]_{\alpha _n} &\calM _{\wedge H_0, \infty} \ar@{=}[r] \ar[d]^{\widetilde{\alpha} _1} \ar@(dr, ur)[dd]^{\alpha _1} &\Spf \calO _{\widehat{K}^{\text{ab}}}  \\
   \widetilde{H} ^n \ar[r]^{\delta} \ar[d]_{\wr}^{\lambda} &\widetilde{\wedge H} \ar[d]^{\lambda'}_{\wr} & \\
   \Nil ^{\flat, n} \ar[r]^{\Delta} & \Nil ^{\flat}. &
   }
   \]

  \paragraph{Base change and related formal models}
   It is the base change to $C$ of the adic generic fiber 
   of $\calM _{H_0, \infty}$
   that admits the group action studied in the non-abelian Lubin-Tate theory.
   In view of \eqref{eq:Wecart}
   we define a formal model of this base change as follows.
   
   We write $\widetilde{H} _{\calO _C}$ 
   (resp.\ $\widetilde{\wedge H} _{\calO _C}$)
   for the base change of $\widetilde{H}$
   (resp.\ $\widetilde{\wedge H}$) 
   from $\calO _{\widehat{K} ^{\text{ur}}}$
   to $\calO _C$.
   Similarly
   we write $\Nil ^{\flat}_{\calO _C}$
   for the base change of $\Nil ^{\flat}$.
   The isomorphisms $\lambda$ and $\lambda'$
   induce isomophisms
   $\widetilde{H} ^n_{\calO _C} \stackrel{\sim}{\to} \Nil ^{\flat, n}_{\calO _C}$
   and
   $\widetilde{\wedge H} _{\calO _C} \stackrel{\sim}{\to} \Nil ^{\flat}_{\calO _C}$,
   which we also denote by $\lambda$ and $\lambda'$.
   We identify
   \[
   \Nil ^{\flat, n}_{\calO _C}=\Spf \calO _C[[X_1^{q^{-\infty}}, \dots, X_n^{q^{-\infty}}]], \quad
   \Nil ^{\flat}_{\calO _C}=\Spf \calO _C[[T^{q^{-\infty}}]].
   \]
   
   Let 
   $\Cont (U_K, \calO _{C})$
   denote the $\calO _{C}$-algebra of continuous maps
   from $U_K$ to $\calO _{C}$.
   We use similar notation 
   for other topological rings as well.
   We define 
   \[
   \calM _{\wedge H_0, \infty, \calO _C}=\Spf \Cont (U_K, \calO _C).
   \]
   Note that $\widehat{K} ^{\text{ab}}\widehat{\otimes} _{\widehat{K} ^{\text{ur}}} C\simeq \Cont (U_K, C)$
   by the local class field theory.
   The formal scheme $\calM _{\wedge H_0, \infty, \calO _C}$
   is considered over $\calM _{\wedge H_0, \infty}=\Spf \calO _{\widehat{K} ^{\text{ab}}}$ by
   \[
   \calO _{\widehat{K}^{\text{ab}}} \rightarrow \Cont (U_K, \calO _{C});
   \ a\mapsto \left( \Art _K(x)(a)\right) _{x\in U_K},
   \]
   which induces a canonical morphism 
   \begin{equation} \label{eq:alpha1onring}
   \calO _C[[T^{q^{-\infty}}]]
   \to \calO _{\widehat{K}^{\text{ab}}}\widehat{\otimes} _{\calO _{\widehat{K} ^{\text{ur}}}} \calO _C
   \to \Cont (U_K, \calO _{C}); \
   T^{q^{-l}}\mapsto \left( \Art _K(x)(t^{q^{-l}})\right) _{x\in U_K}.
   \end{equation}
   
   We set 
   \[
   \calM _{H_0, \infty, \calO _C}=\widetilde{H} ^n_{\calO _C}\times _{\widetilde{\wedge H} _{\calO _C}} \calM _{\wedge H_0, \infty, \calO _C} (\simeq \Nil ^{\flat, n}_{\calO _C}\times _{\Nil ^{\flat}_{\calO _C}} \calM _{\wedge H_0, \infty, \calO _C})
   \]
   and define the Lubin-Tate perfectoid space
   $\calM _{H_0, \infty, \overline{\eta}}^{\text{ad}}$
   as the adic generic fiber.
   By $\lambda$ and $\lambda'$ we have $\calM _{H_0, \infty, \calO _C}\simeq \Spf R_{\infty, \calO _C}$,
   where 
   \[
   R_{\infty, \calO _C}=\calO _C[[X_1^{q^{-\infty}}, \dots, X_n^{q^{-\infty}}]]\widehat{\otimes} _{\calO _C[[T^{q^{-\infty}}]]} \Cont (U_K, \calO _C).
   \]
   With this algebra, we have
   \[
   \calM _{H_0, \infty, \overline{\eta}}^{\text{ad}}
   \simeq \{ \tau \in \Spa (R_{\infty, \calO _{C}}, R_{\infty, \calO _{C}})
   \mid \lvert \varpi (\tau)\rvert \neq 0\}.
   \]
   We also define
   \begin{align*}
  \LTpone &=\{ \tau \in \Spa (\Cont (U_K, \calO _{C}), \Cont (U_K, \calO _{C}))
                                                        \mid \lvert \varpi (\tau)\rvert \neq 0\} \\
              &=\Spa (\Cont (U_K, C), \Cont (U_K, \calO _C)).
  \end{align*}
  Setting $B_n=\calO _{C}[[X_1^{q^{-\infty}}, \dots , X_n^{q^{-\infty}}]]$
  and $B_1=\calO _{C}[[T^{q^{-\infty}}]]$,
  we similarly define
  \begin{align*}
   \Nil ^{\flat, n, \text{ad}}_{\overline{\eta}}
  &=\{ \tau \in \Spa (B_n, B_n)
    \mid \lvert \varpi (\tau)\rvert \neq 0\}, \\
   \Nil ^{\flat, \text{ad}}_{\overline{\eta}}
  &=\{ \tau \in \Spa (B_1, B_1)
    \mid \lvert \varpi (\tau)\rvert \neq 0\}. 
  \end{align*}
  
  In summary, we have the following commutative diagrams:
  \[
 \xymatrix{
 \calM _{H_0, \infty, \calO _C} \ar[r] \ar[d]_{\alpha _n} &\calM _{\wedge H_0, \infty, \calO _C} \ar[d]^{\alpha _1} \\
 \Nil ^{\flat, n}_{\calO _C} \ar[r]^{\Delta} &\Nil ^{\flat}_{\calO _C},
 }
 \qquad
 \xymatrix{
 \calM _{H_0, \infty, \overline{\eta}}^{\text{ad}} \ar[r] \ar[d]_{\alpha _n} &\calM _{\wedge H_0, \infty, \overline{\eta}}^{\text{ad}} \ar[d]^{\alpha _1} \\
 \Nil ^{\flat, n, \text{ad}}_{\overline{\eta}} \ar[r]^{\Delta} &\Nil ^{\flat, \text{ad}}_{\overline{\eta}}.
 }
 \]
  
  \subsection{Group actions on formal models}
   Put
   $G=\GL _n(K)\times D^{\times}\times W_K$.
   We define $N_G$ by
   \[
   N_G\colon \GL _n(K)\times D^{\times}\times W_K\rightarrow K^{\times} 
   ; \ (g, d, \sigma)\mapsto (\det g^{-1})(\Nrd d)(\Art _K^{-1}\sigma),
   \]
   where $\Nrd \colon D^{\times}\rightarrow K^{\times}$
   is the reduced norm,
   and set $G^0=\Ker (v\circ N_G)$.
   As studied in the non-abelian Lubin-Tate theory,
   $\LTp$ admits a natural right action of $G^0$.
  This extends to an action on the formal model $\calM _{H_0, \infty, \calO _C}$.
  The latter can in turn be described in terms of group actions 
  on $\widetilde{H} ^n_{\calO _C}$, $\calM _{\wedge H_0, \infty, \calO _C}$ and $\widetilde{\wedge H} _{\calO _C}$,
  which we now explain, following \cite[\S2.11]{BWMax} and \cite[\S1.2]{ITepitame}.
  Let $\Adic _{\calO _C}$ be the category of complete adic $\calO _C$-algebras
  and 
  let $R\in \Adic _{\calO _C}$.
  
  The larger group $G$ acts on $\widetilde{H} ^n_{\calO _C}$.
   Let $g=(g_{i, j})_{1\leq i, j\leq n}\in \GL _n(K)$ and
   $d\in D^{\times}$.
   Then $g$ and $d$ act on $\widetilde{H} ^n_{\calO _C}(R)$ as
   \begin{align*}
   &g\colon (\tilde{\bx} _i)_{1\leq i\leq n}
   \mapsto \left( (\widetilde{H})\sum _{1\leq j\leq n} [g_{j, i}]_{\widetilde{H}}(\tilde{\bx} _j )\right) _{1\leq i\leq n}, \\
   &d\colon (\tilde{\bx} _i)_{1\leq i\leq n}
   \mapsto \left(  [d^{-1}]_{\widetilde{H}}(\tilde{\bx} _i )\right) _{1\leq i\leq n}.
   \end{align*}
   Let $\sigma \in W_K$ and set $n_{\sigma}=v(\Art _K^{-1}(\sigma))$.
   Then $\sigma \in W_K$ acts on $\widetilde{H} ^n_{\calO _C}$ as
   the composite of $1\times \sigma \colon \widetilde{H} ^n_{\calO _C}\times _{\calO _C} \Spf \calO _C\to \widetilde{H} ^n_{\calO _C}\times _{\calO _C, \sigma}\Spf \calO _C$ 
   and 
   $\varphi _D^{-n_{\sigma}} \colon \widetilde{H} ^n_{\calO _C}\times _{\calO _C, \sigma}\Spf \calO _C
   \to \widetilde{H} ^n_{\calO _C}\times _{\calO _C} \Spf \calO _C$.
   
   For each $i$, the coordinate $X_i^{q^{-l}}$ defines $\bX _i=(X_i^{q^{-l}})_{l\geq 0}\in \Nil ^{\flat}(B_n)$.
   One can easily check that 
   via $\lambda$ 
   each of the above induces an automorphism of the coordinate ring $B_n$ of $\Nil ^{\flat, n}_{\calO _C}$ as follows:
   \begin{align}
   \label{eq:GLactiononLT}
   &g^{\ast}\colon \bX _i
   \mapsto (\Nil ^{\flat}_{H_0})\sum _{1\leq j\leq n} [g_{j, i}]_{\Nil ^{\flat}_{H_0}}(\bX _j ) 
   \quad \text{ for $1\leq i\leq n$}, \\
   \label{eq:DactiononLT}
   &d^{\ast}\colon \bX _i
   \mapsto [d^{-1}]_{\Nil ^{\flat}_{H_0}}(\bX _i )
   \quad \text{ for $1\leq i\leq n$}, \\
   \label{eq:WeilactiononLT}
   &\sigma ^{\ast} \colon
   \bX _i\mapsto \bX _i^{q^{-n_{\sigma}}}, \quad x\mapsto \sigma (x), 
   \quad \text{for $1\leq i\leq n$ and $x\in \calO _{C}$}.
   \end{align}
   Here we indicate the image of $X_i^{q^{-l}}$ for $l\geq 0$ all at once
   by writing the image of $\bX _i$.
   We often use similar notation throughout the paper;
   for a ring homomorphism $f\colon R_1\to R_2$ and $\br =(r^{q^{-l}})_{l\geq 0} \in \Nil ^{\flat}(R_1)$,
   we may write the image of $\br$ 
   by the induced map $\Nil ^{\flat}(R_1)\to \Nil ^{\flat}(R_2)$
   to indicate $f(r^{q^{-l}})$ for $l\geq 0$
   concisely.
   
   Similarly, we have a natural action of $K^{\times}\times W_K$ on $\widetilde{\wedge H} _{\calO _C}$,
   which we regard as that of $G$ by means of $N_G|_{\GL _n(K)\times D^{\times}}$ 
   so that $\delta \colon \widetilde{H} ^n_{\calO _C}\to \widetilde{\wedge H} _{\calO _C}$ is $G$-equivariant. 
   We only record the action on $B_1$ induced via $\lambda'$:
   \[
   (g, d, \sigma)^{\ast} \colon 
   \bT \mapsto [N_G((g, d, 1))^{-1}{\varpi'} ^{-n_{\sigma}}]_{\Nil ^{\flat}_{\wedge H_0}}(\bT), \quad x\mapsto \sigma (x), 
   \quad \text{for $x\in \calO _{C}$},
   \]
   where $n_{\sigma}=v(\Art _K^{-1}(\sigma))$ as before.
   
   We inflate the natural group action on $\calM _{\wedge H_0, \infty, \calO _C}$ to $G^0$ in the following way.
   Let $(g, d, \sigma)\in G^0$ and put $u=N_G((g, d, \sigma)) \in U_K$.
   Then $(g, d, \sigma)$ acts on $\calM _{\wedge H_0, \infty, \calO _C}=\Spf \Cont (U_K, \calO _C)$ as
   \begin{equation} \label{eq:actiononCont}
   (g, d, \sigma)^{\ast} \colon \Cont (U_K, \calO _C)\to \Cont (U_K, \calO _C); \ (a_x)_{x\in U_K}\mapsto (\sigma (a_{u^{-1}x}))_{x\in U_K}.
   \end{equation}
   Again $\widetilde{\alpha} _1\colon \calM _{\wedge H_0, \infty, \calO _C}=\Spf \Cont (U_K, \calO _C)\to \widetilde{\wedge H} _{\calO _C}$
   is $G^0$-equivariant.
   
   The action of $G^0$ on $\calM _{H_0, \infty, \calO _C}$ that we shall study is the one induced by the Cartesian diagram
   and the two equivariant morphisms.
   
   \subsection{Notation related to $\Nil ^{\flat}$} \label{subsec:nilnotation}
   Let $R\in \Adic _{\calO _C}$.
   While we are mainly interested in elements of $\Nil ^{\flat}(R)$,
   we introduce useful notation for elements of 
   a larger set $\Nil (R)^{\bbZ _{\geq 0}}$.
   \begin{Def} \label{Def:Nil}
   Let $\bx =(x_l)_{l\geq 0}, \by =(y_l)_{l\geq 0} \in \Nil (R)^{\bbZ _{\geq 0}}$.
   \begin{enumerate}[(1)]
   \item 
   We define
   \begin{align*}
   \bx +\by&=(x_l+y_l)_{l\geq 0}, \\
   -\bx &=(-x_l)_{l\geq 0}, \\
   \bx \by&=(x_ly_l)_{l \geq 0}.
   \end{align*}
   \item \label{item:ineqforNil}
   For $\tau \in \Spa (R, R)$ 
   we write 
   \[
   \lvert \bx (\tau) \rvert \leq \lvert \by (\tau) \rvert
   \]
   to mean 
   \[
   \lvert x_l (\tau) \rvert \leq \lvert y_l (\tau) \rvert \quad \text{ for all $l\geq 0$}.
   \]
   We also define the notation $\lvert \bx (\tau) \rvert < \lvert \by (\tau) \rvert$
   and $\lvert \bx (\tau) \rvert = \lvert \by (\tau) \rvert$ similarly.
   \end{enumerate}
   \end{Def}
   Thus one needs to be careful 
   of the difference between
   $\bx +_{\Nil ^{\flat}_{H_0}}\by$ and
   $\bx +\by$ for $\bx, \by \in \Nil ^{\flat}(R)$
   (in mixed characteristic).
   These expressions will simplify the exposition
   and also clarify the argument later.
   (See Lemma \ref{lem:app2}, for instance, for a typical use of these expressions.)

\section{Affinoids and the reductions of formal models} \label{sec:affandred}
 \subsection{CM points} \label{subsec:CMpts}
 We briefly review
 some facts on CM points,
 following \cite[\S3.1]{BWMax} and \cite[\S1.3, \S2.1]{ITepitame}.
 
 For a deformation $H'$ of $H_0$
 over $\calO _{C}$,
 we set
 \begin{align*}
 &T_{\frakp}H' =\varprojlim H' [\frakp ^m](\calO _{C}), \\ 
 &V_{\frakp}H' =T_{\frakp}H' \otimes _{\calO _K}K,
 \end{align*}
 where each transition map $H' [\frakp ^{m+1}](\calO _{C})\to H' [\frakp ^m](\calO _{C})$
 is the multiplication by $\varpi$.
 By 
 \cite[Definition 2.10.1]{BWMax},
 a point
 $\xi \in \LTp (C)$ \footnote{More precisely, we mean $\xi \in \LTp (C, \calO _C)$ here. We use similar abbreviations below.}
 defines a corresponding triple
 $(H' , \iota , \phi)$,
 where $H'$ is a formal $\calO _K$-module
 over $\calO _{C}$,
 $\iota \colon H_0\stackrel{\sim}{\rightarrow} H' \otimes _{C}\overline{k}$
 is an isomorphism 
 and $\phi \colon \calO _K^n\stackrel{\sim}{\rightarrow} T_{\frakp}H'$
 is an isomorphism of $\calO _K$-modules. 
 
 \begin{Def}
  Let $L\subset C$ be a separable extension of $K$ of degree $n$
  and let $H'$ be a deformation of $H_0$ to $\calO _{C}$.
  
  We say 
  that $H'$ has CM by $L$
  if there exists a $K$-isomorphism
  $L\xrightarrow{\sim} (\End H')\otimes _{\calO _K}K$
  such that
  the induced homomorphism
  $L\rightarrow \End (\Lie H')\otimes _{\calO _K}K\simeq C$
  agrees with the inclusion
  $L\subset C$.
  We also say 
  that $\xi \in \LTp (C)$
  has CM by $L$
  if the corresponding deformation has
  CM by $L$.
 \end{Def}
 Note that the $K$-isomorphism in the definition
 is determined uniquely
 by the compatibility with the induced homomorphism,
 if it exists.
  
 A point $\xi \in \LTp (C)$
 with CM by $L$
 defines $K$-embeddings
 $i_{\xi}\colon L\rightarrow M_n(K); \ x\mapsto i_{\xi}(x)$
 and
 $i_{\xi}^D\colon L\rightarrow D; \ x\mapsto i_{\xi}^D(x)$
 by the commutativity of the following diagrams
 \[
 \xymatrix{
 K^n \ar[r]^(0.45){\phi \otimes \text{id}} \ar[d]_{i_{\xi}(x)} &V_{\frakp}H' \ar[d]^{V_{\frakp} (x)}  \\
 K^n \ar[r]_(0.45){\phi \otimes \text{id}} &V_{\frakp}H',
 }
 \qquad 
 \xymatrix{
 H_0 \ar[r]^(0.35){\iota} \ar[d]_{i_{\xi}^D(x)} &H'\otimes _{\calO _C} \overline{k} \ar[d]^{x\otimes \text{id}} \\
 H_0 \ar[r]_(0.35){\iota} &H'\otimes _{\calO _C} \overline{k},
 }
 \]
 where the second diagram is considered in the category with isogenies inverted.
 We set $\Delta _{\xi}=(i_{\xi}, i_{\xi}^D)\colon L\rightarrow M_n(K)\times D$.
 After we fix our choice of CM point using Proposition \ref{Prop:explicitCM},
 we will usually suppress $\Delta _{\xi}$ from the notation
 and simply consider $L$ as embedded in $M_n(K)\times D$.

 The following are consequences of the Lubin-Tate theory
 as proved in \cite[Lemmas 3.1.2, 3.1.3]{BWMax} (see also \cite[Lemmas 1.9, 3.4]{ITepitame}).
 \begin{Prop} \label{Prop:S}
 Let $L$ and $\xi \in \calM _{H_0, \infty, \overline{\eta}}^{\emph{ad}} (C)$ be as above.
 \begin{enumerate}[(1)]
 \item \label{item:stabofCMinalggps}
 The group $G^0\cap (\GL _n(K)\times D^{\times})$ acts transitively
  on the set of points on $\calM _{H_0, \infty, \overline{\eta}}^{\emph{ad}}(C)$ with CM by $L$.
  The stabilizer of $\xi$ in this group is $\Delta _{\xi}(L^{\times})$.
 \item \label{item:defofL'}
 Define $L'\subset \overline{K}$ as the finite separable extension for which
 $
 W_{L'}=\{ \sigma \in W_K \mid \sigma (L)=L\}.
 $
 Then, for an element $\sigma \in W_K$,
 the translation 
  $\xi ^{(1, \varphi _D^{-n_{\sigma}}, \sigma )}$ has CM by $L$
 if and only if $\sigma \in W_{L'}$.
 \item \label{item:jinsimplecase}
 If $\sigma \in W_L$, then $\xi ^{(1, (\Art _L^{-1}\sigma )^{-1}, \sigma)} =\xi$.
 \end{enumerate}
 \end{Prop}
 
 By \ref{item:stabofCMinalggps} and \ref{item:defofL'}, for any $\sigma \in W_{L'}$,
 there exists an element $(g, d)\in \GL _n(K)\times D^{\times}$,
 uniquely up to multiplication by $\Delta_{\xi}(L^{\times})$,
 such that $(g, d, \sigma)\in G^0$ and
 $\xi ^{(g, d, \sigma)}=\xi$.
 We define a map $j_{\xi} \colon W_{L'}\rightarrow \Delta_{\xi}(L^{\times}) \backslash (\GL _n(K)\times D^{\times})$
 by $j_{\xi}(\sigma )=\Delta_{\xi}(L^{\times})(g, d)$.
 Then the stabilizer $\calS$ of $\xi$ in $G^0$ is 
 \begin{equation} \label{eq:stabofxi}
 \calS =\{ (g, d, \sigma )\in \GL _n(K)\times D^{\times} \times W_{L'}\mid j_{\xi}(\sigma )=\Delta_{\xi}(L^{\times})(g, d) \}.
 \end{equation}
 The assertion \ref{item:jinsimplecase} says
 that $j_{\xi}(\sigma )=\Delta_{\xi}(L^{\times})(1, (\Art _L^{-1}\sigma )^{-1})$ 
 if $\sigma \in W_L$.
 
 Let $\calN _1$ and $\calN _2$ be the normalizers of $i_{\xi}(L^{\times})$ and
 $i_{\xi}^D(L^{\times})$ in $\GL _n(K)$ and $D^{\times}$ respectively.
 Then by Skolem-Noether's theorem
 both the groups are extensions of $\Gal (L/L')$ by $L^{\times}$.
 In particular, we have canonical morphisms $\calN _1\to \Gal (L/L')$
 and $\calN _2\to \Gal (L/L')$,
 with which we form a fibered product $\calN =\calN _1\times _{\Gal (L/L')}\calN _2$.
 Let $W_{L/L'}=W_{L'}/{\overline{[W_L, W_L]}}$ be the relative Weil group.
 
 \begin{Prop} (\cite[Proposition 3.1.4]{BWMax}) \label{Prop:j}
 We have the following.
 \begin{enumerate}[(1)]
 \item \label{imofj}
 Let $\sigma \in W_{L'}$.
 If $j_{\xi}(\sigma)=\Delta _{\xi}(L^{\times})(g, d, 1)$,
 then 
 \[
 \sigma (x)=gxg^{-1}=dxd^{-1}
 \]
 for $x\in L^{\times}$.
 In particular, $j_{\xi}(\sigma)\in \Delta _{\xi}(L^{\times}) \backslash \calN$.
 \item \label{jasisom}
 The map $j_{\xi}\colon W_{L'}\rightarrow \Delta_{\xi}(L^{\times}) \backslash (\GL _n(K)\times D^{\times})$
 induces an isomorphism of groups
 $j_{\xi}\colon W_{L/L'}\stackrel{\sim}{\rightarrow} \Delta_{\xi}(L^{\times}) \backslash \calN$.
 \end{enumerate}
 \end{Prop}

 Now let $n\geq 2$
 and assume $p\nmid n$.
 We put $n_q=\gcd (n, q-1)$.
 
 For any uniformizer $\varpi \in K$,
 we set $L_{\varpi}=K[X]/(X^n-\varpi)$.
 \begin{lem} (\cite[Lemma 2.1]{ITepitame}) \label{lem:Lpi}
  Let $T(K, n)$ be the set of isomorphism classes of totally ramified
  extensions of $K$ of degree $n$.
  Then the following map is a bijection$:$
  \[
  \mu _{(q-1)/{n_q}}(K)\backslash (\frakp - \frakp^2)/{\frakp ^2}
  \rightarrow T(K, n); \ \varpi \mapsto L_{\varpi}.
  \] 
 \end{lem}
 
 Let $L/K$ be a totally ramified extension of degree $n$ in $\overline{K}$.
 From this point on,
 we work with this fixed field $L$.
 Although we do not indicate in the notation
 the constructions to follow 
 depend on the choice of $L$.
 
 By Lemma \ref{lem:Lpi},
 there exists a uniformizer $\varphi _L\in L$
 such that $\varpi =\varphi _L^n \in K$.
 We apply the arguments of Section \ref{section:LTperf}
 with respect to this uniformizer $\varpi \in K$.
 In particular, 
 $H$, $H_0$, $\varphi _D\in D$
 and $\wedge H$
 are defined.
 We set
 \[
 \varphi =\begin{pmatrix}
             \bm{0} & I_{n-1} \\
             \varpi & \bm{0}
             \end{pmatrix}
             \in M_n(K).
 \]
 Note that
 $\sigma \in W_K$
 lies in $W_{L'}$
 if and only if
 $\varphi _L^{-1}\sigma (\varphi _L)\in \mu _{n_q}(K)$.
 
  For a point $\xi \in \LTp (C)$,
  we write $(\bxi _1, \dots , \bxi _n)\in \Nil ^{\flat, n, \text{ad}}_{\overline{\eta}}(C)=\Nil ^{\flat, n}(\calO _C)$
  for the image.
  We also write $\bxi _i=(\xi _i^{q^{-l}})_{l\geq 0}$ for $1\leq i\leq n$.
  \begin{Prop} \label{Prop:explicitCM}
  Let $L\subset C$, $\varphi _L\in L$, $\varpi \in K$ be as above.
  There exists a point $\xi \in \calM _{H_0, \infty, \overline{\eta}}^{\emph{ad}} (C)$ 
  with CM by $L$ 
  satisfying the following conditions$:$
  \begin{enumerate}[(1)]
   \item $\bxi _i=\bxi _{i+1}^q$ for $1\leq i\leq n-1$.
   \item $v(\xi _i)=1/(nq^{i-1}(q-1))$ for $1\leq i\leq n$.
   \item $i_{\xi}(\varphi _L)=\varphi$, $i_{\xi}^D(\varphi _L)=\varphi _D$.
   \item \label{item:Weilactiononxi}
   For any $\sigma \in W_{K}$,
   we have
   $\overline{\xi _1^{-1} \sigma (\xi _1)}^{q-1} = \overline{\varphi _L^{-1}\sigma (\varphi _L)}$
   in $\overline{k}$.
  \end{enumerate}
  \end{Prop}
  \begin{proof}
  This is essentially \cite[Lemma 2.2]{ITepitame},
  where a construction of $\xi \in \LTp (C)$
  is given.
  The assertion \ref{item:Weilactiononxi} is not stated there 
  but can be deduced as follows.
  
  The underlying CM formal $\calO _K$-module of $\xi$ is
  the lift $\calG ^L$ of $H_0$ considered as a formal $\calO _L$-module of height one
  by the embedding $L=K(\varphi _L)\to D$,
  and the Drinfeld level structure on $\calG ^L$ as a formal $\calO _K$-module is
  given based on a Drinfeld level structure on $\calG ^L$ as a formal $\calO _L$-module.
  Thus let $t_{L, 1}\in C$ be a non-zero root of $[\varphi _L]_{\calG ^L}(X)$ satisfying
  $t_{L, 1}\equiv \xi _1 \pmod{\varphi _L\calO _C}$ as in \cite[Lemma 2.2]{ITepitame}.
  By the congruence we see that $\overline{t_{L, 1}^{-1}\sigma (t_{L, 1})}=\overline{\xi _1^{-1}\sigma (\xi _1)}$.
  On the other hand, 
  since $\calG ^L$ is defined over $\calO _{\widehat{L} ^{\text{ur}}}$
  and we have 
  \[
  [\varphi _L]_{\calG ^L}(X)\equiv X^q \pmod{\varphi _L\calO _{\widehat{L} ^{\text{ur}}}[[X]]}, \quad
  [\varphi _L]_{\calG ^L}(X)\equiv \varphi_LX \pmod{\text{deg 2}},
  \]
  we compare the equation $[\varphi _L]_{\calG ^L}(t_{L, 1})=0$ with its translate by $\sigma$ to see
  $\overline{\varphi _L^{-1}\sigma (\varphi _L)} =\overline{t_{L, 1}^{-1}\sigma (t_{L, 1})} ^{q-1}$
  as desired.
  \end{proof}
  The continuous $\calO _C$-algebra homomorphism
  \[
  \calO _C[[T^{q^{-\infty}}]]\to \calO _C; \ \bT \mapsto \Delta (\bxi _1, \dots, \bxi _n)
  \]
  factors through \eqref{eq:alpha1onring},
  which is to say, $\Delta (\bxi _1, \dots, \bxi _n)=\Art _K(x)(\bt)$
  for some $x\in U_K$.
  We replace the choice of \eqref{eq:systemt}
  by the translates by $\Art _K(x)$ (cf.\ \cite[\S2.1]{ITepitame}),
  so that 
  \begin{equation} \label{eq:tasdet}
  \Delta (\bxi _1, \dots, \bxi _n)=\bt.
  \end{equation}
  
  \subsection{Construction of affinoids} \label{subsec:defofaf}
  We put
  \begin{align*}
  \bY _i &=\bX _i-_{\Nil ^{\flat}_{H_0}}\bxi _i \in \Nil ^{\flat}(B_n) \quad \text{ for all $1\leq i\leq n$}, \\
  \bZ &=(\Nil ^{\flat}_{H_0})\sum _{1\leq i\leq n} [\varphi _D^{i-1}]_{\Nil ^{\flat}_{H_0}}(\bY _i) \\
        &=(\Nil ^{\flat}_{H_0})\sum _{1\leq i\leq n} \bY _i^{q^{i-1}} \in \Nil ^{\flat}(B_n).
  \end{align*}
  By writing $\bY _i=(Y_i^{q^{-l}})_{l\geq 0}$, $\bZ =(Z^{q^{-l}})_{l\geq 0}$
  we define elements $Y_i^{q^{-l}}, Z^{q^{-l}}\in B_n$.
  
  From now on, we fix a valuation $\lvert \cdot \rvert$ of rank one on $C$ inducing the natural topology.
  Accordingly, for a point $\tau$ of an adic space over $C$ and $c\in C$ we simply write $\lvert c\rvert$ for $\lvert c(\tau)\rvert$.

  For each integer $\nu >0$, 
  we define an affinoid
  $\calX _{\nu} \subset \Nil ^{\flat, n, \text{ad}}_{\overline{\eta}}$ 
  by
  \begin{equation*}
   \lvert \bZ (\tau) \rvert \leq \lvert \bxi _1 \rvert ^{q^{\nu}}, 
   \ \ \lvert \bY _i(\tau) \rvert \leq \begin{cases}
                              \lvert \bxi _i \rvert ^{q^{\mu} (q+1)/2} &\text{ if $\nu =2\mu +1$ is odd} \\
                               \lvert \bxi _i \rvert ^{q^{\mu}} &\text{ if $\nu =2\mu$ is even}
                              \end{cases} 
  \end{equation*}
  (in the notation of Definition \ref{Def:Nil} \ref{item:ineqforNil})
  for all $1\leq i\leq n$,
  and an affinoid
  $\calZ _{\nu} \subset \LTp$ by
  the pull-back of $\calX _{\nu}$ in $\LTp$.
  Note that, for instance, $\calX _{2\mu}$ can equally be defined as
  the rational subset 
  \[
  R\left( \frac{ Z}{\xi _1^{q^{\nu}}}, \frac{  Y_1 }{\xi _1^{q^{\mu}}}, \dots, \frac{  Y_n }{\xi _n^{q^{\mu}}}\right) \subset \Spa (B_n, B_n)
  \]
  and is indeed an affinoid.
  
  Take a square root $\xi _n^{q^{-l}/2}$ of $\xi _n^{q^{-l}}$ for each $l\geq 0$
  so that $(\xi _n^{q^{-l}/2})_{l\geq 0} \in \Nil ^{\flat}(\calO _C)$ 
  and put $\xi _i^{q^{-l}/2}=(\xi _n^{q^{-l}/2})^{q^{n-i}}$.
  We regard 
  \[
  \scrX _{\nu}=\Spf \calO _{C} \langle {z'}^{q^{-\infty}}, {y'_2}^{q^{-\infty}}, \dots , {y'_n}^{q^{-\infty}}\rangle
  \]
  as a formal model of $\calX _{\nu}$ by
  \begin{align}
   &(\calO _{C} \langle {z'}^{q^{-\infty}}, {y'_2}^{q^{-\infty}}, \dots , {y'_n}^{q^{-\infty}}\rangle, \calO _{C} \langle {z'}^{q^{-\infty}}, {y'_2}^{q^{-\infty}}, \dots , {y'_n}^{q^{-\infty}}\rangle) \nonumber \\
   &\qquad \qquad \qquad \qquad \qquad \qquad \qquad \to (\Gamma (\calX _{\nu}, \calO _{\calX _{\nu}}), \Gamma (\calX _{\nu}, \calO _{\calX _{\nu}}^{+})) \nonumber \\
   &\bz' \mapsto \bxi _1^{-q^{\nu }}\bZ, \quad \by' _i \mapsto \begin{cases} \label{eq:capitaltoprime}
                                          \bxi _i^{-q^{\mu } (q+1)/2}\bY _i &\text{ if $\nu =2\mu +1$ is odd} \\
                                          \bxi _i^{-q^{\mu }}\bY _i &\text{ if $\nu =2\mu$ is even},
                                          \end{cases}
  \end{align}
  which induces an isomorphism of $\calX _{\nu}$ and the adic generic fiber of $\scrX _{\nu}$. 
  Here we put $\bz'=({z'}^{q^{-l}})_{l\geq 0}$, $\by' _i=({y'_i}^{q^{-l}})_{l\geq 0}$.
  
  To construct a formal model of $\calZ _{\nu}$
  and to study its special fiber,
  we prove several lemmas
  on the approximation 
  of elements of $\Nil ^{\flat}$ (Subsection \ref{subsec:lemsonnil}) and
  of the determinant morphism  
  (Subsection \ref{subsec:lemsondet}).  
  
 \subsection{Approximation of valued points of $\Nil ^{\flat}$} \label{subsec:lemsonnil} 
 Let $R$ be a complete adic $\calO _C$-algebra.
 Let $\Spa (R, R)_{\overline{\eta}}=\{ \tau \in \Spa (R, R)\mid \lvert \varpi (\tau)\rvert \neq 0\}$.
 In the following we use the notation introduced in Subsection \ref{subsec:nilnotation}.
 
 Recall the isomorphisms $\lambda, \lambda'$ defined in \eqref{eq:lambda}, \eqref{eq:lambda'}
 and the formula \eqref{eq:nilop0}--\eqref{eq:nilop2}.
 
 \begin{lem} \label{lem:app1}
 Let $\tau \in \Spa (R, R)_{\overline{\eta}}$.
 Let $\bc=(c^{q^{-l}})_{l\geq 0}\in \Nil ^{\flat}(\calO _C)$ such that $c\neq 0$.
 Let $\widetilde{\bx}=(\widetilde{x} _m)_{m\geq 0}\in \widetilde{H}(R)$ and 
 put $\bx =(x^{q^{-l}})_{l\geq 0}=\lambda (\widetilde{\bx})\in \Nil ^{\flat}(R)$.
 Then the following are equivalent.
 \begin{enumerate}[(i)]
 \item \label{item:tilde} For sufficiently large $m$, we have $\lvert \widetilde{x}_m (\tau)\rvert \leq  \lvert c^{q^{-nm}}\rvert$ (resp.\ $\lvert \widetilde{x}_m (\tau)\rvert < \lvert c^{q^{-nm}}\rvert$).
 \item We have $\lvert \bx (\tau)\rvert \leq \lvert \bc \rvert$ (resp.\ $\lvert \bx (\tau)\rvert < \lvert \bc \rvert$).
 \end{enumerate}
 The same equivalence holds for $\widetilde{\wedge H}$ and $\lambda'$ (with $nm$ replaced by $m$ in \ref{item:tilde}).
 \end{lem}
 \begin{proof}
 First we assume (i) and take $l\geq 0$. 
 Then for sufficiently large $m$ we have $\lvert \widetilde{x}_m^{q^{nm-l}}(\tau) \rvert \leq \lvert c^{q^{-l}}\rvert$.
 By $\lvert c\rvert \neq 0$ and by continuity of the valuation,
 the ideal $\{ y\in R\mid \lvert y(\tau)\rvert <\lvert c^{q^{-l}}\rvert \}$ is open in $R$.
 Thus, $\{ y\in R\mid \lvert y(\tau)\rvert \leq \lvert c^{q^{-l}}\rvert \}$ is also open and hence closed. 
 This shows that $x^{q^{-l}}= \lim _{m\to \infty}\widetilde{x}_m^{q^{nm-l}}$ belongs to the latter ideal, 
 which is to say, $\lvert x^{q^{-l}}(\tau)\rvert \leq \lvert c^{q^{-l}}\rvert$.
 Conversely, let us assume (ii).
 Since $[\varpi ^{l-m}]_H(X)\equiv X^{q^{n(l-m)}} \pmod{\varpi, X^{q^{n(l-m)+1}}}$ for $l\geq m$,
 we have $\lvert [\varpi ^{l-m}]_H(x^{q^{-nl}})(\tau)\rvert \leq \max \{ \lvert x^{q^{-nm}}(\tau)\rvert, \lvert \varpi \rvert \}$.
 Now for sufficiently large $m$ for which $\lvert \varpi \rvert \leq \lvert c^{q^{-nm}} \rvert$, we have
 $\lvert \widetilde{x} _m(\tau)\rvert=\lvert \lim _{l\to \infty}[\varpi ^{l-m}]_H(x^{q^{-nl}})(\tau)\rvert \leq \max \{ \lvert x^{q^{-nm}}(\tau)\rvert, \lvert \varpi \rvert \} \leq \lvert c^{q^{-nm}} \rvert$ by continuity of the valuation.
 A similar argument works if we replace $\leq \lvert c^{q^{-\ast}}\rvert$ with $<\lvert c^{q^{-\ast}}\rvert$.
 
 The assertion for $\wedge H$ and $\lambda'$ is proved in essentially the same way.
 \end{proof}
 
 \begin{lem} \label{lem:app2}
 Let $\tau \in \Spa (R, R)_{\overline{\eta}}$.
 Let $\bc=(c^{q^{-l}})_{l\geq 0}\in \Nil ^{\flat}(\calO _C)$ such that $c\neq 0$
 and let $\bx=(x^{q^{-l}})_{l\geq 0}, \by=(y^{q^{-l}})_{l\geq 0} \in \Nil ^{\flat}(R)$.
 Then the following hold.
 \begin{enumerate}[(1)]
 \item We put $\bz =(z^{q^{-l}})_{l\geq 0}=\bx+_{\Nil ^{\flat}_{H_0}} \by \in \Nil ^{\flat}(R)$.
 Assume that $\lvert \bx (\tau)\rvert \leq \lvert \bc \rvert$, $\lvert \by (\tau)\rvert \leq \lvert \bc \rvert$.
 Then we have $\lvert \bz (\tau)\rvert \leq \lvert \bc \rvert$
 and $\lvert (\bz - (\bx +\by))(\tau)\rvert < \lvert \bc \rvert$.
 \item We put $\bz =(z^{q^{-l}})_{l\geq 0}=[-1]_{\Nil ^{\flat}_{H_0}}(\bx) \in \Nil ^{\flat}(R)$.
 Assume that $\lvert \bx (\tau) \rvert \leq \lvert \bc \rvert$.
 Then we have $\lvert \bz (\tau) \rvert \leq \lvert \bc \rvert$
 and $\lvert (\bz - (-\bx))(\tau)\rvert <\lvert \bc \rvert$.
 \end{enumerate}
 The same properties hold for $\widetilde{\wedge H}$ and $\lambda'$.
 \end{lem}
 \begin{proof}
 We prove (1).
 If we put 
 $(\widetilde{x} _m)_{m\geq 0}=\lambda ^{-1}(\bx), (\widetilde{y} _m)_{m\geq 0}=\lambda ^{-1}(\by)\in \widetilde{H}(R)$,
 then by Lemma \ref{lem:app1} we have $\lvert \widetilde{x}_m(\tau)\rvert \leq \lvert c^{q^{-nm}}\rvert$, $\lvert \widetilde{y}_m(\tau)\rvert \leq \lvert c^{q^{-nm}}\rvert$ for sufficiently large $m$.
 We have $\lambda ^{-1}(\bz)=(\widetilde{z} _m)_{m\geq 0}=(\widetilde{x} _m+_H\widetilde{y} _m)_{m\geq 0}$.
 Since $X+_{H}Y\equiv X+Y \pmod{\text{deg }2}$, we have
 $\lvert (\widetilde{z} _m- (\widetilde{x} _m+\widetilde{y} _m))(\tau)\rvert <\lvert c^{q^{-nm}}\rvert$
 for sufficiently large $m$.
 This shows
 \begin{align*}
 &\lvert (\widetilde{z} _m^{q^{nm-l}}- (\widetilde{x} _m+\widetilde{y} _m)^{q^{nm-l}})(\tau)\rvert \\
 &=\lvert (\widetilde{z} _m- (\widetilde{x} _m+\widetilde{y} _m))(\tau)\rvert
 \cdot \lvert (\widetilde{z} _m^{q^{nm-l}-1}+\dots + (\widetilde{x} _m+\widetilde{y} _m)^{q^{nm-l}-1})(\tau)\rvert 
 <\lvert c^{q^{-l}}\rvert.
 \end{align*}
 On the other hand,
 since $\lvert p\rvert <1$, 
 we have 
\[
\lvert ((\widetilde{x} _m+\widetilde{y} _m)^{q^{nm-l}}-(\widetilde{x} _m^{q^{nm-l}}+\widetilde{y} _m^{q^{nm-l}}))(\tau) \rvert <\lvert c^{q^{-l}}\rvert
\] 
and therefore 
 \begin{align*}
 &\lvert (\widetilde{z} _m^{q^{nm-l}} -(\widetilde{x} _m^{q^{nm-l}}+\widetilde{y} _m^{q^{nm-l}}))(\tau) \rvert \\
 &\leq\max \{ \lvert (\widetilde{z} _m^{q^{nm-l}}- (\widetilde{x} _m+\widetilde{y} _m)^{q^{nm-l}})(\tau) \rvert, \\
 &\qquad \qquad \lvert ((\widetilde{x} _m+\widetilde{y} _m)^{q^{nm-l}}-(\widetilde{x} _m^{q^{nm-l}}+\widetilde{y} _m^{q^{nm-l}}))(\tau) \rvert \} \\
 &<\lvert c^{q^{-l}}\rvert.
 \end{align*} 
 for every $l\geq 0$.
 Taking the limit $m\to \infty$ we have 
 $\lvert (z^{q^{-l}}- (x^{q^{-l}}+y^{q^{-l}}))(\tau)\rvert <\lvert c^{q^{-l}}\rvert$ by continuity of the valuation.
 
 We can prove (2) similarly.
 \end{proof}
 
 \begin{Rem}
 Lemmas \ref{lem:app1} and \ref{lem:app2}, crucial in the following argument, are based on lemmas communicated to the author by Yoichi Mieda.
 The author had essentially been aware of Lemma \ref{lem:app1}, but not of Lemma \ref{lem:app2}.
 The latter would be easier to prove if $H$ could be taken so that $X+_{H_0} Y=X+Y$, which is not possible if $K$ is of mixed-characteristic.
 \end{Rem}
 
 We put $v_D=v\circ \Nrd$. 
 We write $k_n$ for the residue field of $\calO _{K_n}$,
 which is identified with that of $\calO _D$ via the embedding $K_n\to D$.
 For $d\in \calO _D$ we denote by $\overline{d} \in k_n$ the image in $k_n$.
 \begin{lem} \label{lem:appD}
 Let $\tau \in \Spa (R, R)_{\overline{\eta}}$.
 Let $\bc=(c^{q^{-l}})_{l\geq 0}\in \Nil ^{\flat}(\calO _C)$ such that $c\neq 0$
 and let $\bx=(x^{q^{-l}})_{l\geq 0}\in \Nil ^{\flat}(R)$. 
 Assume that $\lvert \bx (\tau) \rvert \leq \lvert \bc \rvert$.
 Then the following hold.
 \begin{enumerate}[(1)]
 \item \label{htnnilaction}
 Let $d\in D^{\times}$.
 We define $\zeta _d\in \mu _{q^n-1}(\calO _{K_n})$ as the unique element 
 such that $\overline{\zeta} _d=\overline{d\varphi _D^{-v_D(d)}}$,
 and $\bzeta _d=(\zeta _d^{q^{-l}})_{l\geq 0}\in \Nil ^{\flat}(\calO _{\widehat{K} ^{\text{ur}}})$ as the unique element such that $\zeta _d^{q^{-l}}\in \mu _{q^{n-1}}(\calO _{K_n})$ for every $l\geq 0$.
 We put $\bz =(z^{q^{-l}})_{l\geq 0}=[d]_{\Nil ^{\flat}_{H_0}}(\bx)\in \Nil ^{\flat}(R)$.
 Then we have $\lvert \bz (\tau) \rvert \leq \lvert \bc ^{q^{v_D(d)}}\rvert$
 and
 $\lvert (\bz - \bzeta _d\bx ^{q^{v_D(d)}})(\tau)\rvert < \lvert \bc ^{q^{v_D(d)}}\rvert$.
 
 \item \label{ht1nilaction}
 Let $a\in K^{\times}$.
 We define $\zeta _a\in \mu _{q-1}(\calO _{K})$ as the unique element 
 such that $\overline{\zeta} _a=\overline{a{\varpi'}^{-v(a)}}$.
 We put $\bz =(z^{q^{-l}})_{l\geq 0}=[a]_{\Nil ^{\flat}_{\wedge H_0}}(\bx)\in \Nil ^{\flat}(R)$.
 Then we have 
 $\lvert \bz (\tau) \rvert \leq \lvert \bc ^{q^{v(a)}}\rvert$
 and $\lvert (\bz - \zeta _a\bx ^{q^{v(a)}})(\tau) \rvert < \lvert \bc ^{q^{v(a)}}\rvert$.
 \end{enumerate}
 \end{lem}
 \begin{proof}
 We prove \ref{htnnilaction}.
 By \eqref{eq:nilop} we have
 \begin{equation*}
 [\zeta] _{\Nil ^{\flat}_{H_0}}(\bx)=(\zeta ^{q^{-l}}x^{q^{-l}})_{l\geq 0}, \quad [\varphi _D^i]_{\Nil ^{\flat}_{H_0}}(\bx)=(x^{q^{i-l}})_{l\geq 0}
 \end{equation*}
 for any $\zeta \in \mu _{q^n-1}(\calO _{K_n})\subset \calO _D$ and $i\in \bbZ$.
 Observing that if $d=\lim _{i\to \infty}d_i$ then $[d]_{\Nil ^{\flat}_{H_0}}(\bx)=\lim _{i\to \infty}[d_i]_{\Nil ^{\flat}_{H_0}}(\bx)$,
 we deduce the assertion from the above equalities and Lemma \ref{lem:app2}.
 
 We can prove \ref{ht1nilaction} similarly, using \eqref{eq:nilop2} and Lemma \ref{lem:app2} for $\wedge H$ and $\lambda'$.
 \end{proof}
 
 So far we have discussed approximations of valued points of $\Nil ^{\flat}$ with respect to a valuation.
 The following lemma allows us to 
 pass from such approximations to congruences later.
 \begin{lem} \label{lem:esttoapp}
 Let $R=\calO _C\langle T_1^{q^{-\infty}}, \dots, T_n^{q^{-\infty}}\rangle$,
 the $\varpi$-adic completion of $\calO _C[T_1^{q^{-\infty}}, \dots, T_n^{q^{-\infty}}]$
 with $T_1, \dots, T_n$ indeterminates.
 Let $f\in R[1/\varpi]$ and $c\in C$ with $c\neq 0$.
 If $\lvert f(\tau)\rvert <\lvert c\rvert$ for all $\tau \in \Spa (R, R)_{\overline{\eta}}=\Spa (R[1/\varpi], R)$,
 then $f\in c\frakp _CR$.
 \end{lem}
 \begin{proof}
 This is merely a variant of an analogous well-known fact for $\calO _C\langle T_1, \dots, T_n\rangle$.
 
 As $c$ is invertible in $R[1/\varpi]$ we may assume $c=1$.
 In particular, we have $\lvert f(\tau)\rvert \leq 1$ for all $\tau \in \Spa (R[1/\varpi], R)$
 and this implies $f\in R$ by \cite[Lemma 3.3 (i)]{HuCont} (see also \cite[Proposition 2.12 (iii)]{SchPerf}).
 Assume, for a contradiction, that $f\not\in \frakp _CR$.
 Then the image $\overline{f}$ in $R/\frakp _CR=\overline{k}[T_1^{q^{-\infty}}, \dots, T_n^{q^{-\infty}}]$
 would be non-zero, and thus there would exist $\overline{\ba} _1, \dots, \overline{\ba} _n\in \varprojlim _{x\mapsto x^q}\overline{k}$
 such that $\overline{f} (\overline{\ba} _1, \dots, \overline{\ba} _n)\neq 0$.
 Now lifts $\ba _1, \dots, \ba _n\in \varprojlim _{x\mapsto x^q}\calO _C$
 would satisfy $f(\ba _1, \dots, \ba _n)\in \calO _C^{\times}$
 and give rise to a point $\tau \in \Spa (R[1/\varpi], R)$
 such that $\lvert f(\tau)\rvert =1$, which is a contradiction.
 \end{proof}
  
 \subsection{Approximation of the determinant morphism} \label{subsec:lemsondet}
  \begin{lem} \label{lem:cancel} 
 Let $\bx _1, \dots \bx _n \in \Nil ^{\flat}(R)$. 
   Then we have 
   \[
   \Delta (\bx _1, \dots , \bx _n)=\Delta (\bx _2^q, \dots , \bx _n^q, \bx _1^{q^{-n+1}}).
   \]
   In particular, for $\bx, \bw \in \Nil ^{\flat}(R)$ we have
   \begin{align*}
   \Delta (\bw ^{q^{n-1}}, \bx ^{q^{n-2}}, \dots , \bx)&=\Delta (\bx ^{q^{n-1}}, \dots , \bx ^q, \bw)=\Delta (\bx ^{q^{n-1}}, \dots , \bx ^{q^2}, \bw ^q, \bx) \\
   &=\dots =\Delta (\bx ^{q^{n-1}}, \dots , \bx ^{q^{i+1}}, \bw ^{q^i}, \bx ^{q^{i-1}}, \dots , \bx).
   \end{align*}
  \end{lem}
  \begin{Rem}
  In the previous version of the present paper,
  this lemma was stated in an unnecessarily restricted situation 
  and proved by a painstaking computation.
  Put this way the lemma is proved in a simpler and more natural way.
  The simplified statement and proof given here are kindly informed to the author by Yoichi Mieda.
  
  Another (perhaps even simpler) proof is to note $\Nrd \varphi _D^{-1}=(-1)^{n-1}\varpi ^{-1}$ 
  and apply \cite[Lemma 2.14]{WeSemi} to deduce
  \begin{align*}
  \Delta \left([\varphi _D^{-1}]_{\Nil ^{\flat}_{H_0}}(\bx _1), \dots , [\varphi _D^{-1}]_{\Nil ^{\flat}_{H_0}}(\bx _n)\right) &=[(-1)^{n-1}\varpi ^{-1}]_{\Nil ^{\flat}_{\wedge H_0}}\left( \Delta (\bx _1, \dots , \bx _n)\right) \\
  &=\Delta \left( \bx _2, \dots , \bx _n, [\varpi ^{-1}]_{\Nil ^{\flat}_{H_0}}(\bx _1)\right), 
  \end{align*}
  where in the last equality we use the multilinear and alternating property of $\Delta$.
  \end{Rem}
  \begin{proof}
  Let $S$ and $\sigma _m$ be as in \eqref{eq:S} and \eqref{eq:sigmam}, respectively.
  We define 
  \[
  f\colon S\to S; \ (m_1, \dots , m_{n-1}, m_n)\mapsto (m_2-1, \dots , m_n-1, m_1+n-1),
  \]
  which is clearly a bijection.
  Then we have 
  \begin{align*}
  &\delta \circ \lambda ^{-1}(\bx _2^q, \dots , \bx _n^q, \bx _1^{q^{-n+1}}) \\
  &=\widetilde{\wedge H}\sum _{m=(m_j)\in S}[\sgn \sigma _m]_{\widetilde{\wedge H}} \lambda'^{-1}(\bx _2^{q^{m_1+1}} \cdots  \bx _n^{q^{m_{n-1}+1}} \bx _1^{q^{m_n-n+1}}) \\
  &=\widetilde{\wedge H}\sum _{m=(m_j)\in S}[\sgn \sigma _{f(m)}]_{\widetilde{\wedge H}} \lambda'^{-1}(\bx _1^{q^{m_1}} \cdots  \bx _n^{q^{m_{n}}}).
  \end{align*}
  Observing that 
  \begin{align*}
  \sigma _{f(m)}
  &=
  \begin{pmatrix}
  \overline{0} & \cdots & \overline{n-2} & \overline{n-1} \\
  \overline{m_2-1} & \cdots & \overline{m_n-1} & \overline{m_1+n-1}
  \end{pmatrix} \\
  &=
  \begin{pmatrix}
  0 & 1 & \cdots  & n-1 \\
  n-1 & 0 & \cdots & n-2
  \end{pmatrix}
  \sigma _m
  \begin{pmatrix}
  0 & 1 & \cdots  & n-1 \\
  n-1 & 0 & \cdots & n-2
  \end{pmatrix}^{-1},
  \end{align*}
  we see that $\sgn \sigma _m=\sgn \sigma _{f(m)}$ and thus
  $\delta \circ \lambda ^{-1}(\bx _2^q, \dots , \bx _n^q, \bx _1^{q^{-n+1}})=\delta \circ \lambda ^{-1}(\bx _1, \dots , \bx _n)$
  as desired.
  \end{proof}

  \begin{Def}
  Let $\bx, \by \in \Nil ^{\flat}(R)$.
  For $\iota \in \{ \pm 1\}$ we write $\bx \equiv _{\iota} \by$ to mean
  \begin{quote}
  for any $\bc \in \Nil ^{\flat}(\calO _C)$ and $\tau \in \Spa (R, R)_{\overline{\eta}}$,
  \begin{align*}
  &\text{if } \lvert \by (\tau) \rvert \leq \lvert \bc \rvert, 
  \text{then } \lvert \bx (\tau) \rvert \leq \lvert \bc \rvert \text{ and } \lvert (\by- \iota \bx) (\tau) \rvert <\lvert \bc \rvert 
  \end{align*}
  and conversely
  \begin{align*}
  &\text{if } \lvert \bx (\tau) \rvert \leq \lvert \bc \rvert,
  \text{then } \lvert \by (\tau) \rvert \leq \lvert \bc \rvert \text{ and } \lvert (\bx- \iota \by) (\tau)\rvert <\lvert \bc \rvert.
  \end{align*}
  \end{quote}
  \end{Def}
  \begin{Rem}
  Let $\bx , \by , \bz \in \Nil ^{\flat}(R)$ and $\iota, \iota _1, \iota _2\in \{ \pm 1\}$. 
  The relation $\equiv _{\iota}$ clearly satisfies the following.
  \begin{itemize}
  \item If $\bx =\by$, then $\bx \equiv _{1}\by$.
  \item If $\bx \equiv _{\iota} \by$, then $\by \equiv _{\iota} \bx$.
  \item If $\bx \equiv _{\iota _1} \by$ and $\by \equiv _{\iota _2}\bz$, then $\bx \equiv _{\iota _1\iota _2} \bz$.
  \end{itemize}
  \end{Rem}

  \begin{lem} \label{lem:det}
   Let $\bx _1, \dots , \bx _n \in \Nil ^{\flat}(R)$.
   \begin{enumerate}[(1)]
   \item \label{pimult}
   We have 
   \[
   \Delta (\bx _1^{q^n}, \dots , \bx _n)\equiv _{(-1)^{n-1}} \Delta (\bx _1, \dots , \bx _n)^{q}.
   \]
   \item \label{pimult2}
   We have 
   \[
   \Delta (\bx _1, \dots , \bx _n)^{q}= \Delta (\bx _2, \dots , \bx _n, \bx _1^{q^n}).
   \]
   \item \label{perm}
   For $\sigma \in \frakS _n$, we have 
   \[
   \Delta (\bx _{\sigma (1)}, \dots , \bx _{\sigma (n)})\equiv _{\sgn \sigma} \Delta (\bx _1, \dots , \bx _n).
   \]
   \end{enumerate}
  \end{lem}
  \begin{proof}
  By the multilinear and alternating property of $\Delta$, we have
  \begin{align*}
   \Delta (\bx _1^{q^n}, \dots , \bx _n)
   &= [\varpi]_{\Nil ^{\flat}_{\wedge H_0}}\Delta (\bx _1, \dots , \bx _n) \\
   &= [(-1)^{n-1}\varpi']_{\Nil ^{\flat}_{\wedge H_0}}\Delta (\bx _1, \dots , \bx _n), \\
   \Delta (\bx _{\sigma (1)}, \dots , \bx _{\sigma (n)})
   &= [(\sgn \sigma)]_{\Nil ^{\flat}_{\wedge H_0}}\Delta (\bx _1, \dots , \bx _n).
  \end{align*}
  Hence \ref{pimult} and \ref{perm} follow from Lemma \ref{lem:app2} (2) for $\wedge H$ and $\lambda'$,
  and \ref{pimult2} follows from the equality $\sgn (12 \cdots n)=(-1)^{n-1}$.
  \end{proof}
  \begin{Rem}
  If $K$ is of equal-characteristic or $p\neq 2$, then 
  \[
  [-1]_{\Nil ^{\flat}_{\wedge H_0}}\Delta (\bx _1, \dots , \bx _n)
  = -\Delta (\bx _1, \dots , \bx _n)
  \]
  and hence \ref{pimult} and \ref{perm} can be stated with $\ast =\iota \cdot \ast$ in place of $\ast \equiv _{\iota} \ast$.
  Otherwise, $-1$ does not lie in $\mu _{q-1}(K)$ and
  this leads to a slightly careful formulation of \ref{pimult} and \ref{perm} as above.
  \end{Rem}

  \begin{lem} \label{lem:estimate}
   Let $\tau \in \Spa (R, R)_{\overline{\eta}}$.
   Let $r_1\geq r_2 \geq \dots \geq r_n$ be positive rational numbers
   such that $r_1< r_nq^n $.
   Let $\bx \in \Nil ^{\flat}(\calO _{C})$ 
   and $\bx _i \in \Nil ^{\flat}(R)$ 
   $(1\leq i\leq n)$.
   Suppose that $\lvert \bx _i(\tau) \rvert \leq \lvert \bx^{r_i}\rvert$ for all $1\leq i\leq n$.
   Then 
   \begin{gather*}
    \lvert \Delta (\bx_1, \dots , \bx_n) (\tau) \rvert \leq
    \lvert \bx ^{\sum _ir_iq^{i-1}}\rvert, \\
    \left\vert \left( \Delta (\bx_1, \dots , \bx_n) - 
    \sum _{\substack{\sigma \in \frakS _n \\ r_{\sigma (1)}\geq \dots \geq r_{\sigma (n)}}} (\sgn \sigma) \bx _{\sigma (1)}\bx _{\sigma (2)}^{q^{1}}\dots \bx _{\sigma (n)}^{q^{n-1}} \right) (\tau)\right\vert
    <\lvert \bx ^{\sum _ir_iq^{i-1}}\rvert.
   \end{gather*}
  \end{lem}
  
  \begin{proof}
  Let $S_0=\{ (\sigma ^{-1}(1)-1, \dots, \sigma ^{-1}(n)-1)\in S\mid \sigma \in \frakS _n, r_{\sigma (1)}\geq \dots \geq r_{\sigma (n)} \}$.
  In the following we shall show that for $(m_1, \dots , m_n)\in S$
  \begin{equation} \label{estim}
  \sum _ir_iq^{m_i}
  \begin{cases}
  =\sum _ir_iq^{i-1}  \ &\text{if $(m_1, \dots , m_n)\in S_0$} \\
  >\sum _ir_iq^{i-1} \ &\text{otherwise}.
  \end{cases}
  \end{equation}
  This will imply the lemma because then
  \[
  \lv \bx _1^{q^{m_1}}\dotsm \bx _n^{q^{m_n}} (\tau) \rv 
  \begin{cases}
  \leq \lvert \bx ^{\sum _ir_iq^{i-1}}\rvert \ &\text{if $(m_1, \dots , m_n)\in S_0$} \\
  <\lvert \bx ^{\sum _ir_iq^{i-1}}\rvert \ &\text{otherwise}
  \end{cases}
  \]
  and thus by Lemma \ref{lem:app2}
  \begin{align*}
  \lv \bx ^{\sum _ir_iq^{i-1}}\rv
  &> \lv \left( \Delta (\bx _1, \dots , \bx _n) 
  - \sum _{m=(m_i)\in S_0}(\sgn \sigma _m) \bx _{1}^{q^{m_1}}\bx _{2}^{q^{m_2}}\dots \bx _{n}^{q^{m_n}}\right) (\tau)\rv \\
  &= \lv \left( \Delta (\bx _1, \dots , \bx _n) 
  -\sum _{\substack{\sigma \in \frakS _n \\ r_{\sigma (1)}\geq \dots \geq r_{\sigma (n)}}} (\sgn \sigma) \bx _{\sigma (1)}^{q^{0}}\bx _{\sigma (2)}^{q^{1}}\dots \bx _{\sigma (n)}^{q^{n-1}} \right) (\tau) \rv
  \end{align*}
  as desired.  
  
  Now we show \eqref{estim}.
   Put 
   $
   d _i=\log _qr_i
   $
   for all $1\leq i\leq n$,
   so that 
   $d_1-d_n<n$ and
   $\sum _iq^{m_i}r_i=\sum _iq^{m_i+d_i}$.
   To facilitate our argument,
   we introduce a total order structure $\geq$
   on the set of all the multisets of $n$ real numbers
   by deeming 
   $[m_1, \dots , m_n]\geq [m'_1, \dots , m'_n]$
   if and only if,
   when altering the indexing so that
   $m_1\geq \dots \geq m_n$
   and 
   $m'_1\geq \dots \geq m'_n$,
   we have
   $(m_1, \dots , m_n)\geq (m'_1, \dots , m'_n)$
   with respect to the lexicographic order on 
   $\bbR ^n$.
   Then it is easily verified that,
   assuming 
   $\sum _{1\leq i\leq n}m_i=\sum _{1\leq i\leq n}m'_i$,
   we have 
   $\sum _{1\leq i\leq n}q^{m_i}\geq \sum _{1\leq i\leq n}q^{m'_i}$
   if and only if
   $[m_1, \dots , m_n]\geq [m'_1, \dots , m'_n]$.
   Thus, 
   putting
   $f(m_1, \dots , m_n)=[m_1+d_1, \dots , m_n+d_n]$,
   we are to show that
   the set
   \[
   O=\left\{
         f(m_1, \dots , m_n)
         \mid
         (m_1, \dots , m_n)\in S
         \right\}
         \]
   admits the smallest element
   $f(0, \dots , n-1)$
   (with respect to the induced order structure)
   and that
   it is attained only by those
   $(m_1, \dots , m_n)\in S$
   such that 
   \begin{equation} \label{atcond}
   \{ m_1, \dots , m_n\}=\{ 0, \dots , n-1\}
   \text{ and } 
   d_{\sigma (1)}\geq \dots \geq d_{\sigma (n)},
   \end{equation}
   where
   \[ 
   \sigma =\begin{pmatrix}
                 m_1+1 & \dots & m_n+1 \\
                 1       & \dots & n
                \end{pmatrix}
             \in \frakS _n.
   \]

   We show that if 
   $(m_1, \dots , m_n)\in S$
   does not satisfy
   (\ref{atcond})
   then $f(m_1, \dots , m_n)$
   admits a strictly smaller element in $O$.
   
   First assume that 
   $\{ m_1, \dots , m_n\} \neq \{ 0, \dots , n-1\}$.
   Then there exist 
   $1\leq i, j\leq n$
   such that 
   $m_i\geq n$
   and 
   $m_j\leq -1$.
   Now replacing 
   $m_i$ with $m_j+n$ and
   $m_j$ with $m_i-n$ yields 
   a strictly smaller element:\footnote{
                                                 Here the inequality 
                                                 is written
                                                 as if
                                                 $i<j$,
                                                 but this is only for a notational convenience.
                                                 We do not assume $i<j$ 
                                                 and the argument clearly works without this assumption.
                                                 }
   \[
   f(m_1, \dots , m_i, \dots , m_j, \dots , m_n)
   >
   f(m_1,  \dots , m_j+n, \dots , m_i-n, \dots , m_n)
   \]
   because
   $
   m_i+d_i>m_j+d_j, \ 
   m_j+n+d_i, \ 
   m_i-n+d_j.
   $
   
   Next assume that 
   $\{ m_1, \dots , m_n\} =\{ 0, \dots , n-1\}$
   but (\ref{atcond}) does not hold.
   Then there exists
   $1\leq i\leq n-1$
   such that 
   $d_{\sigma (i)} <d_{\sigma (i+1)}$
   with $\sigma$ as before
   (so that $\sigma (i)>\sigma (i+1)$).
   Now interchanging 
   $m_{\sigma (i)}=i-1$ and 
   $m_{\sigma (i+1)}=i$
   yields 
   a strictly smaller element:
   \[
   f(m_1, \dots , \stackrel{\sigma (i+1)}{\stackrel{\smile}{i}}, \dots , 
     \stackrel{\sigma (i)}{\stackrel{\smile}{i-1}}, \dots , m_n)
   >
   f(m_1, \dots , \stackrel{\sigma (i+1)}{\stackrel{\smile}{i-1}}, \dots , 
     \stackrel{\sigma (i)}{\stackrel{\smile}{i}}, \dots , m_n).
   \]
   
   Finally, given an element
   $(m_1, \dots , m_n)\in S$
   such that 
   $f(m_1, \dots , m_n)\neq f(0, \dots , n-1)$,
   we may apply the above procedures finitely many times 
   to obtain strict inequalities
   $f(m_1, \dots , m_n)> f(m'_1, \dots , m'_n)> \cdots$
   until we eventually 
   find 
   some element
   $(\tilde{m}_1, \dots , \tilde{m}_n)\in S$
   such that 
   \[
   f(m_1, \dots , m_n)>\cdots >f(\tilde{m}_1, \dots , \tilde{m}_n)=f(0, \dots , n-1).
   \]
   Therefore, 
   $f(0, \dots , n-1)$
   is indeed the smallest in $O$.
   The same argument shows that 
   $f(m_1, \dots , m_n)=f(0, \dots , n-1)$
   only if 
   $(m_1, \dots , m_n)$
   satisfies (\ref{atcond}).
   Now the proof is complete. 
   \end{proof}
   
   \begin{Cor} \label{Cor:tandxi}
   For $l\geq 0$ we have
   \[
   t^{q^{-l}}\equiv \xi _n^{nq^{n-1-l}} \pmod{t^{q^{-l}}\frakp _C\calO _C}.
   \]
   \end{Cor}
   \begin{proof}
   This immediately follows by applying Lemma \ref{lem:estimate}
   with $\bx _i=\bxi _i$, $\bx =\bxi _n$, $r_i=q^{n-i}$.
   (Note that we do not need to use Lemma \ref{lem:esttoapp}, since we are working with elements in $\calO _C$.)
   \end{proof}
      
  Although the following two lemmas are in principle simple applications of Lemmas \ref{lem:det}, \ref{lem:estimate},
  they involve many cases.
  To state them concisely
  we define,
  for an integer $0\leq \mu \leq n-1$ 
  and a rational number $1\leq c<q$,
  \begin{align*}
   &M_1(\mu)=(n+\mu (q-1))q^{n-1}, \\
   &M_2(\mu, c)=\begin{cases}
                         (n+2(c-1)+2\mu (q-1))q^{n-1}      &\text{if $0\leq \mu <n/2$} \\
                         (n+2(c-1)+(2\mu -n)(q-1))q^n &\text{if $n/2\leq \mu <n$}.
                        \end{cases}
  \end{align*}
  
  \begin{lem} \label{lem:estimate1}
   Let $\tau \in \Spa (R, R)_{\overline{\eta}}$.
   Let $0\leq \mu \leq n-1$ be an integer.
   Let $\bx _n\in \Nil ^{\flat}(\calO _{C})$, 
   $\bT \in \Nil ^{\flat}(R)$. 
   Put
   $\bx _{i}=\bx _n^{q^{n-i}}$ 
   for all $1\leq i\leq n-1$.
   Suppose that
   $\lvert \bT (\tau) \rvert \leq \lvert \bx _1^{q^{\mu }}\rvert$.
   Put
   \[
   \Delta =\Delta (\bT, \bx _2, \dots , \bx_n)
   \]
   and $M_1=M_1(\mu)$.   
   
   Then the following assertions hold.
   \begin{enumerate}[(1)]
    \item \label{1:mu=0}
            Suppose that $\mu =0$.
            Then
            \begin{gather*}
             \lvert \Delta (\tau) \rvert \leq \lvert \bx _n^{M_1}\rvert,
             \\
             \lvert (\Delta - \bx _n^{(n-1)q^{n-1}}\bT) (\tau)  \rvert <\lvert \bx _n^{M_1}\rvert
             . 
            \end{gather*}
    \item \label{1:mu>0}
            Suppose that $\mu >0$.
            Then
            \begin{gather*}
             \lvert \Delta (\tau) \rvert \leq \lvert \bx _n^{M_1}\rvert, \\
             \lvert (\Delta - (-1)^{\mu} \bx _n^{(n-\mu -1+(\mu -1)q)q^{n-1}}
             (\bx _n^{q^{n}}\bT ^{q^{-\mu}}-\bx _n^{q^{n-1}}\bT ^{q^{-\mu +1}}) ) (\tau) \rvert
             <\lvert \bx _n^{M_1}\rvert
             . 
            \end{gather*}

   \end{enumerate}
  \end{lem}
  \begin{proof}  
   The case \ref{1:mu=0} follows immediately from Lemma \ref{lem:estimate}.
   
   Suppose that $\mu >0$.
   In this case 
   we arrange the order of the variables of $\Delta (\bT, \bx _2, \dots , \bx _n)$
   to apply Lemma \ref{lem:estimate}.
   By Lemma \ref{lem:det} \ref{pimult2} and \ref{perm} we have
   \begin{align*}
   \Delta (\bT, \bx _2, \dots , \bx _n)&\equiv _{1}\Delta (\bx _2, \dots , \bx _n, \bT ^{q^{-n}})^{q} \\
   &\equiv _{(-1)^{\mu}}\Delta (\bx _2, \dots , \bx _{n-\mu}, \bT ^{q^{-n}}, \bx _{n-\mu +1}, \dots , \bx _n)^{q}.
   \end{align*}
   Now we apply Lemma \ref{lem:estimate} with $\bx =\bx _n$ and
   \[
   (r_1, \dots , r_n)=(q^{n-2}, q^{n-3}, \dots , q^{\mu}, q^{\mu -1}, q^{\mu -1}, \dots , 1).
   \]
   Then $\sum _ir_iq^{i-1}=(n-\mu)q^{n-2}+\mu q^{n-1}=M_1q^{-1}$.
   Since $\lvert \bT ^{q^{-n}} (\tau) \rvert \leq \lvert \bx _n^{q^{\mu -1}}\rvert$,
   we have
   \begin{align*}
   &\left\lvert \left( \Delta (\bx _2, \dots , \bx _{n-\mu}, \bT ^{q^{-n}}, \bx _{n-\mu +1}, \dots , \bx _n) \right. \right. \\
   &\quad -\bx _2\dotsm \bx _{n-\mu}^{q^{n-\mu-2}}
   ((\bT ^{q^{-n}})^{q^{n-\mu -1}}\bx _{n-\mu +1}^{q^{n-\mu }}-(\bT ^{q^{-n}})^{q^{n-\mu }}\bx _{n-\mu +1}^{q^{n-\mu -1}}) \\
   &\quad \quad \left. \cdot \left. \bx _{n-\mu +2}^{q^{n-\mu +1}}\dotsm \bx_n^{q^{n-1}} \right) (\tau) \right\rvert \\
   &<\lvert \bx _n^{M_1q^{-1}}\rvert.
   \end{align*}
   A simple computation shows that
   \begin{align*}
   \bx _2\dotsm \bx _{n-\mu}^{q^{n-\mu-2}}&=\bx _n^{(n-\mu -1)q^{n-2}}, \\
   (\bT ^{q^{-n}})^{q^{n-\mu -1}}\bx _{n-\mu +1}^{q^{n-\mu }}-(\bT ^{q^{-n}})^{q^{n-\mu }}\bx _{n-\mu +1}^{q^{n-\mu -1}}&=\bx _n^{q^{n-1}}\bT ^{q^{-\mu -1}}-\bx _{n}^{q^{n-2}}\bT ^{q^{-\mu}}, \\
   \bx _{n-\mu +2}^{q^{n-\mu +1}}\dotsm \bx_n^{q^{n-1}}&=\bx _n^{(\mu -1)q^{n-1}}
   \end{align*}
   in $\Nil (R)^{\bbZ _{\geq 0}}$. 
   Hence we conclude that 
   \begin{align*}
   \lvert \bx _n^{M_1}\rvert &>
   \lv \left( \Delta  
   -(-1)^{\mu}(\bx _n^{(n-\mu -1)q^{n-2}})^q
   \cdot ((\bx _n^{q^{n-1}}\bT ^{q^{-\mu -1}})^q-(\bx _{n}^{q^{n-2}}\bT ^{q^{-\mu}})^q) \right. \right. \\
   &\qquad \left. \left. \cdot (\bx _n^{(\mu -1)q^{n-1}})^q \right) (\tau) \rv \\
   &=\lv \left( \Delta 
   -(-1)^{\mu}\bx _n^{(n-\mu -1)q^{n-1}+(\mu -1)q^n}
   (\bx _n^{q^{n}} \bT ^{q^{-\mu}}-\bx _{n}^{q^{n-1}} \bT ^{q^{-\mu +1}}) \right) (\tau) \rv
   \end{align*}
   as desired.   
  \end{proof}
  
  \begin{lem} \label{lem:estimate2}
   Let $\tau \in \Spa (R, R)_{\overline{\eta}}$.
   Let $1\leq i<j\leq n$ be integers, $0\leq \mu \leq n-1$ an integer 
   and $1\leq c<q$ a rational number.
   Let $\bx _n\in \Nil ^{\flat}(\calO _{C})$ 
   and $\bT _i, \bT _j \in \Nil ^{\flat}(R)$.
   Put
   $\bx _{i}=\bx _n^{q^{n-i}}$ 
   for all $1\leq i\leq n$.
   Suppose that
   $\lvert \bT _i (\tau) \rvert \leq \lvert \bx _i^{cq^{\mu }}\rvert$ and $\lvert \bT _j (\tau) \rvert \leq \lvert \bx _j^{cq^{\mu }}\rvert$
   for $1\leq i<j\leq n$.
   Put
   \[
   \Delta =\Delta (\bx _1, \dots , \bT _i, \dots , \bT_j, \dots , \bx_n)
   \]
   and $M_2=M_2(\mu, c)$.
   
   Then the following assertions hold.
   \begin{enumerate}[(1)]
    \item \label{2:vanish}
            $\lvert \Delta (\tau) \rvert \leq \lvert \bx _n^{M_2}\rvert$.
    \item \label{2:c>1, mu:small}
            Suppose that $c>1$ and $\mu <n/2$.
            Then 
            \[
            \lvert \bx _n^{M_2}\rvert > 
            \begin{cases}
            \lvert (\Delta -\bx _n^{(n-2\mu -2+2\mu q)q^{n-1}}\bT _i^{q^{i-\mu-1}}\bT _j^{q^{j-\mu-1}}) (\tau) \rvert & \\
            & \hspace{-25pt} \text{if $\mu <j-i<n-\mu$} \\
            \lvert \Delta (\tau) \rvert & \hspace{-25pt} \text{otherwise.} \\
            \end{cases}
            \]
    \item \label{2:c>1, mu:large}
            Suppose that $c>1$ and $\mu \geq n/2$.
            Then 
            \[
            \lvert \bx _n^{M_2}\rvert > 
            \begin{cases}
            \lvert (\Delta -\bx _n^{(2(n-\mu -1)+(2\mu -n)q)q^n} \bT _i^{q^{i-\mu }}\bT _j^{q^{j-\mu }}) (\tau) \rvert & \\
            & \hspace{-25pt} \text{if $n-\mu \leq j-i\leq \mu$} \\
            \lvert \Delta (\tau) \rvert  & \hspace{-25pt} \text{otherwise.} \\
            \end{cases}
            \]
    \item \label{2:c=1, mu=0}
            Suppose that $c=1$ and $\mu =0$.
            Then 
            \[
            \lvert (\Delta - \bx _n^{(n-2)q^{n-1}} \bT _i^{q^{i-\mu }} \bT _j^{q^{j-\mu }}) (\tau) \rvert <\lvert \bx _n^{M_2}\rvert.
            \]
    \item \label{2:c=1, mu:small}
            Suppose that $c=1$ and $0<\mu <n/2$.
            
            Then 
            $\lvert (\Delta - \bx _n^{(n-2\mu -2+(2\mu -2)q)q^{n-1}}\bd _1) (\tau) \rvert <\lvert \bx _n^{M_2}\rvert$,
            where
            \[
            \bd _1= \begin{cases}
                               (\bx _n^{q^{n}}\bT _i^{q^{i-\mu -1}}-\bx _n^{q^{n-1}}\bT _i^{q^{i-\mu}})
                               (\bx _n^{q^{n}}\bT _j^{q^{j-\mu -1}}-\bx _n^{q^{n-1}}\bT _j^{q^{j-\mu}}) & \\
                               & \hspace{-65pt} \text{if $\mu <j-i<n-\mu$} \\
                               -(\bx _n^{q^{n}}\bT _i^{q^{i-\mu -1}}-\bx _n^{q^{n-1}}\bT _i^{q^{i-\mu}})
                               \bT _j^{q^{j-\mu}}                                                      & \hspace{-65pt} \text{if $\mu =j-i<n-\mu$} \\
                               -\bT _i^{q^{i-\mu}} 
                               (\bx _n^{q^{n}}\bT _j^{q^{j-\mu -1}}-\bx _n^{q^{n-1}}\bT _j^{q^{j-\mu}}) & \hspace{-65pt} \text{if $\mu <j-i=n-\mu$} \\
                               0 \quad \text{(constant system)}                                              & \hspace{-65pt} \text{otherwise.}
                              \end{cases}
            \]
    \item \label{2:c=1, mu=n/2}
            Suppose that $c=1$ and $\mu =n/2$ (hence $n$ is even).
            Then 
            \[
            \lvert \bx _n^{M_2}\rvert > 
            \begin{cases}
            \lvert (\Delta -\bx _n^{(n-2)q^n}
                               \bT _i^{q^{i-\mu }}
                               \bT _j^{q^{j-\mu }}) (\tau) \rvert  & \text{if $j-i=n/2$} \\
            \lvert \Delta (\tau) \rvert & \text{otherwise.}
            \end{cases}
            \]
    \item \label{2:c=1, mu:large}
            Suppose that $c=1$ and $\mu >n/2$.
            
            Then 
            $\lvert (\Delta - \bx _n^{(2(n-\mu -1)+(2\mu -n-2)q)q^n}\bd _2) (\tau) \rvert <\lvert \bx _n^{M_2}\rvert$,
            where
            \[
            \bd _2= \begin{cases}
                               (\bx _n^{q^{n+1}}\bT _i^{q^{i-\mu}}-\bx _n^{q^{n}}\bT _i^{q^{i-\mu +1}})
                               (\bx _n^{q^{n+1}}\bT _j^{q^{j-\mu}}-\bx _n^{q^{n}}\bT _j^{q^{j-\mu +1}}) & \\
                               & \hspace{-65pt} \text{if $n-\mu <j-i<\mu$} \\
                               \bT _i^{q^{i-\mu}}
                               (\bx _n^{q^{n+1}}\bT _j^{q^{j-\mu}}-\bx _n^{q^{n}}\bT _j^{q^{j-\mu +1}}) & \hspace{-65pt} \text{if $n-\mu =j-i<\mu$} \\
                               (\bx _n^{q^{n+1}}\bT _i^{q^{i-\mu}}-\bx _n^{q^{n}}\bT _i^{q^{i-\mu +1}})
                               \bT _j^{q^{j-\mu}}                                                       & \hspace{-65pt} \text{if $n-\mu <j-i=\mu$} \\
                               0 \quad \text{(constant system)}                                              & \hspace{-65pt} \text{otherwise.}
                              \end{cases}
            \]
   \end{enumerate}
  \end{lem}
  \begin{proof}
   To lighten the notation
   we define $\bx _i$ for all $i\in \bbZ$ by $\bx _i=\bx _n^{q^{n-i}}$.
   
   Let us first prove \ref{2:vanish}.
   Again we use Lemma \ref{lem:det} repeatedly so that we can apply Lemma \ref{lem:estimate}.
   
   Suppose that $j-i\leq \mu$.
   By Lemma \ref{lem:det} \ref{pimult2}, \ref{perm} we have
   \begin{align}
   &\Delta (\bx _1, \dots , \bT _i, \dots , \bT_j, \dots , \bx_n) \nonumber \\
   &\equiv _{1}\Delta (\bx _{j+1-n}, \dots , \bx _0, 
                                 \bx _1, \dots , \bx _{i-1}, 
                                 \bT _i, \bx  _{i+1}, \dots , \bx _{j-1}, \bT _j)^{q^{j-n}} \nonumber \\
   &\equiv _{(-1)^{\mu}} \Delta (\bx _{j+1-n}, \dots , \bx _{j-\mu -1}, 
             \bT _j, \bx _{j-\mu}, \dots , \bx _{i-1}, 
             \bT _i, \bx  _{i+1}, \dots , \bx _{j-1})^{q^{j-n}} \nonumber \\
   &\equiv _{1} \Delta (\bx _{i+1-n}, \dots , \bx _{j-1-n}, 
                 \bx _{j+1-n}, \dots , \bx _{j-\mu -1}, 
                 \bT _j, \bx _{j-\mu}, \dots , \bx _{i-1}, \bT _i)^{q^{i-n+1}} \nonumber \\
   &\begin{cases}
   \equiv _{(-1)^{\mu +1}} 
   \Delta (\bx _{i+1-n}, \dots , \bx _{j-1-n}, 
             \bx _{j+1-n}, \dots , \bx _{i-\mu -1}, 
             \bT _i, \bx _{i-\mu}, \dots \cr
           \qquad \qquad \qquad  \dots , \bx _{j-\mu -1}, 
             \bT _j, \bx _{j-\mu}, \dots , \bx _{i-1})^{q^{i-n+1}} \cr
             \qquad \qquad \qquad \qquad \qquad \qquad \text{if $j-i\leq \mu$ and $j-i<n-\mu$} \\
   \equiv _{(-1)^{\mu}}
   \Delta (\bx _{i+1-n}, \dots , \bx _{i-\mu -1},
                                   \bT _i, \bx _{i-\mu}, \dots , \bx _{j-1-n}, 
                                   \bx _{j+1-n}, \dots \cr
                                 \qquad \qquad \qquad  \dots , \bx _{j-\mu -1}, 
                                   \bT _j, \bx _{j-\mu}, \dots , \bx _{i-1})^{q^{i-n+1}} \cr
             \qquad \qquad \qquad \qquad \qquad \qquad \text{if $n-\mu \leq j-i\leq \mu$}. \\
   \end{cases} \label{eq:j-i:small}
   \end{align}

   Thus, if $j-i\leq \mu$ and $j-i<n-\mu$, 
   then Lemma \ref{lem:estimate} applied with $\bx =\bx _n$ and
   \begin{align*}
   &(r_1, \dots , r_n) \\
   &=(q^{2n-i-1}, q^{2n-i-2}, \dots , q^{2n-j+1}, q^{2n-j-1}, 
                         \dots, q^{n-i+\mu +1}, cq^{n-i+\mu}, q^{n-i+\mu}, 
                         \dots \\ 
   &\quad \quad \dots, q^{n-j+\mu +1}, cq^{n-j+\mu}, q^{n-j+\mu}, \dots, q^{n-i+1})
   \end{align*}
   shows that 
   \begin{align*}
   &\lvert \Delta (\bx _{i+1-n}, \dots , \bx _{j-1-n}, 
             \bx _{j+1-n}, \dots , \bx _{i-\mu -1}, 
             \bT _i, \\
             &\quad \bx _{i-\mu},  
             \dots, \bx _{j-\mu -1}, 
             \bT _j, \bx _{j-\mu}, \dots , \bx _{i-1})^{q^{i-n+1}} (\tau) \rvert 
             \leq \lvert \bx _n^M\rvert,
   \end{align*}
   where
   \begin{align*}
   M&=\left( \sum _{1\leq h\leq n}r_hq^{h-1}\right) q^{i-n+1} \nonumber \\
     &=\big( (j-i-1)q^{2n-i-1}+(n-\mu -(j-i)-1)q^{2n-i-2}+cq^{2n-i-2} \nonumber \\
     &\qquad +(j-i)q^{2n-i-1}+cq^{2n-i-1}+(\mu -(j-i))q^{2n-i}\big) q^{i-n+1} \nonumber \\
     &=\big(
         n-\mu -(j-i)+c-1
         +(2(j-i)+c-1)q
         +(\mu -(j-i))q^2
         \big) q^{n-1}. 
   \end{align*}
   Let us show $M\geq M_2$.
   Now
   \begin{align}
    &\frac{1}{q^{n-1}}
     (
     M
     -(n+2(c-1)+2\mu (q-1))q^{n-1}
     ) \nonumber \\
    &=\mu -(j-i)-(c-1)
      +(2(j-i-\mu)+c-1)q
      +(\mu -(j-i))q^2
      \nonumber \\
    &=(\mu -(j-i))(q-1)^2
      +(c-1)(q-1) 
      \nonumber \\
    &\geq 0. \label{estofbound1}
   \end{align}
   Here, the equality holds
   if and only if 
   $\mu =j-i$ and $c=1$
   (in which case 
   $\mu <n/2$).
   Similarly, noting $q>c$ we have
   \begin{align}
    &\frac{1}{q^{n-1}}
     (
     M
     -(n+2(c-1)+(2\mu -n)(q-1))q^n 
     )\nonumber \\
    &=n-\mu -(j-i)+c-1
      +(2(j-i-(n-\mu))-(c-1))q \nonumber \\
    &\qquad +(n-\mu -(j-i))q^2
      \nonumber \\
    &=(n-\mu -(j-i))(q-1)^2
      -(c-1)(q-1) 
      \nonumber \\
    &> (n-\mu -(j-i)-1)(q-1)^2 \nonumber \\
    &\geq 0. \label{estofbound2}
   \end{align}
   We see that $M\geq M_2$.
   
   Applying similarly Lemma \ref{lem:estimate},
   we see that 
   if $n-\mu \leq j-i\leq \mu$, then 
   \begin{align} 
   &\lvert \Delta (\bx _{i+1-n}, \dots , \bx _{i-\mu -1},
                                   \bT _i, \bx _{i-\mu}, \dots , \bx _{j-1-n}, \nonumber \\
                                   &\qquad \bx _{j+1-n}, \dots , \bx _{j-\mu -1}, 
                                   \bT _j, \bx _{j-\mu}, \dots , \bx _{i-1})
                                   ^{q^{i-n+1}} (\tau) \rvert \leq \lvert \bx _n^{M_2}\rvert. \label{est:mu:large}
   \end{align}
       
   Suppose that $j-i>\mu$.
   By Lemma \ref{lem:det} \ref{pimult2}, \ref{perm} we have
   \begin{align}
   &\Delta (\bx _1, \dots , \bT _i, \dots , \bT_j, \dots , \bx_n) \nonumber \\
            &\equiv _{1} \Delta (\bx _{j+1-n}, \dots , \bx _0, 
                               \bx _1, \dots , \bx _{i-1}, 
                               \bT _i, \bx  _{i+1}, \dots , \bx _{j-1}, \bT _j
                              )^{q^{j-n}} \nonumber \\
            &\equiv _{(-1)^{\mu}} 
               \Delta (\bx _{j+1-n}, \dots , \bx _{i-1},
                               \bT _i, \bx  _{i+1}, \dots , \bx _{j-\mu -1}, 
                               \bT _j, \bx _{j-\mu}, \dots , \bx _{j-1} 
                               )^{q^{j-n}} \nonumber \\
            &\equiv _{1} 
               \Delta (
                         \bT _j^{q^n}, \bx _{j-\mu -n}, \dots , \bx _{j-1-n},
                         \bx _{j+1-n}, \dots , \bx _{i-1},
                         \bT _i, \bx  _{i+1}, \dots , \bx _{j-\mu -1} 
                         )^{q^{j-n-\mu -1}} \nonumber \\
            &\begin{cases}
            \equiv _{(-1)^{\mu -1}} 
               \Delta (
                         \bT _j^{q^n}, \bx _{j-\mu -n}, \dots , \bx _{i-\mu -1},
                         \bT _i, \bx _{i-\mu}, \dots \cr
                         \qquad \qquad \qquad 
                         \dots , \bx _{j-1-n},
                         \bx _{j+1-n}, \dots , \bx _{i-1},
                         \bx  _{i+1}, \dots , \bx _{j-\mu -1} 
                         )^{q^{j-\mu -n-1}} \cr
                         \qquad \qquad \qquad \qquad \qquad \qquad 
                         \text{if $j-i>\mu$ and $j-i\geq n-\mu$} \\
            \equiv _{(-1)^{\mu}}
               \Delta (
                        \bT _j^{q^n}, \bx _{j-\mu -n}, \dots , \bx _{j-1-n},
                        \bx _{j+1-n}, \dots \cr
                        \qquad \qquad \qquad 
                        \dots , \bx _{i-\mu -1},
                        \bT _i, \bx _{i-\mu}, \dots , \bx _{i-1},
                        \bx  _{i+1}, \dots , \bx _{j-\mu -1} 
                        )^{q^{j-\mu -n-1}} \cr
                        \qquad \qquad \qquad \qquad \qquad \qquad 
                        \text{if $\mu <j-i<n-\mu$}.
               \end{cases} \label{eq:j-i:large}
   \end{align}

   If $j-i>\mu$ and $j-i\geq n-\mu$,
   then Lemma \ref{lem:estimate} 
   applied with $\bx =\bx _n$ and
   \begin{align*}
   &(r_1, \dots , r_n) \\
   &=(cq^{2n-j+\mu}, q^{2n-j+\mu}, \dots , q^{n-i+\mu +1}, cq^{n-i+\mu}, q^{n-i+\mu}, \dots, q^{2n-j+1}, q^{2n-j-1}, \dots \\ 
   &\quad \quad \dots, q^{n-i+1}, q^{n-i-1}, \dots, q^{n-j+\mu +1})
   \end{align*}
   shows that
   \begin{align*}
   &\lvert \Delta 
              (
              \bT _j^{q^n}, \bx _{j-\mu -n}, \dots , \bx _{i-\mu -1},
              \bT _i, \bx _{i-\mu}, \dots, \bx _{j-1-n}, \\
              &\qquad \bx _{j+1-n}, \dots , \bx _{i-1},
              \bx  _{i+1}, \dots , \bx _{j-\mu -1} 
              ) ^{q^{j-\mu -n-1}} (\tau) \rvert 
             \leq \lvert \bx_n^{M'} \rvert,
   \end{align*}
   where
   \begin{align*}
   M'&=\left( \sum _{1\leq h\leq n}r_hq^{h-1}\right) q^{j-\mu -n-1} \\
     &=\big( cq^{2n-j+\mu}+ (n-(j-i))q^{2n-j+\mu +1} \\
     &\qquad +cq^{2n-j+\mu +1} +(\mu +j-i-n)q^{2n-j+\mu +2} \\
     &\qquad \qquad +(n-(j-i)-1)q^{2n-j+\mu +1}+(j-i-\mu -1)q^{2n-j+\mu} \big) q^{j-\mu -n-1} \\
     &=\big(
              j-i-\mu +c-1
              +(2n-2(j-i)+c-1)q
              +(j-i-(n-\mu))q^2
              \big) q^{n-1}.
   \end{align*}
   We have
   \begin{align}
    &\frac{1}{q^{n-1}}
     (        
     M'
     -(n+2(c-1)+2\mu (q-1))q^{n-1}
     ) \nonumber \\
    &
    =(j-i-(n-\mu))(q-1)^2
      +(c-1)(q-1) 
    \geq 0, \label{estofbound3}
       \\ 
    &\frac{1}{q^{n-1}}
    (
     M'
     -(n+2(c-1)+(2\mu -n)(q-1))q^n
     ) \nonumber \\
    &
    =(j-i-\mu)(q-1)^2
      -(c-1)(q-1) 
    >(j-i-\mu -1)(q-1)^2   
    \geq 0. \label{estofbound4}
   \end{align}
   Thus, $M'\geq M_2$ as desired.
   In the first inequality,
   the equality holds 
   if and only if
   $j-i=n-\mu$ and $c=1$
   (in which case
   $\mu <n/2$).
   
   Similarly, it follows from Lemma \ref{lem:estimate} that
   if $\mu < j-i< n-\mu$, then 
   \begin{align} 
   &\lvert \Delta 
              (
              \bT _j^{q^n}, \bx _{j-\mu -n}, \dots , \bx _{i-\mu -1},
              \bT _i, \bx _{i-\mu}, \dots, \bx _{j-1-n}, \nonumber \\
              &\qquad \bx _{j+1-n}, \dots , \bx _{i-1},
              \bx  _{i+1}, \dots , \bx _{j-\mu -1} 
              ) ^{q^{j-\mu -n-1}} (\tau) \rvert \leq \lvert \bx _n^{M_2}\rvert. \label{est:mu:small}
   \end{align}
   The assertion \ref{2:vanish}
   follows from
   \eqref{estofbound1}, 
   \eqref{estofbound2}, 
   \eqref{est:mu:large}, 
   \eqref{estofbound3}, 
   \eqref{estofbound4}, 
   \eqref{est:mu:small}.

   In proving the rest of the lemma,
   we first assume
   $c=1$.
   If $\mu \geq n/2$
   then we see from the previous argument 
   that $\lvert \Delta (\tau) \rvert < \lvert \bx _n^{M_2}\rvert$
   unless $n-\mu \leq j-i\leq \mu$.
   Assuming $\mu \geq n/2$ 
   and $n-\mu \leq j-i\leq \mu$,
   we have,
   by the second case of 
   \eqref{eq:j-i:small},
   \begin{align*}
    &\Delta (\bx _1, \dots, \bT _i, \dots, \bT _j, \dots, \bx _n) \equiv _{1} \\
    &\Delta (\bx _{i+1-n}, \dots , \bx _{i-\mu -1},
                                   \bT _i, \bx _{i-\mu}, \dots , \bx _{j-1-n}, \\
                                   &\qquad \bx _{j+1-n}, 
                                   \dots , \bx _{j-\mu -1}, 
                                   \bT _j, \bx _{j-\mu}, \dots , \bx _{i-1}) ^{q^{i-n+1}}
   \end{align*}
   and by Lemma \ref{lem:estimate},
    \begin{align*}
    &\lvert (\Delta (\bx _{i+1-n}, \dots , \bx _{i-\mu -1},
                                   \bT _i, \bx _{i-\mu}, \dots , \bx _{j-1-n}, \\
                                   &\qquad \bx _{j+1-n}, 
                                   \dots , \bx _{j-\mu -1}, 
                                   \bT _j, \bx _{j-\mu}, \dots , \bx _{i-1}) ^{q^{i-n+1}} 
                                   - \bd )(\tau) \rvert <\lvert \bx _n^{M_2}\rvert,
    \end{align*}
   where
   \[
   \bd =\begin{cases}
               \bx _{i+1-n}^{q^{i+1-n}} \dotsm \bx _{i-\mu -1}^{q^{i-\mu -1}}(
                                                                                       \bT _i^{q^{i-\mu}} \bx _{i-\mu}^{q^{i-\mu +1}}
                                                                                        -\bT _i^{q^{i-\mu +1}} \bx _{i-\mu}^{q^{i-\mu}}
                                                                                       )
                                                                                       \bx _{i-\mu +1}^{q^{i-\mu +2}} \dotsm \bx _{j-1-n}^{q^{j-n}} 
                                                                                       \cr
                                                                                       \qquad 
                                                                                       \cdot
                                                                                       \bx _{j+1-n}^{q^{j+1-n}} \dotsm \bx _{j-\mu -1}^{q^{j-\mu -1}} 
                                                                                       (
                                                                                       \bT _j^{q^{j-\mu}} \bx _{j-\mu}^{q^{j-\mu +1}}
                                                                                       -\bT _j^{q^{j-\mu+1}} \bx _{j-\mu}^{q^{j-\mu}}
                                                                                       )
                                                                                       \bx _{j-\mu +1}^{q^{j-\mu +2}} \dotsm \bx _{i-1}^{q^i}
                                                                                       \cr
                                                                                       \qquad \qquad
                                                                                       \text{if $n-\mu <j-i<\mu$}
                                                                                       \\
               \bx _{i+1-n}^{q^{i+1-n}} \dotsm \bx _{i-\mu -1}^{q^{i-\mu -1}}
                                                                                       \bT _i^{q^{i-\mu}} 
                                                                                       \cr
                                                                                       \qquad 
                                                                                       \cdot
                                                                                       \bx _{j+1-n}^{q^{j+1-n}} \dotsm \bx _{j-\mu -1}^{q^{j-\mu -1}} 
                                                                                       (
                                                                                       \bT _j^{q^{j-\mu}} \bx _{j-\mu}^{q^{j-\mu +1}}
                                                                                       -\bT _j^{q^{j-\mu+1}} \bx _{j-\mu}^{q^{j-\mu}}
                                                                                       )
                                                                                       \bx _{j-\mu +1}^{q^{j-\mu +2}} \dotsm \bx _{i-1}^{q^i}
                                                                                       \cr
                                                                                       \qquad \qquad
                                                                                       \text{if $n-\mu =j-i<\mu$}
                                                                                       \\
               \bx _{i+1-n}^{q^{i+1-n}} \dotsm \bx _{i-\mu -1}^{q^{i-\mu -1}}(
                                                                                       \bT _i^{q^{i-\mu}} \bx _{i-\mu}^{q^{i-\mu +1}}
                                                                                       -\bT _i^{q^{i-\mu +1}} \bx _{i-\mu}^{q^{i-\mu}}
                                                                                       )
                                                                                       \bx _{i-\mu +1}^{q^{i-\mu +2}} \dotsm \bx _{j-1-n}^{q^{j-n}} 
                                                                                       \cr
                                                                                       \qquad 
                                                                                       \cdot
                                                                                       \bx _{j+1-n}^{q^{j+1-n}} \dotsm \bx _{j-\mu -1}^{q^{j-\mu -1}} 
                                                                                       \bT _j^{q^{j-\mu}} 
                                                                                       \cr
                                                                                       \qquad \qquad
                                                                                       \text{if $n-\mu <j-i=\mu$}
                                                                                       \\
               \bx _{i+1-n}^{q^{i+1-n}} \dotsm \bx _{i-\mu -1}^{q^{i-\mu -1}}
                                                                                       \bT _i^{q^{i-\mu}} 
                                                                                       \cr
                                                                                       \qquad 
                                                                                       \cdot
                                                                                       \bx _{j+1-n}^{q^{j+1-n}} \dotsm \bx _{j-\mu -1}^{q^{j-\mu -1}} 
                                                                                       \bT _j^{q^{j-\mu}} 
                                                                                       \cr
                                                                                       \qquad \qquad
                                                                                       \text{if $n-\mu =j-i=\mu$}.
              \end{cases}
   \]
   From this, 
   (\ref{estofbound2}) and (\ref{estofbound4}), 
   the assertions
   \ref{2:c=1, mu=n/2}
   and
   \ref{2:c=1, mu:large}
   follow.
  
   Except for minor complications
   the assertions \ref{2:c=1, mu=0} and \ref{2:c=1, mu:small}
   are proved similarly; 
   we apply Lemma \ref{lem:estimate}
   to the second case of (\ref{eq:j-i:large})
   (resp.\ the first case of (\ref{eq:j-i:small}) with $j-i=\mu$
   and the first case of (\ref{eq:j-i:large}) with $j-i=n-\mu$)
   to obtain the desired estimates
   for those $(i, j)$ 
   such that $\mu <j-i<n-\mu$
   (resp.\ $\mu =j-i<n-\mu$ and $\mu <j-i=n-\mu$).
   Note that in the second case of (\ref{eq:j-i:large})
   we have to treat separately the cases 
   where $\mu =0$ and $\mu >0$.
   
   The assertions 
   \ref{2:c>1, mu:small} and \ref{2:c>1, mu:large}
   are similarly and more easily proved by
   applying Lemma \ref{lem:estimate}
   to the second case of (\ref{eq:j-i:large})
   and (\ref{eq:j-i:small})
   respectively.   
  \end{proof}

  \subsection{Reductions of formal models} \label{subsec:redoffms}
  Let $\nu >0$ be an integer
  and set $M_3(\nu)=(1-s/n)q^r+(s/n)q^{r+1}$,
  where
  $\nu =rn+s$ with $r, s\in \bbZ$ and $0\leq s\leq n-1$.
  We put 
  \[
  \bU=(U^{q^{-l}})_{l\geq 0}=\bT -_{\Nil ^{\flat}_{\wedge  H_0}}\bt \in \Nil ^{\flat}(B_1)
  \]
  and define an affinoid $\calY _{\nu} \subset \Nil ^{\flat, \text{ad}}_{\overline{\eta}}$ by
  \[
  \lvert \bU (\tau) \rvert \leq \lvert \bt \rvert ^{M_3(\nu)}(=\lvert \bxi _n \rvert ^{nq^{n-1}M_3(\nu)}).
  \]
  
  \begin{Prop} \label{Prop:premstudyonaf}
   \begin{enumerate}[(1)]
    \item \label{item:imofaf}
    The set-theoretic image
    of $\calX _{\nu}\subset \Nil ^{\flat, n, \emph{ad}}_{\overline{\eta}}$
    in $\Nil ^{\flat, \emph{ad}}_{\overline{\eta}}$
    under
    $\Delta \colon \Nil ^{\flat, n, \emph{ad}}_{\overline{\eta}} \rightarrow \Nil ^{\flat, \emph{ad}}_{\overline{\eta}}$
    is contained in an affinoid
    $\calY _{\nu} \subset \Nil ^{\flat, \emph{ad}}_{\overline{\eta}}$.
    \item \label{item:imcapunit}
    The pull-back of $\calY _{\nu}$ in
    $\calM _{\wedge H_0, \infty, \overline{\eta}}^{\emph{ad}}
    \simeq U_K$,
    which we simply denote by
    $\calY _{\nu}\cap \calM _{\wedge H_0, \infty, \overline{\eta}}^{\emph{ad}}$,
    is identified with $U_K^{\lceil \nu /n\rceil}$.
   \end{enumerate}
  \end{Prop}

  \begin{proof}
  Let us prove \ref{item:imofaf}.
  Let $\tau \in \calX _{\nu}$.
  We are to show
  $\lvert (\Delta (\bX _1, \dots, \bX _n)-_{\Nil ^{\flat}_{\wedge H_0}} \bt)(\tau) \rvert \leq \lvert \bt ^{M_3(\nu)} \rvert$.
  In $\Nil ^{\flat}(B_n)$,
  we expand as follows:
  \begin{align}  
  &\Delta (\bX_1, \dots , \bX_n)-_{\Nil ^{\flat}_{\wedge H_0}}\bt \nonumber \\ 
  &=\Delta (\bxi _1-_{\Nil ^{\flat}_{H_0}}\bY _2^q-_{\Nil ^{\flat}_{H_0}} \dots -_{\Nil ^{\flat}_{H_0}}\bY _n^{q^{n-1}}+_{\Nil ^{\flat}_{ H_0}}\bZ, \nonumber \\ 
  &\qquad \qquad \bxi _2+_{\Nil ^{\flat}_{H_0}}\bY _2, \dots , \bxi _n+_{\Nil ^{\flat}_{H_0}}\bY _n)-_{\Nil ^{\flat}_{\wedge H_0}}\bt \nonumber \\ 
  &=\Delta (\bxi _1, \dots , \bxi _n)-_{\Nil ^{\flat}_{\wedge H_0}} \bt \nonumber \\ 
  &\quad +_{\Nil ^{\flat}_{\wedge H_0}} \left( {\Nil ^{\flat}_{\wedge H_0}}\right) \sum _{2\leq i\leq n} \left( \Delta ([-1]_{\Nil ^{\flat}_{H_0}} (\bY _i^{q^{i-1}}), \dots , \bxi _n) \right. \nonumber \\
  &\left. \quad \quad +_{\Nil ^{\flat}_{\wedge H_0}} \Delta (\bxi _1, \dots , \bY _i, \dots , \bxi _n) \right) \nonumber \\ 
  &\quad \quad \quad +\Delta (\bZ, \bxi _2, \dots , \bxi _n) \nonumber \\  
  &\quad \quad \quad \quad +_{\Nil ^{\flat}_{\wedge H_0}} \left( {\Nil ^{\flat}_{\wedge H_0}}\right) \sum _{2\leq i, j\leq n} \Delta ([-1]_{\Nil ^{\flat}_{H_0}} (\bY _i^{q^{i-1}}), \bxi _2, \dots , \bY _j, \dots , \bxi _n) \nonumber \\ 
  &\quad \quad \quad \quad \quad +_{\Nil ^{\flat}_{\wedge H_0}} \left( {\Nil ^{\flat}_{\wedge H_0}}\right) \sum _{2\leq i< j\leq n} \Delta (\bxi _1, \dots , \bY _i, \dots , \bY _j, \dots , \bxi _n) 
                                    +_{\Nil ^{\flat}_{\wedge H_0}} \cdots \nonumber \\ 
                                   &=\Delta (\bZ, \bxi _2, \dots , \bxi _n) \label{mainterm} \\ 
                                   &\quad +_{\Nil ^{\flat}_{\wedge H_0}} \left( {\Nil ^{\flat}_{\wedge H_0}}\right) \sum _{2\leq i, j\leq n} \Delta ([-1]_{\Nil ^{\flat}_{H_0}} (\bY _i^{q^{i-1}}), \bxi _2, \dots , \bY _j, \dots , \bxi _n) \nonumber \\ 
                                    &\quad \quad +_{\Nil ^{\flat}_{\wedge H_0}} \left( {\Nil ^{\flat}_{\wedge H_0}}\right) \sum _{2\leq i< j\leq n} \Delta (\bxi _1, \dots , \bY _i, \dots , \bY _j, \dots , \bxi _n) 
                                    +_{\Nil ^{\flat}_{\wedge H_0}} \cdots, \nonumber
  \end{align}
  where we use \eqref{eq:tasdet} and Lemma \ref{lem:cancel} in the last equality.
  Note that 
  we may estimate each term separately by Lemma \ref{lem:app2} and
  the terms not indicated here are negligible because $\lvert \bZ (\tau) \rvert <\lvert \bxi _1 \rvert$
  and $\lvert \bY _i (\tau) \rvert <\lvert \bxi _i \rvert$.
  
  The required estimates are obtained by applying 
  Lemmas \ref{lem:det}, \ref{lem:estimate1} and \ref{lem:estimate2}.
  For instance,
  expressing $\nu =rn+s$ 
  with $r, s\in \bbZ$ and $0\leq s<n$,
  we have $\Delta (\bZ, \bxi _2, \dots , \bxi _n) \equiv _{(-1)^{r(n-1)}} \Delta (\bZ ^{q^{-rn}}, \bxi _2, \dots , \bxi _n)^{q^r}$
  and
  \begin{equation*}
  \lvert \Delta (\bZ ^{q^{-rn}}, \bxi _2, \dots , \bxi _n) ^{q^r} (\tau) \rvert 
                                                         \leq \lvert \bxi _n \rvert ^{M_1(s)q^r} 
                                                         =\lvert \bt \rvert ^{n^{-1}q^{-(n-1)}M_1(s)q^r} 
                                                         =\lvert \bt \rvert ^{M_3(\nu)}.
  \end{equation*}
  Similarly,
  if $c=(q+1)/2$ or $c=1$
  according to the parity of $\nu$,
  and if $\mu =r'n+s'$
  with $r', s'\in \bbZ$ and $0\leq s'<n$,
  then 
  we have
  \begin{align*}
  &\Delta ([-1]_{\Nil ^{\flat}_{H_0}} (\bY _i^{q^{i-1}}), \bxi _2, \dots , \bY _j, \dots , \bxi _n) \\
  &\quad \equiv _{(-1)^{2r'(n-1)}} \Delta \left( [-1]_{\Nil ^{\flat}_{H_0}} (\bY _i^{q^{i-1-r'n}}), \bxi _2, \dots , \bY _j^{q^{-r'n}}, \dots , \bxi _n \right) ^{q^{2r'}}
  \end{align*}
  and 
  \begin{align*}
  &\lv \Delta \left( [-1]_{\Nil ^{\flat}_{H_0}} (\bY _i^{q^{i-1-r'n}}), \bxi _2, \dots , \bY _j^{q^{-r'n}}, \dots , \bxi _n \right) ^{q^{2r'}} (\tau) \rv \\ 
  &\leq \lvert \bxi _n \rvert ^{M_2(s', c)q^{2r'}} 
  =\lvert \bt \rvert ^{n^{-1}q^{-(n-1)}M_2(s', c)q^{2r'}}, \\
  &\lvert \Delta (\bxi _1, \dots , \bY _i, \dots , \bY _j, \dots , \bxi _n) (\tau) \rvert 
  \leq \lvert \bt \rvert ^{n^{-1}q^{-(n-1)}M_2(s', c)q^{2r'}},
  \end{align*}
  and one can check the equality $n^{-1}q^{-(n-1)}M_2(s', c)q^{2r'}=M_3(\nu)$ 
  by case-by-case calculations.

  Now we prove \ref{item:imcapunit}.
  It follows from \eqref{eq:alpha1onring}
  that under the canonical identifications
  $\LTpone (C)=U_K$,
  $\Nil ^{\flat, \text{ad}}_{\overline{\eta}} (C)=\Nil ^{\flat} (\calO _C)$,
  the morphism 
  $\alpha _1\colon \LTpone \rightarrow \Nil ^{\flat, \text{ad}}_{\overline{\eta}}$
  induces a map
  \begin{equation*}
  U_K\rightarrow \Nil ^{\flat} (\calO _C); \ x\mapsto (\Art _K(x)(t^{q^{-l}}))_{l\geq 0} 
  =[x]_{\Nil ^{\flat}_{\wedge H_0}}(\bt).
  \end{equation*}
  As we have $[x]_{\Nil ^{\flat}_{\wedge H_0}}(\bt)-_{\Nil ^{\flat}_{\wedge H_0}} \bt=[x-1]_{\Nil ^{\flat}_{\wedge H_0}}(\bt)$,
  Lemma \ref{lem:appD} \ref{ht1nilaction} shows that 
  $[x]_{\Nil ^{\flat}_{\wedge H_0}}(\bt) \in \Nil ^{\flat, \text{ad}}_{\overline{\eta}} (C)$
  is in the image of the inclusion $\calY _{\nu}(C)\hookrightarrow \Nil ^{\flat, \text{ad}}_{\overline{\eta}} (C)$
  if and only if $q^{v(x-1)}\geq M_3(\nu)$,
  which is to say,
  $\calY _{\nu}\cap \LTpone \simeq U_K^{\lceil \nu /n\rceil}$.
  \end{proof}
  \begin{Rem}
  The cancellations in \eqref{mainterm} are the point of departure of this paper.
  As noted in the introduction
  we generalized the computation in the $n=2$ case found in \cite[\S5.5]{WeSemi}
  by comparing the coordinate with that used in \cite{ITepitame}. 
  \end{Rem}
  
  We regard $\scrY _{\nu}=\Spf \calO _{C} \langle u^{q^{-\infty}}\rangle$
  as a formal model of $\calY _{\nu}$ by
  \begin{align*}
  (\calO _{C} \langle u^{q^{-\infty}}\rangle, \calO _{C} \langle u^{q^{-\infty}}\rangle) &\to 
  (\Gamma (\calY _{\nu}, \calO _{\calY _{\nu}}), \Gamma (\calY _{\nu}, \calO _{\calY _{\nu}}^{+})) \\
  \bu &\mapsto \bxi _n^{-nq^{n-1}M_3(\nu)} \bU.
  \end{align*}
  Then by the universality of a rational subset $\calY _{\nu}\subset \Nil ^{\flat, \text{ad}}_{\overline{\eta}}$
  (\cite[Prop.\ 1.3]{HuGen})
  the morphism $\Delta \colon \Nil ^{\flat, n, \text{ad}}_{\overline{\eta}} \to \Nil ^{\flat, \text{ad}}_{\overline{\eta}}$
  restricts to $\calX _{\nu}\to \calY _{\nu}$ and further extends to $\scrX _{\nu} \to \scrY _{\nu}$,
  the latter being
  induced by
  \begin{align} \label{eq:moroffms2}
  \calO _{C} \langle u^{q^{-\infty}}\rangle &\to \calO _{C} \langle {z'}^{q^{-\infty}}, {y'_2}^{q^{-\infty}}, \dots , {y'_n}^{q^{-\infty}}\rangle \\ \nonumber 
  \bu &\mapsto \bxi _n^{-nq^{n-1}M_3(\nu)} (\Delta (\bX _1, \dots , \bX _n)-_{\Nil ^{\flat}_{\wedge H_0}} \bt).
  \end{align}
  Here 
  \[
  \bxi _n^{-nq^{n-1}M_3(\nu)} (\Delta (\bX _1, \dots , \bX _n)-_{\Nil ^{\flat}_{\wedge H_0}} \bt)
  \in \Nil ^{\flat} (\calO _{C} \langle {z'}^{q^{-\infty}}, {y'_2}^{q^{-\infty}}, \dots , {y'_n}^{q^{-\infty}}\rangle)
  \]
  by Proposition \ref{Prop:premstudyonaf} \ref{item:imofaf}
  and Lemma \ref{lem:esttoapp}.
  On the other hand,
  the morphism
  $U_K^{\lceil \nu /n\rceil}\simeq \calY _{\nu} \cap \LTpone \rightarrow \calY _{\nu}$
  extends to a morphism 
  $\Spf \Cont (U_K^{\lceil \nu /n\rceil}, \calO _{C}) \rightarrow 
  \scrY _{\nu}=\Spf \calO _{C} \langle u^{q^{-\infty}}\rangle$
  of formal models
  induced by
  \begin{align} \label{eq:moroffms}
  \calO _{C} \langle u^{q^{-\infty}}\rangle &\to \Cont (U_K^{\lceil \nu /n\rceil}, \calO _{C}) \\ \nonumber
  \bu &\mapsto \left( \bxi _n^{-nq^{n-1}M_3(\nu)} ([x]_{\Nil ^{\flat}_{\wedge H_0}}(\bt)-_{\Nil ^{\flat}_{\wedge H_0}} \bt) \right)_{x\in U_K^{\lceil \nu /n\rceil}}.
  \end{align}
  Here $\bxi _n^{-nq^{n-1}M_3(\nu)} ([x]_{\Nil ^{\flat}_{\wedge H_0}}(\bt)-_{\Nil ^{\flat}_{\wedge H_0}} \bt)$
  indeed lies in $\Nil ^{\flat} (\calO _C)$
  by Proposition \ref{Prop:premstudyonaf} \ref{item:imcapunit}.
  With these morphisms
  we finally define a formal model $\scrZ _{\nu}$ of $\calZ _{\nu}$ by
  \begin{equation} \label{fmofZ}
  \scrZ _{\nu}=\scrX _{\nu} \times _{\scrY _{\nu}} \Spf \Cont (U_K^{\lceil \nu /n\rceil}, \calO _{C}).
  \end{equation}

  For a formal scheme $\scrA$ over $\calO _C$
  we denote its special fiber by $\overline{\scrA}$.
  
  We note that in the following theorem only the cases where $n$ and $\nu$ are coprime 
  (in particular $n$ does not divide $\nu$)
  are relevant to Main Theorem.
  \begin{Thm} \label{Thm:reduction}
   Let $\nu >0$ be an integer.
   Let $\calZ _{\nu}$ be the affinoid 
   defined in Subsection \ref{subsec:defofaf}
   and $\scrZ _{\nu}$ its formal model
   defined by (\ref{fmofZ}).
   For each integer $0\leq m \leq n-1$, 
   define a set $T(m)$ by
   \[  
   T(m)=\begin{cases}
               \{ (i, j)\in \bbZ ^2 \mid 1\leq i<j\leq n, \ m <j-i<n-m \} & \text{if $m <n/2$} \\
               \{ (i, j)\in \bbZ ^2 \mid 1\leq i<j\leq n, \ n-m \leq j-i\leq m \} & \text{if $m \geq n/2$}. 
              \end{cases}
   \]
   
   Then the special fiber $\overline{\scrZ} _{\nu}$ of $\scrZ _{\nu}$ 
   fits into the following Cartesian diagrams
   \begin{equation} \label{eq:cartdiag}
    \begin{CD}
     \overline{\scrZ} _{\nu} @>>> U_K^{\lceil \nu /n\rceil}=N_{L/K}U_L^{\nu} \\
     @VVV                             @VVV \\
     Z_{\nu}^{\emph{perf}} @>>> {N_{L/K}U_L^{\nu}}/{N_{L/K}U_L^{\nu +1}} \\
     @VVV @VVV \\
     \overline{\scrX} _{\nu}=\bbA _{\overline{k}}^{n, \emph{perf}} @>>> \overline{\scrY} _{\nu}=\bbA _{\overline{k} }^{1, \emph{perf}},
    \end{CD}
   \end{equation}
   where
   $Z_{\nu}$ is a smooth affine variety 
   defined below,
   $(\cdot)^{\emph{perf}}$ denotes the inverse perfection 
   of an affine scheme in characteristic $p$
   and we simply write
   \begin{align*}
   N_{L/K}U_L^{\nu}&=\Spec \Cont (N_{L/K}U_L^{\nu}, \overline{k} )=\Spec \Cont (U_K^{\lceil \nu/n\rceil}, \overline{k} ), \\
   {N_{L/K}U_L^{\nu}}/{N_{L/K}U_L^{\nu +1}}&=\Spec \Cont ({N_{L/K}U_L^{\nu}}/{N_{L/K}U_L^{\nu +1}}, \overline{k} ) \\
   &=
   \begin{cases}
   \Spec \overline{k} & \text{ if  $n$ does not divide $\nu$} \\
   \Spec \Cont (k, \overline{k}) & \text{ if $n$ divides $\nu$}.
   \end{cases}
   \end{align*}
   \begin{enumerate}[(1)]
    \item \label{item:red:divisible}
    Suppose that $n$ divides $\nu$ 
            (so that ${N_{L/K}U_L^{\nu}}/{N_{L/K}U_L^{\nu +1}}$
            is identified with $k$ via $\varpi'$).
            Then $Z_{\nu}$ is the trivial affine space bundle 
            $\amalg _{k} \bbA _{\overline{k} }^{n-1}$ 
            over ${N_{L/K}U_L^{\nu}}/{N_{L/K}U_L^{\nu +1}}=\Spec \Cont (k, \overline{k})$.
    \item \label{item:red:smallodd}
    Suppose\footnote{In the assertions to follow, the relations between various $y_i$ (resp.\ $z$) and various $y_i'$ (resp.\ $z'$)
                        are given in the proof.}
                         that $\nu$ is odd and $\nu \equiv \nu' \bmod {2n}$ with $0<\nu' <n$.
            Define $\mu' <n/2$ by $\nu' =2\mu' +1$.
            Then $Z_{\nu}$ is an affine variety defined by
            \[
             \begin{cases}
             y_1+\cdots +y_n=0                                              & \\
             z^q-z=-\displaystyle \sum _{(i, j)\in T(\mu')} y_iy_j &
             \end{cases}
            \]
            in $\bbA _{\overline{k}}^{n+1}$.
    \item \label{item:red:largeodd}
    Suppose that $\nu$ is odd and $\nu \equiv \nu' \bmod {2n}$ with $n<\nu' <2n$.
            Define $\mu' \geq n/2$ by $\nu' =2\mu' +1$.
            Then $Z_{\nu}$ is an affine variety defined by
            \[
             \begin{cases}
             y_1+\cdots +y_n=0                                              & \\
             z^q-z=\displaystyle \sum _{(i, j)\in T(\mu')} y_iy_j &
             \end{cases}
            \]
            in $\bbA _{\overline{k}}^{n+1}$.
    \item \label{item:red:smalleven}
    Suppose that $\nu$ is even and $\nu \equiv \nu' \bmod {2n}$ with $0<\nu' <n$.
            Define $\mu' <n/2$ by $\nu' =2\mu'$.
            Then $Z_{\nu}$ is an affine variety defined by
            \[
             \begin{cases}
             y_1+\cdots +y_n=0                                              & \\
             z^q-z= & \\
             \displaystyle \sum _{(i, j)\in T(\mu')} (y_i^q-y_i)(y_j^q-y_j)
                     + \displaystyle \sum _{j-i=\mu'} (y_i^q-y_i)y_j^q
                     + \displaystyle \sum _{j-i=n-\mu'} y_i^q(y_j^q-y_j)&
              \end{cases}
            \]
            in $\bbA _{\overline{k}}^{n+1}$.
    \item \label{item:red:largeeven}
    Suppose that $\nu$ is even and $\nu \equiv \nu' \bmod {2n}$ with $n<\nu' <2n$.
            Define $\mu' >n/2$ by $\nu' =2\mu'$.
            Then $Z_{\nu}$ is an affine variety defined by
            \[
             \begin{cases}
             y_1+\cdots +y_n=0                                              & \\
             z^q-z= &\\
             -\displaystyle \sum _{(i, j)\in T(n-\mu')} (y_i^q-y_i)(y_j^q-y_j)
                     + \displaystyle \sum _{j-i=\mu'} (y_i^q-y_i)y_j
                     + \displaystyle \sum _{j-i=n-\mu'} y_i(y_j^q-y_j)&
             \end{cases}
            \]
            in $\bbA _{\overline{k}}^{n+1}$.

   \end{enumerate}
  \end{Thm}
  \begin{Rem}
  	An analogous Cartesian diagram appears in the work of Boyarchenko-Weinstein \cite[Theorem 3.6.1]{BWMax}.
  	We were inspired by their result.
  \end{Rem}
  \begin{proof}
  It follows from (\ref{eq:moroffms}) and Lemma \ref{lem:appD} \ref{ht1nilaction} that
  $\Spec \Cont  (U_K^{\lceil \nu /n\rceil}, \overline{k} )=N_{L/K}U_L^{\nu}\rightarrow \overline{\scrY} _{\nu}=\bbA _{\overline{k} }^{1, \text{perf}}=\Spec \overline{k} [u^{q^{-\infty}}]$
  is given by 
  \[
  \overline{k} [u^{q^{-\infty}}]\rightarrow \Cont (U_K^{\lceil \nu /n\rceil}, \overline{k}); \ \bu \mapsto \begin{cases}
  (0)_{x\in U_K^{\lceil \nu /n\rceil}} &\text{ if $n$ does not divide $\nu$} \\
  (\overline{y})_{1+y{\varpi'} ^{\nu /n}\in U_K^{ \nu /n}} &\text{ if $n$ divides $\nu$}
 \end{cases}
  \]
  where $0, \overline{y} \in k=\mu _{q-1}(\overline{k})\cup \{ 0\}$ are identified with the constant system
  and in the second case we apply Lemma \ref{lem:appD} \ref{ht1nilaction} and Corollary \ref{Cor:tandxi} to compute
  $\overline{\bxi _n^{-nq^{n-1}M_3(\nu)} [y{\varpi'}^{\nu /n}]_{\Nil ^{\flat}_{\wedge H_0}}(\bt)} =\overline{y}$ in $\Nil ^{\flat}(\overline{k})$. 
  This shows the desired factorization of
  $N_{L/K}U_L^{\nu}\rightarrow \overline{\scrY} _{\nu}$.

  Thus,
  it remains to 
  study the morphism 
  $\overline{\scrX} _{\nu} 
  \rightarrow 
  \overline{\scrY} _{\nu}$
  induced as the reduction of \eqref{eq:moroffms2}.
  We claim that
  (in \ref{item:red:smallodd} to \ref{item:red:largeeven})
  an isomorphism is given by\footnote{Since we are working with perfect rings, there are many other obvious possibilities;
                                                   for instance, we may leave out all $r$ from the definition of the above map.} 
  \begin{equation} \label{eq:zyandprime}
  \bz'\mapsto \bz ^{q^{-(r-\nu)}}, \quad \by'_i\mapsto {\by _i}^{q^{-(r-\mu +i-1)}} \ (2\leq i\leq n),
  \end{equation}
  where we write $\nu =rn+s$ with $r, s\in \bbZ$ and $0\leq s<n$,
  and $\mu =\lfloor \nu /2\rfloor$.
  By Lemma \ref{lem:esttoapp}
  it suffices to approximate 
  $\Delta (\bX _1, \dots , \bX _n)-_{\Nil ^{\flat}_{\wedge H_0}} \bt 
  \in \Nil ^{\flat}(\calO _{C} \langle {z'}^{q^{-\infty}}, {y'_2}^{q^{-\infty}}, \dots , {y'_n}^{q^{-\infty}}\rangle)$
  in terms of valuations.
  As in the proof of Proposition \ref{Prop:premstudyonaf} \ref{item:imofaf}
  we only need to study the three kinds of terms explicitly appearing in \eqref{mainterm}
  by Lemma \ref{lem:app2}.
  We again use Lemma \ref{lem:det} to reduce to the case 
  where $0\leq \nu <n$ (resp.\ where $0\leq \mu <n$),
  so that Lemma \ref{lem:estimate1} (resp.\ Lemma \ref{lem:estimate2})
  is applicable.
  Since the computation is rather complicated, we only indicate several typical cases.
  Let $\tau \in \calX _{\nu}$.
  
  Since $\Delta (\bZ, \bxi_2, \dots , \bxi _n) \equiv _{(-1)^{r(n-1)}} \Delta (\bZ ^{q^{-rn}}, \bxi_2, \dots , \bxi _n)^{q^r}$,
  we have, by Lemma \ref{lem:estimate1}, 
  $\lvert (\Delta (\bZ, \bxi_2, \dots , \bxi _n)-\bm{f}) (\tau)\rvert 
  <\lvert \bxi _n \rvert ^{nq^{n-1}M_3(\nu)}$,
  where, if $s>0$, 
  \begin{align*}
  \bm{f} &=
     (-1)^{r(n-1)} \cdot (-1)^{s}\bxi _n^{(n-s-1+(s-1)q)q^{n-1+r}} \\
     &\qquad \cdot \left( \bxi _n^{q^{n+r}} (\bZ ^{q^{-rn}})^{q^{-s+r}}-\bxi _n^{q^{n-1+r}} (\bZ ^{q^{-rn}})^{q^{-s+1+r}}
                                                                           \right)  \\
   &=(-1)^{\nu -r}  \bxi _n^{(n-s-1+(s-1)q)q^{n-1+r}} \\
   &\qquad \cdot \left( \bxi _n^{q^{n+r}} (\bxi _n^{q^{n-1+\nu}}\bz')^{q^{-\nu +r}}-\bxi _n^{q^{n-1+r}} (\bxi _n^{q^{n-1+\nu}}\bz')^{q^{-\nu +1+r}} 
                                                                       \right) \\
   &=(-1)^{\nu -r}  \bxi _n^{(n-s+sq)q^{n-1+r}}  ({\bz'}^{q^{-\nu +r}}- {\bz'}^{q^{-\nu +1+r}}) \\
   &=(-1)^{\nu -r}  \bxi _n^{nq^{n-1}M_3(\nu)}  (\bz- \bz ^q) \\
  \intertext{and, if $s=0$,}
  \bm{f} &=(-1)^{r(n-1)} \bxi _n^{(n-1)q^{n-1+r}} \bZ ^{q^{-rn+r}}  \\
   &=(-1)^{\nu -r} \bxi _n^{nq^{n-1}M_3(\nu)} \bz.
  \end{align*}
  We put 
  $\Delta (\bZ, \bxi_2, \dots , \bxi _n)=(\Delta ^{q^{-l}}(\bZ, \bxi_2, \dots , \bxi _n))_{l\geq 0}$
  and 
  $\bm{f}=(f^{(l)})_{l\geq 0}$.  
  Then by Lemma \ref{lem:esttoapp} we have 
  \[
  \Delta ^{q^{-l}}(\bZ, \bxi_2, \dots , \bxi _n)\equiv f^{(l)} 
  \]
  modulo $(\xi _n^{q^{-l}})^{nq^{n-1}M_3(\nu)}\frakp _C \calO _{C} \langle {z'}^{q^{-\infty}}, {y'_2}^{q^{-\infty}}, \dots , {y'_n}^{q^{-\infty}}\rangle$.
  In particular, \ref{item:red:divisible} follows.
  
  Suppose that we are in the case \ref{item:red:smallodd}.
  In particular, we have $\mu =(\nu -1)/2=r'n+\mu '$ with $r'=r/2\in \bbZ$.
  Similarly to the above computation, since
  $\Delta (\bxi _1, \dots, \bY _i, \dots , \bY _j, \dots , \bxi _n)\equiv _{(-1)^{2r'(n-1)}} \Delta (\bxi _1, \dots, \bY _i^{q^{-r'n}}, \dots , \bY _j^{q^{-r'n}}, \dots , \bxi _n)^{q^{2r'}}$,
  we have, by Lemma \ref{lem:estimate2} \ref{2:c>1, mu:small},
  $\lvert (\Delta (\bxi _1, \dots, \bY _i, \dots , \bY _j, \dots , \bxi _n) -\bm{f} _{i, j}) (\tau) \rvert
  <\lvert \bxi _n\rvert ^{nq^{n-1}M_3(\nu)}$,
  where, if $\mu '<j-i<n-\mu '$,
  \begin{align*}
   \bm{f} _{i, j}&= \bxi _n^{(n-2\mu '-2+2\mu 'q)q^{n-1+2r'}} (\bY _i^{q^{-r'n}})^{q^{i-\mu '-1+2r'}} (\bY _j^{q^{-r'n}})^{q^{j-\mu '-1+2r'}} \\
          &=\bxi _n^{(n-2\mu '-2+2\mu 'q)q^{n-1+r}} (\bxi _i^{q^{\mu}(q+1)/2}\by'_i)^{q^{i-\mu -1+r}} (\bxi _j^{q^{\mu}(q+1)/2}\by'_j)^{q^{j-\mu -1+r}} \\
          &=\bxi _n^{(n-2\mu '-2+2\mu 'q)q^{n-1+r}} \bxi _n^{q^{n-1+r}(q+1)/2}{\by'_i}^{q^{i-\mu -1+r}} \bxi _n^{q^{n-1+r}(q+1)/2}{\by'_j}^{q^{j-\mu -1+r}} \\
          &=\bxi _n^{(n-\nu '+\nu 'q)q^{n-1+r}} {\by'_i}^{q^{i-\mu -1+r}} {\by'_j}^{q^{j-\mu -1+r}} \\
          &=\bxi _n^{nq^{n-1}M_3(\nu)} \by _i \by _j
  \end{align*}
  and, otherwise, $\bm{f} _{i, j}=0$.
  Essentially the same computation shows that 
  \[
  \lvert (\Delta ([-1]_{\Nil ^{\flat}_{H_0}}(\bY _i^{q^{i-1}}), \bxi _2, \dots, \bY _j, \dots , \bxi _n) -\bm{f}' _{i, j})(\tau) \rvert
  <\lvert \bxi _n \rvert ^{nq^{n-1}M_3(\nu)},
  \]
  where
  \[
  \bm{f}' _{i, j}=
  \begin{cases}
  -\bxi _n^{nq^{n-1}M_3(\nu)} \by _i \by _j
  &\text{ if $\mu '<j-1<n-\mu '$} \\
  0
  &\text{ otherwise}.
  \end{cases}
  \]
  Therefore,
  if we put 
  \begin{gather*}
  \bxi _n^{-nq^{n-1}M_3(\nu)}(\Delta (\bX _1, \dots, \bX _n)-_{\Nil ^{\flat}_{\wedge H_0}} \bt)=(\rho ^{(l)})_{l\geq 0}, \\
  \quad \bm{f} _{i, j}=(f_{i, j}^{(l)})_{l\geq 0},
  \quad \bm{f}' _{i, j}=({f'}_{i, j}^{(l)})_{l\geq 0},
  \end{gather*}
  then by Lemma \ref{lem:app2}
  and Lemma \ref{lem:esttoapp}
  we have
  \begin{align*}
  \rho ^{(l)}
  &\equiv 
  (\xi _n^{q^{-l}})^{-nq^{n-1}M_3(\nu)}(f^{(l)}
  +\sum _{\substack{2\leq i, j\leq n \\ \mu '<j-i<n-\mu '}} f_{i, j}^{(l)}
  +\sum _{\substack{2\leq i, j\leq n \\ \mu '<j-1<n-\mu '}} {f'}_{i, j}^{(l)}) \\
  &\equiv (z^q-z +\sum _{(i, j)\in T(\mu ')} y _iy _j)^{q^{-l}}
  \end{align*}
  modulo $\frakp _C \calO _{C} \langle {z'}^{q^{-\infty}}, {y'_2}^{q^{-\infty}}, \dots , {y'_n}^{q^{-\infty}}\rangle$.
  Here and below we identify 
  \begin{align*}
  \overline{\scrX _{\nu}}&(=\Spec \overline{k}[z^{q^{-\infty}}, y_2^{q^{-\infty}}, \dots, y_n^{q^{-\infty}}]) \\
  &=\Spec \overline{k}[z^{q^{-\infty}}, y_1^{q^{-\infty}}, \dots, y_n^{q^{-\infty}}]/((y_1+\dots +y_n)^{q^{-l}}\mid l\geq 0).
  \end{align*}
  This completes the proof of \ref{item:red:smallodd}.
  
  Suppose that we are in the case \ref{item:red:smalleven}.
  In this case $\mu =\nu /2=r'n+\mu '$ with $r'=r/2\in \bbZ$.
  Again, by Lemma \ref{lem:estimate2} \ref{2:c=1, mu:small},
  $\lvert (\Delta (\bxi _1, \dots, \bY _i, \dots , \bY _j, \dots , \bxi _n) -\bg _{i, j}) (\tau)\rvert
  <\lvert \bxi _n \rvert ^{nq^{n-1}M_3(\nu)}$,
  where, if $\mu '<j-i<n-\mu '$,
  \begin{align*}
  \bg _{i, j}&= \bxi _n^{(n-2\mu '-2+(2\mu '-2)q)q^{n-1+2r'}} \\
  &\quad \quad \cdot \left( \bxi _n^{q^{n+2r'}}(\bY _i^{q^{-r'n}})^{q^{i-\mu '-1+2r'}}-\bxi _n^{q^{n-1+2r'}}(\bY _i^{q^{-r'n}})^{q^{i-\mu '+2r'}}
                                                                        \right) \\
                                                                        &\quad \quad \quad \quad \cdot \left( \bxi _n^{q^{n+2r'}}(\bY _j^{q^{-r'n}})^{q^{j-\mu '-1+2r'}}-\bxi _n^{q^{n-1+2r'}}(\bY _j^{q^{-r'n}})^{q^{j-\mu '+2r'}}
                                                                          \right) \\
         &= \bxi _n^{(n-2\mu '-2+(2\mu '-2)q)q^{n-1+r}} \left( \bxi _n^{q^{n+r}}(\bxi _i^{q^{\mu}}\by' _i)^{q^{i-\mu -1+r}}-\bxi _n^{q^{n-1+r}}(\bxi _i^{q^{\mu}}\by' _i)^{q^{i-\mu +r}}
                                                                        \right) \\
                                                                        &\quad \quad \cdot \left( \bxi _n^{q^{n+r}}(\bxi _j^{q^{\mu }}\by' _j)^{q^{j-\mu -1+r}}-\bxi _n^{q^{n-1+r}}(\bxi _j^{q^{\mu}}\by' _j)^{q^{j-\mu +r}}
                                                                          \right) \\
          &=\bxi _n^{(n-\nu' +\nu' q)q^{n-1+r}} ( {\by' _i}^{q^{i-\mu -1+r}}-{\by' _i}^{q^{i-\mu +r}}) 
                                                             ( {\by' _j}^{q^{j-\mu -1+r}}-{\by' _j}^{q^{j-\mu +r}}) \\
          &=\bxi _n^{nq^{n-1}M_3(\nu)} 
          ( {\by_i}-{\by_i}^q ) 
          ( {\by_j}-{\by_j}^q )
  \end{align*}
  and 
  \[
  \bg _{i, j}=\begin{cases}
            -\bxi _n^{nq^{n-1}M_3(\nu)}  ( \by _i-\by _i^q ) \by _j^q &\text{ if $\mu '=j-i$} \\
            -\bxi _n^{nq^{n-1}M_3(\nu)}  \by _i^q ( \by_j-\by _j^q ) &\text{ if $j-i=n-\mu '$} \\
            0 &\text{ otherwise.}
           \end{cases}
  \]
  We also compute an approximation for $\Delta ^{q^{-l}}([-1]_{\Nil ^{\flat}_{H_0}}(\bY _i^{q^{i-1}}), \bxi _2, \dots, \bY _j, \dots , \bxi _n)$ in essentially the same way
  and thus obtain as before (by Lemma \ref{lem:esttoapp})
  \begin{align*}
  &\rho ^{(l)} \\ 
  &\equiv (z-z^q)^{q^{-l}} \\
  &\quad +\sum _{\substack{ 2\leq i, j\leq n \\ \mu '<j-i<n-\mu '}} \left( \left( y_i-y_i^q\right) \left( y_j-y_j^q\right) \right)^{q^{-l}}
                     -\sum _{\substack{ 2\leq i, j\leq n \\ \mu '<j-1<n-\mu '}} \left( \left( y_i-y_i^q\right) \left( y_j-y_j^q\right) \right)^{q^{-l}} \\
  &\quad \quad -\sum _{\substack{ 2\leq i, j\leq n \\ \mu '=j-i}} \left( \left( y_i-y_i^q\right) y_j^q\right) ^{q^{-l}}
            +\sum _{\substack{ 2\leq i, j\leq n \\ \mu '=j-1}} \left( \left( y_i-y_i^q\right) y_j^q\right)^{q^{-l}} \\
  &\quad \quad \quad
  -\sum _{\substack{ 2\leq i, j\leq n \\ j-i=n-\mu '}} \left( y_i^q \left( y_j-y_j^q\right) \right)^{q^{-l}}
            +\sum _{\substack{ 2\leq i, j\leq n \\ j-1=n-\mu '}} \left( y_i^q \left( y_j-y_j^q\right) \right)^{q^{-l}} \\
  &\equiv \Biggl( z-z^q \\
  &\quad +\sum _{(i, j)\in T(\mu ')} \left( y_i-y_i^q\right) \left( y_j-y_j^q\right) 
  -\sum _{\mu '=j-i} \left( y_i-y_i^q\right) y_j^q 
  -\sum _{j-i=n-\mu '} y_i^q \left( y_j-y_j^q\right) \Biggr)^{q^{-l}}.
  \end{align*}
  The other cases are treated in the same vein.
  \end{proof}
  
 \subsection{Stabilizers of the affinoids and their actions on the reductions} \label{subsec:stabandaction}
 To state the main result of this subsection
 we define several subgroups of $\GL _n(K)$ and $D^{\times}$.
 
 Let $\frakI \subset M_n(K)$ (resp.\ Let $\frakP \subset \frakI$) be the inverse image of the set of upper triangular matrices
 (resp.\ upper triangular matrices with all the diagonal entries zeros) by the canonical map $M_n(\calO )\rightarrow M_n(k)$.
 Note that $\frakP ^i=\varphi ^i\frakI$ for $i\geq 0$.
 We set $U_{\frakI}=\frakI ^{\times}$ and $U_{\frakI}^i=1+\frakP ^i\subset U_{\frakI}$ for $i\geq 1$ 
 as usual.
 
 Recall that $L$ is identified with $K(\varphi )\subset M_n(K)$
 via $i_{\xi}$.
 We also define 
 $C_1$
 to be the orthogonal complement of 
 $L\subset M_n(K)$
 with respect to the non-degenerate symmetric pairing
 $M_n(K)\times M_n(K)\rightarrow K; \ (x, y)\mapsto \tr (xy)$.
 For $i\geq 1$ we set $\frakP _{C_1}^i=\frakP ^i\cap {C_1}$ and 
 $
 U_{\frakI}^{(i)}=1+\frakP ^i+\frakP _{C_1}^{\lfloor(i+1)/2\rfloor}.
 $
 Then we have $\frakP ^i=\frakp _L^i\oplus \frakP _{C_1}^i$\footnote{This follows 
 either from the tame ramification assumption (\cite[Lemma 3]{HoTame}, \cite[Remarks after (1.4.10), (1.3.8)]{BKtypes}), or from Proposition \ref{Prop:GL(nu)explicit} below.
 }
 and therefore
 $U_{\frakI}^{(i)}=1+\frakp _L ^i+\frakP _{C_1}^{\lfloor(i+1)/2\rfloor}$.
 Also, $U_{\frakI}^{(i)}$ is clearly a subgroup of $U_{\frakI}^{\lfloor(i+1)/2\rfloor}$
 containing $U_{\frakI}^i$.
 \begin{Rem}
 The author learned the importance of $C_1$ and $U_{\frakI}^{(i)}$
 in \cite{BWMax},
 where similar objects are considered for unramified extensions $L/K$,
 although these have been studied before elsewhere
 (see for instance, \cite[(6.2)]{BFGdiv}).
 \end{Rem}
 As $\varphi$ is monomial,
 we can make the above objects more explicit.
 \begin{Prop} \label{Prop:GL(nu)explicit}
 Let $D_n(K)\subset M_n(K)$ be the $K$-subspace of diagonal matrices
 and $D_n(K)^0=\{ a\in D_n(K)\mid \tr a=0\}$.
 \begin{enumerate}[(1)]
 \item We have
 $C_1=\bigoplus _{0\leq i\leq n-1}\varphi ^{i}D_n(K)^0$
 and any $K$-basis of $D_n(K)^0$ is an $L$-basis of $C_1$.
 \item
 The orthogonal projection $p_1\colon M_n(K)\to L$
 is given by
 \[
 p_1\colon M_n(K)\to L; \ a\mapsto 
 \frac{1}{n}\tr a+\frac{1}{n}\varphi\tr (\varphi ^{-1}a) +\cdots +\frac{1}{n}\varphi ^{n-1}\tr (\varphi ^{-(n-1)}a).
 \]
 \item \label{UforGL}
 We have 
 \begin{align*}
 U_{\frakI }^m
 &=\{ 1+(a_{i, j})_{1\leq i, j\leq n}\in M_n(K)\mid \sum _{1\leq j\leq n}a_{j, i}\varphi _L^{i-j}\in \frakp _L^m \text{ for all $1\leq i\leq n$}\} \\
 &=\{ 1+(a_{i, j})_{1\leq i, j\leq n}\in M_n(K)\mid a_{j, i}\varphi _L^{i-j}\in \frakp _L^m \text{ for all $1\leq i, j\leq n$}\}, \\
 U_{\frakI}^{(m)}&=\{ g\in U_{\frakI }^{\lfloor (m+1)/2\rfloor} \mid p_1(g-1)\in \frakp _L^m\}.
 \end{align*}
 \end{enumerate}
 \end{Prop} 
 \begin{proof}
 All the assertions follow from simple computations.
 We only show the first and second equalities of \ref{UforGL}.
 The set $\{ \diag(0,\cdots,0,\overset{j}{\check{1}},0,\cdots,0) \mid 1\leq j\leq n\}$
 is a $K$-basis of $D_n(K)$
 and hence an $L$-basis of $M_n(K)=\bigoplus _{0\leq i\leq n-1}\varphi ^{i}D_n(K)=D_n(K)\otimes _K L$,
 thus giving rise to an isomorphism
 \[
 M_n(K)\to L^n; \ \sum _{0\leq i\leq n-1}\varphi ^{i}\diag(a_1^{(i)}, \dots, a_n^{(i)})\mapsto \left( \sum _{0\leq i\leq n-1}a_j^{(i)}\varphi _L^i\right) _{1\leq j\leq n}.
 \]
 We deduce the desired equality by observing that 
 this isomorphism sends $(a_{i, j})_{1\leq i, j\leq n}\in M_n(K)$ to $(\sum _{1\leq j\leq n}a_{j, i}\varphi _L^{i-j})_{1\leq i\leq n}$
 and $\frakP ^m\subset M_n(K)$ to $(\frakp _L^m)^n$.
 The second equality follows because $v_D(a_{j, i}\varphi _L^{i-j})$ are distinct for $1\leq j\leq n$.
 \end{proof}
 \begin{Prop} \label{Prop:gpS1}
   The subgroup ${U_{\frakI}^{(m+1 )}}\subset {U_{\frakI}^{(m )}}$ is normal. 
   The quotient group $S_{1, m}={U_{\frakI}^{(m)}}/{U_{\frakI}^{(m+1)}}$ is described as follows.
  \begin{enumerate}[(1)]
   \item Suppose that $m$ is odd.
   Then $S_{1, m}$ is isomorphic
   to the additive group $k$: 
   \begin{equation*}
    S_{1, m}\rightarrow k; \ (1+a)U_{\frakI}^{(m+1)}\mapsto 
   \overline{\tr (\varphi ^{-m}a)}
   \end{equation*}
   \item \label{item:evenS1}
   Suppose that $m$ is even.
   Put
   \[
   T_{1, m}=\{ (v, (w_i))\in k\times k^{{\bbZ} /{n\bbZ}}
                  \mid 
                  \sum _{i}w_i=0 \}.
   \]
   Then the following is a bijection:
   \begin{equation*}
    S_{1, m}\rightarrow T_{1, m}
    ; \ (1+a)U_{\frakI}^{(m+1)}\mapsto 
    \left( \overline{\tr (\varphi ^{-m}a)},
     \overline{\varphi ^{-\lfloor (m+1)/2\rfloor} a}\right).
   \end{equation*}
   Under this identification, 
   the induced group operation on $T_{1, m}$ is described as follows:
   \begin{equation} \label{eq:S1formula}
    (v, (w_i))\cdot (v', (w'_i))
    =(v+v'+\sum _{i}w_i w'_{i+(m/2)}, (w_i+w_i')). 
   \end{equation}
  \end{enumerate}
 \end{Prop}  
 \begin{proof}
 The first assertion being straightforward, 
 we only show the second assertion.
 Thus we assume $m$ is even, so that $\lfloor (m+1)/2\rfloor=m/2$
 and $\lfloor (m+2)/2\rfloor=m/2+1$.
 We consider the following composite of maps:
 \begin{alignat*}{2}
 &U_{\frakI}^{(m)}=1+(\frakp _L^m\oplus \frakP _{C_1}^{m/2}) &\quad &1+a=1+(p_1(a), a-p_1(a)) \\
 &\to ({\frakp _L^m}/{\frakp _L^{m+1}})\oplus ({\frakP _{C_1}^{m/2}}/{\frakP _{C_1}^{m/2+1}}) &\quad &\mapsto (p_1(a)+\frakp _L^{m+1}, (a-p_1(a))+\frakP _{C_1}^{m/2+1}) \\
 &\stackrel{\sim}{\to} T_{1, m} &\quad &\mapsto (\overline{\varphi ^{-m}p_1(a)}, \overline{\varphi ^{-m/2}(a-p_1(a))}) \\
 &\stackrel{(n, \text{id})}{\stackrel{\sim}{\to}} T_{1, m} &\quad & \mapsto (\overline{\tr (\varphi ^{-m}a)}, \overline{\varphi ^{-m/2}a}). 
 \end{alignat*}
 We check that this composite is a group homomorphism as follows.
 Let $b, b'\in \frakp _L^{m}$ and $c, c'\in \frakP _{C_1}^{m/2}$.
 As $bb', bc', cb'\in \frakP ^{m+1}$, we have
 $(1+b+c)(1+b'+c')-1\equiv (b+b'+p_1(cc'))+(c+c') \pmod{\frakp _L^{m+1}+\frakP _{C_1}^{m/2+1}}$.
 We express $c=\sum _{1\leq j\leq n}\varphi ^{m/2 +j-1}d_j$, 
 $c'=\sum _{1\leq j\leq n}\varphi ^{m/2 +j-1}d'_j$
 with $d_j, d'_j\in D_n(K)^0\cap M_n(\calO)$,
 and also $d_1=\diag(w_1, \dots, w_n)$, $d'_1=\diag(w'_1, \dots, w'_n)$.
 Then we have
 $\tr (\varphi ^{-m}cc')\equiv \sum _{i}w_iw'_{i+(m/2)} \pmod{\frakp}$
 (here the index is identified with its image in ${\bbZ}/n{\bbZ}$)
 because $\varphi \diag(w'_1, \dots, w'_n) \varphi ^{-1}=\diag(w'_2, \dots, w'_n, w'_1)$.
 
 The homomorphism clearly induces the isomorphism stated in the proposition.
 \end{proof}
 For even $m$,
 we identify $S_{1, m}$ with $T_{1, m}$.
  
 We have similar subgroups for $D^{\times}$.
 Let $\calO _D$ be the maximal order of $D$
 and $\frakP _D\subset \calO _D$ the maximal ideal.
 Note that $\frakP _D^i=\varphi _D^i\calO _D$ for all $i$.
 We set $U_D=\calO _D^{\times}$ and $U_D^i=1+\frakP _D^i\subset U_D$
 as usual.
 
 Recall that $L$ is identified with $K(\varphi _D)\subset D$
 via $i_{\xi}^D$.
 Similarly to the above, we define
 $C_2$
 to be the orthogonal complement of 
 $L\subset D$
 with respect to the non-degenerate symmetric pairing
 $D\times D\rightarrow K; \ (x, y)\mapsto \Trd (xy)$.
We set $\frakP _{C_2}^i=\frakP _D^i\cap {C_2}$ and 
 $
 U_{D}^{(i)}=1+\frakP _D^i+\frakP _{C_2}^{\lfloor(i+1)/2\rfloor}.
 $
 Then we have $\frakP _D^i=\frakp _L^i\oplus \frakP _{C_2}^i$\footnote{This follows either from the tame ramification assumption (\cite[(6.2.3)]{BFGdiv}), or from Remark \ref{Rem:D(nu)explicit} below.}
 and 
 $U_{D}^{(i)}=1+\frakp _L ^i+\frakP _{C_2}^{\lfloor(i+1)/2\rfloor}$.
 Also, $U_{D}^{(i)}$ is clearly a subgroup of $U_{D}^{\lfloor(i+1)/2\rfloor}$
 containing $U_{D}^i$.
 
 \begin{Rem} \label{Rem:D(nu)explicit}
 As is common it is often convenient to realize $D$ as a $K$-subalgebra of $M_n(K_n)$ in the following way.
 Let $F\in \Gal ({K_n}/K)$ be the lift of the arithmetic Frobenius element.
 We also denote by $F$ the $K$-algebra automorphism of $M_n(K_n)$ 
 obtained by applying $F$ to each entries.
 Then we have a $K$-algebra isomorphism:
 \begin{align*}
 D&\stackrel{\sim}{\to} M_n(K_n)^{F\circ \Ad(\varphi ^{-1})}=\{ a\in M_n(K_n)\mid F(a)=\varphi a\varphi ^{-1}\} \\
 K_n\ni a&\mapsto M(a)=\diag(a, F(a), \dots, F^{n-1}(a)) \\
 \varphi _D&\mapsto \varphi.
 \end{align*}
 With this realization the following are easy to verify. 
 \begin{enumerate}[(1)]
 \item Let $K_n^0=\{ a\in K_n\mid \Tr _{K_n/K}a=0\}$. 
 We have
 $C_2=\bigoplus _{0\leq i\leq n-1}K_n^0\varphi _D^{i}$
 and any $K$-basis of $K_n^0$ is an $L$-basis of $C_2$.
 \item
 The orthogonal projection $p_2\colon D\to L$
 is given by
 \[
 p_2\colon D\to L; \ a\mapsto 
 \frac{1}{n}\Trd a+\frac{1}{n}\Trd (\varphi _D^{-1}a)\varphi _D+\cdots +\frac{1}{n}\Trd (\varphi _D^{-(n-1)}a)\varphi _D^{n-1}.
 \]
 \item \label{UforD}
 We have 
 \begin{align*}
 U_{D}^m&=\{ d\in D^{\times}\mid v_D(d-1)\geq m\}, \\
 U_{D}^{(m)}&=\{ d\in U_{D}^{\lfloor (m+1)/2\rfloor} \mid p_2(d-1)\in \frakp _L^m\}.
 \end{align*}
 \end{enumerate} 
 \end{Rem}
 As before, we have the following proposition.
 \begin{Prop} \label{Prop:gpS2}
   The subgroup ${U_{D}^{(m+1 )}}\subset {U_{D}^{(m )}}$ is normal. 
   The quotient group $S_{2, m}={U_{D}^{(m)}}/{U_{D}^{(m+1)}}$ is described as follows.
  \begin{enumerate}[(1)]
   \item Suppose that $m$ is odd.
   Then $S_{2, m}$ is isomorphic
   to the additive group $k$: 
   \begin{equation*}
    S_{2, m}\rightarrow k; \ d{U_{D}}^{(m+1)}\mapsto 
    \overline{\Trd ((d^{-1}-1)\varphi _D^{-m})}
    \end{equation*}
   \item \label{item:evenS2}
   Suppose that $m$ is even.
   Put
   \[
   T_{2, m}=\{ (v, w)\in k\times k_n
                  \mid 
                  \Tr _{k_n/k}w=0 \}.
   \]
   Then the following is a bijection:
   \begin{equation*}
    S_{2, m}\rightarrow T_{2, m}
    ; \ d{U_{D}}^{(m+1)}\mapsto 
    \left(\overline{\Trd ((d^{-1}-1)\varphi _D^{-m})},
     \overline{(d^{-1}-1)\varphi _D^{-m/2}}\right)
   \end{equation*}
   and under this identification, 
   the induced group operation on $T_{2, m}$ is described as follows:
   \begin{equation*}
    (v, w)\cdot (v', w')
    =(v+v'+\Tr _{k_n/k} (w^{q^{m/2}}w'), w+w') .
   \end{equation*}
  \end{enumerate}
 \end{Prop}
 \begin{Rem}
 In the definition of the bijection in Proposition \ref{Prop:gpS2} \ref{item:evenS2}
 one could replace $d^{-1}$ (appearing twice) with $d$.
 The resulting operation on $T_{2, m}$ would be opposed (see the following proof).
 Although the bijection thus obtained may be more natural,
 we choose the above normalization in this paper.
 \end{Rem}
 \begin{proof}
 Again we only show the second assertion.
 Thus we assume $m$ is even, so that $\lfloor (m+1)/2\rfloor=m/2$.
 Let $T_{2, m}^{\text{op}}$ be the opposite group of $T_{2, m}$;
 for $(v, w), (v', w')\in T_{2, m}^{\text{op}}$
 we have 
 \[
 (v, w)\cdot (v', w')
    =(v+v'+\Tr _{k_n/k} (w{w'}^{q^{m/2}}), w+w').
 \]
 We consider the following composite of maps:
 \begin{alignat*}{2}
 &U_D^{(m)}=1+(\frakp _L^m\oplus \frakP _{C_2}^{m/2}) &\quad &d=1+(p_2(d-1), d-1-p_2(d-1)) \\
 &\to ({\frakp _L^m}/{\frakp _L^{m+1}})\oplus ({\frakP _{C_2}^{m/2}}/{\frakP _{C_2}^{m/2+1}}) &\quad &\mapsto (p_2(d-1)+\frakp _L^{m+1}, \\
 &&&\qquad (d-1-p_2(d-1))+\frakP _{C_2}^{m/2+1}) \\
 &\stackrel{\sim}{\to} T_{2, m} &\quad &\mapsto \left( \overline{p_2(d-1)\varphi _D^{-m}} \right. , \\
 &&&\qquad \left. \overline{(d-1-p_2(d-1))\varphi _D^{-m/2}}\right) \\
 &\stackrel{(n, \text{id})}{\stackrel{\sim}{\to}} T_{2, m} &\quad & \mapsto \left( \overline{\Trd ((d-1)\varphi _D^{-m})}, \overline{(d-1)\varphi _D^{-m/2}}\right). 
 \end{alignat*}
 We check as follows that this composite is a group homomorphism if we regard the target as $T_{2, m}^{\text{op}}$.
 Let $b, b'\in \frakp _L^{m}$ and $c, c'\in \frakP _{C_2}^{m/2}$.
 As $bb', bc', cb'\in \frakP _D^{m+1}$, we have
 $(1+b+c)(1+b'+c')-1\equiv (b+b'+p_2(cc'))+(c+c') \pmod{\frakp _L^{m+1}+\frakP _{C_2}^{m/2+1}}$.
 In the notation of Remark \ref{Rem:D(nu)explicit}
 we express $c=\sum _{1\leq j\leq n}M(w^{(j)})\varphi ^{m/2 +j-1}$, 
 $c'=\sum _{1\leq j\leq n}M({w'}^{(j)})\varphi ^{m/2 +j-1}$
 with $w^{(j)}, {w'}^{(j)}\in \calO _{K_n}$
 and put $w=w^{(1)}$, $w'={w'}^{(1)}$. 
 Then we have
 \[
 \Trd (cc'\varphi _D^{-m})\equiv \sum _{1\leq i\leq n}w^{q^i}w'^{q^{i+m/2}} \pmod{\frakp _L}
 \]
 as desired,
 because $\varphi M(w') \varphi ^{-1}=M(F(w'))$.
 
 The homomorphism clearly induces an isomorphism $S_{2, m}\stackrel{\sim}{\to} T_{2, m}^{\text{op}}$.
 We compose two anti-isomorphisms
 \begin{align*}
 &S_{2, m}\stackrel{\sim}{\to} S_{2, m}; \ dU_D^{(m+1)}\mapsto d^{-1}U_D^{(m+1)}, \\
 &T_{2, m}^{\text{op}}\stackrel{\sim}{\to} T_{2, m}; \ (v, w)\mapsto (v, w)
 \end{align*}
 on the source and on the target respectively,
 to obtain the isomorphism $S_{2, m}\stackrel{\sim}{\to} T_{2, m}$ in the statement.
 \end{proof}
 For even $m$,
 we identify $S_{2, m}$ with $T_{2, m}$.

 \begin{Thm} \label{Thm:stabandaction}
 Let $\nu >0$ be an integer.
 Let $\Stab _{\nu}\subset G^0$ be the stabilizer of the affinoid $\calZ _{\nu}$.
 We denote by $\Frob _q$ the $q$-th power geometric Frobenius element.
 \begin{enumerate}[(1)]
 \item \label{item:stabinducesaction}
 We have $\Stab _{\nu}=(U_{\frakI}^{(\nu)}\times U_D^{(\nu)}\times \{ 1\})\cdot \calS$.
 The action of $\Stab _{\nu}$ on the affinoid $\calZ _{\nu}$ induces 
 actions 
 on the formal models $\scrZ _{\nu}$,  $\Spf \Cont (U_K^{\lceil \nu/n\rceil})$, $\scrX _{\nu}$, $\scrY _{\nu}$
 and
 on the six objects in the Cartesian diagrams (\ref{eq:cartdiag}),
 in which the morphisms are equivariant for the actions.
 \item \label{item:actiononcomps}
  In \ref{item:stabinducesaction}, $(g, d, \sigma) \in \Stab _{\nu}$ acts on $N_{L/K}U_L^{\nu}=U_K^{\lceil \nu /n\rceil}$ 
 as the composite of the translation by $N_G((g, d, \sigma))^{-1}\in U_K^{\lceil \nu /n\rceil}$ and $\Frob _q^{n_\sigma}$, where $n_{\sigma}=v(\Art _K^{-1}(\sigma))$.
 The action on ${N_{L/K}U_L^{\nu}}/{N_{L/K}U_L^{\nu +1}}$ is the action uniquely induced from that on $N_{L/K}U_L^{\nu}$.
 \item \label{item:actionofvarphi}
 The action of $\varphi _G=(\varphi , \varphi _D, 1)\in \calS$
 on $\bbA _{\overline{k}}^{n, \emph{perf}}$ is described as
 \[
 z^{q^{-l}}\mapsto z^{q^{-l}}, \quad y_1^{q^{-l}}\mapsto y_n^{q^{-l}}, \quad y_i^{q^{-l}}\mapsto y_{i-1}^{q^{-l}} \quad \emph{for $2\leq i\leq n$ and $l\geq 0$}.
 \]
 \item \label{item:actionofU_L}
 The action of $\Delta _{\xi}(U_L)\subset \calS$
 on $\bbA _{\overline{k}}^{n, \emph{perf}}$ is trivial.
 \item \label{item:actionofWeilgp}
 For $\sigma \in W_L$, set $a_{\sigma}=\Art _L^{-1}(\sigma )\in L\subset D$, $n_{\sigma}=v_L(a_{\sigma})=v(\Art _K^{-1}(\sigma ))$
 and $u_{\sigma}=a_{\sigma}\varphi _D^{-n_{\sigma}} \in U_D$.
 Then the action of $(1, a_{\sigma}^{-1}, \sigma )\in \calS$
 on $\bbA _{\overline{k}}^{n, \emph{perf}}$ is described as $\Frob _q^{n_{\sigma}}$ if $\nu$ is even, 
 and as the composite of $\Frob _q^{n_{\sigma}}$ and the automorphism
 \[
 z^{q^{-l}}\mapsto z^{q^{-l}}, \quad y_i^{q^{-l}}\mapsto \overline{u} _{\sigma}^{(q-1)/2}y_i^{q^{-l}} \quad \emph{for $1\leq i\leq n$ and $l\geq 0$}
 \]
 if $\nu$ is odd.
 \item \label{item:actionofGL} 
 The action of $U_{\frakI}^{(\nu)}=\Stab _{\nu}\cap \GL _n(K)$ on $\bbA _{\overline{k}}^{n, \emph{perf}}$
 factors through 
 $U_{\frakI}^{(\nu)}\rightarrow S_{1, \nu}$.
 If $\nu$ is odd, then the induced action of $x\in k=S_{1, \nu}$ is described as 
 \[
 z^{q^{-l}}\mapsto z^{q^{-l}}+x, \quad y_i^{q^{-l}}\mapsto y_i^{q^{-l}} \quad \emph{for $1\leq i\leq n$ and $l\geq 0$}.
 \]
 If $\nu =2\mu$ is even, then the induced action of $(v, (w_i))\in S_{1, \nu}$ is described as 
 \[
 z^{q^{-l}}\mapsto z^{q^{-l}}+v+\sum _{i\in {\bbZ}/{n\bbZ}}w_iy_{i-\mu}^{q^{-l}}, \quad y_i^{q^{-l}}\mapsto y_i^{q^{-l}}+w_i \quad \emph{for $i\in {\bbZ}/{n\bbZ}$ and $l\geq 0$},
 \]
 where we regard $\{ y_i^{q^{-l}}\}$ as indexed by ${\bbZ}/{n\bbZ}$.
 \item \label{item:actionofD} 
 The action of $U_D^{(\nu)}=\Stab _{\nu}\cap D^{\times}$ on $\bbA _{\overline{k}}^{n, \emph{perf}}$ 
 factors through 
 $U_D^{(\nu)}\rightarrow S_{2, \nu}$.
 If $\nu$ is odd, then the induced action of $x\in k=S_{2, \nu}$ is described as 
 \[
 z^{q^{-l}}\mapsto z^{q^{-l}}+x, \quad y_i^{q^{-l}}\mapsto y_i^{q^{-l}} \quad \emph{for $1\leq i\leq n$ and $l\geq 0$}.
 \]
 If $\nu =2\mu $ is even, then the induced action\footnote{Here, $\nu =rn+s$. Complicated values like $r-\nu +i-1-l$
                                                                               result from our choices of the normalization  
                                                                               \eqref{eq:zyandprime}
                                                                               and the definition 
                                                                               ``$w=\overline{(d^{-1}-1)\varphi _D^{-m/2}}$''
                                                                               (instead of $w=\overline{\varphi _D^{-m/2}(d^{-1}-1)}$)
                                                                               in Proposition \ref{Prop:gpS2} \ref{item:evenS2}.}
  of $(v, w)\in S_{2, \nu}$ is described as 
 \begin{align*}
 &z^{q^{-l}}\mapsto z^{q^{-l}}+v+\sum _{i\in {\bbZ}/{n\bbZ}}w^{q^{r-\nu +i-1-l}}y_{i}^{q^{-l}}, \\
 &y_i^{q^{-l}}\mapsto y_i^{q^{-l}}+w^{q^{r-\mu +i-1-l}} \quad \emph{for $i\in {\bbZ}/{n\bbZ}$ and $l\geq 0$},
 \end{align*}
 where we regard $\{ y_i^{q^{-l}}\}$ as indexed by ${\bbZ}/{n\bbZ}$. 
 \end{enumerate}
 \end{Thm}
 
 \begin{Rem} \label{Rem:Stab}
 \begin{enumerate}[(1)]
 \item \label{item:LandL'}
  The subgroups and elements appearing in
  \ref{item:actionofvarphi} to \ref{item:actionofD}
  do not generate the whole group $\Stab _{\nu}$
  unless $L'= L$.
  Later, we shall study the action of $\Stab _{\nu}$
  on the cohomology in an indirect way (see Remark \ref{Rem:Stabastorsor}, Theorem \ref{Thm:realizationinPi}).
 \item
 Although a similar computation applies,
 we omit the description of the action on $\overline{\scrY} _{\nu}$.
 \end{enumerate}
 \end{Rem}
 
 Since we have $N_G(\calS)=\{ 1\}$, we easily check that 
 $N_G$ maps $(U_{\frakI}^{(\nu)} \times U_D^{(\nu)} \times \{ 1\}) \cdot \calS$
 into $U_K^{\lceil \nu /n\rceil}$.
 Thus it immediately follows from \eqref{eq:actiononCont} 
 that $(g, d, \sigma)\in (U_{\frakI}^{(\nu)} \times U_D^{(\nu)} \times \{ 1\}) \cdot \calS$
 stabilizes 
 $U_K^{\lceil \nu /n\rceil} \subset U_K\simeq \LTpone$
 and 
 acts on the reduction $U_K^{\lceil \nu /n\rceil}=\Spec (U_K^{\lceil \nu /n\rceil}, \overline{k})$
 in the way stated in \ref{item:actiononcomps}.
 
 We denote the stabilizer of $\calX _{\nu}$ in $G^0$ by $\Stab _{\nu}'$.
 In the rest of this subsection we prove that
 $\Stab _{\nu}'=(U_{\frakI}^{(\nu)} \times U_D^{(\nu)} \times \{ 1\}) \cdot \calS$
 and the elements act on $\overline{\scrX} _{\nu}=\bbA _{\overline{k}}^{n, \text{perf}}$ in the way stated 
 in \ref{item:actionofvarphi} to \ref{item:actionofD};
 this is sufficient to conclude.
 The proof to follow is inspired by that in \cite[\S3]{ITepitame}.
 
 We begin by describing how a general element of $G^0$ acts on $\Nil ^{\flat, n}_{\calO _C}$.
 Take an element in $G^0$ and express it as $(g, d\varphi _D^{-n_{\sigma}}, \sigma)$ with $(g, d, 1)\in G^0$
 and $\sigma \in W_K$.
 Here $n_{\sigma}=v(\Art _{K}^{-1}(\sigma))$.
 We also write $g=(g_{i, j})_{1\leq i, j\leq n}\in \GL _n(K)$.
 By \eqref{eq:GLactiononLT}, \eqref{eq:DactiononLT}, \eqref{eq:WeilactiononLT}
 we have
 \begin{align}
 (g, d\varphi _D^{-n_{\sigma}}, \sigma)^{\ast}\bxi _i&=\sigma (\bxi _i), \label{eq:actiononxi} \\
 (g, d\varphi _D^{-n_{\sigma}}, \sigma)^{\ast}\bX _i&=(\Nil ^{\flat}_{H_0})\sum _{1\leq j\leq n}[g_{j, i}d^{-1}]_{\Nil ^{\flat}_{H_0}} (\bX _j) \nonumber
 \end{align}
 for $1\leq i\leq n$ in $\Nil ^{\flat}(B_n)$.
 Using these we compute
 \begin{align}
 (g, d\varphi _D^{-n_\sigma}, \sigma )^{\ast}\bY _i&=(g, d\varphi _D^{-n_{\sigma}}, \sigma )^{\ast} (\bX _i-_{\Nil ^{\flat}_{H_0}} \bxi _i) \nonumber \\
                                                                &\stackrel{(\ast)}{=}(g, d\varphi _D^{-n_{\sigma}}, \sigma )^{\ast} \bX _i-_{\Nil ^{\flat}_{H_0}} (g, d\varphi _D^{-n_{\sigma}}, \sigma )^{\ast} \bxi _i \nonumber \\
                                                                &=(\Nil ^{\flat}_{H_0})
                                                                   \sum _{1\leq j\leq n}[g_{j, i}d^{-1}]_{\Nil ^{\flat}_{H_0}}(\bY _j+_{\Nil ^{\flat}_{H_0}} \bxi _j) -_{\Nil ^{\flat}_{H_0}} \sigma (\bxi _i) \nonumber \\
                                                                &\stackrel{(\ast \ast)}{=}(\Nil ^{\flat}_{H_0})
                                                                   \sum _{1\leq j\leq n}[g_{j, i}d^{-1}]_{\Nil ^{\flat}_{H_0}}(\bY _j) \nonumber \\
                                                                   &\quad \quad +_{\Nil ^{\flat}_{H_0}} 
                                                                   (\Nil ^{\flat}_{H_0})
                                                                   \sum _{1\leq j\leq n}[g_{j, i}d^{-1}]_{\Nil ^{\flat}_{H_0}}(\bxi _j) 
                                                                   -_{\Nil ^{\flat}_{H_0}} \sigma (\bxi _i) \nonumber \\
                                                                   &=(\Nil ^{\flat}_{H_0})
                                                                   \sum _{1\leq j\leq n}[g_{j, i}d^{-1}]_{\Nil ^{\flat}_{H_0}}(\bY _j) \nonumber \\
                                                                   &\quad \quad +_{\Nil ^{\flat}_{H_0}} 
                                                                   \left[ \sum _{1\leq j\leq n}g_{j, i}d^{-1}\varphi _D^{i-j}\right] _{\Nil ^{\flat}_{H_0}}(\bxi _i) \label{eq:actiononYi} 
                                                                   -_{\Nil ^{\flat}_{H_0}} \sigma (\bxi _i). 
 \end{align}
 For $(\ast)$, note that $(g, d\varphi _D^{-n_\sigma}, \sigma )^{\ast}$ is a continuous $\calO _K$-automorphism of $B_n$
 while if we write $\bX -_{\Nil ^{\flat}_{H_0}} \bY=(f^{q^{-l}})_{l\geq 0}$ with $f^{q^{-l}}\in \calO _C[[X^{q^{-\infty}}, Y^{q^{-\infty}}]]$
 then $f^{q^{-l}}$ lies in fact in $\calO _K[[X^{q^{-\infty}}, Y^{q^{-\infty}}]]$, for $H$ is defined over $\calO _K$.
 On the other hand, $(\ast \ast)$ holds since the action of $D^{\times}$ on $\Nil ^{\flat}_{\calO _C}(R)$ respects the $K$-vector space structure for any $R\in \Adic _{\calO _C}$.
 Similarly, we have
 \begin{align}
 (g, d\varphi _D^{-n_\sigma}, \sigma )^{\ast}\bZ&=(g, d\varphi _D^{-n_\sigma}, \sigma )^{\ast}\left( (\Nil ^{\flat}_{H_0})\sum _{1\leq i\leq n} [\varphi _D^{i-1}]_{\Nil ^{\flat}_{H_0}} (\bY _i) \right) \nonumber \\
                                                             &=(\Nil ^{\flat}_{H_0})
                                                                   \sum _{1\leq i\leq n} [\varphi _D^{i-1}]_{\Nil ^{\flat}_{H_0}} ((g, d\varphi _D^{-n_\sigma}, \sigma )^{\ast}\bY _i) \nonumber \\
                                                             &=(\Nil ^{\flat}_{H_0})
                                                                   \sum _{1\leq i, j\leq n} [\varphi _D^{i-1}g_{j, i}d^{-1}]_{\Nil ^{\flat}_{H_0}} (\bY _j) \nonumber \\
                                                             &\qquad +_{\Nil ^{\flat}_{H_0}} (\Nil ^{\flat}_{H_0})
                                                                   \sum _{1\leq i\leq n} \left[ \varphi _D^{i-1}\sum _{1\leq j\leq n}g_{j, i}d^{-1}\varphi _D^{i-j}\right] _{\Nil ^{\flat}_{H_0}} (\bxi _i) \nonumber \\
                                                             &\qquad \qquad  -_{\Nil ^{\flat}_{H_0}} (\Nil ^{\flat}_{H_0})\sum _{1\leq i\leq n}[\varphi _D^{i-1}]_{\Nil ^{\flat}_{H_0}} (\sigma (\bxi _i)) \nonumber \\
                                                             &=(\Nil ^{\flat}_{H_0})
                                                                   \sum _{1\leq i, j\leq n} [\varphi _D^{i-1}g_{j, i}d^{-1}]_{\Nil ^{\flat}_{H_0}} (\bY _j) \label{eq:actiononZ} \\
                                                             &\qquad +_{\Nil ^{\flat}_{H_0}} 
                                                                   \left[ \sum _{1\leq i, j\leq n}g_{j, i}(\varphi _D^{i-1}d^{-1}\varphi _D^{-(i-1)})\varphi _D^{i-j}\right] _{\Nil ^{\flat}_{H_0}} (\bxi _1) \nonumber \\
                                                             &\qquad \qquad -_{\Nil ^{\flat}_{H_0}} (\Nil ^{\flat}_{H_0})\sum _{1\leq i\leq n}[\varphi _D^{i-1}]_{\Nil ^{\flat}_{H_0}} (\sigma (\bxi _i)). \nonumber
 \end{align}
  
 \paragraph{Action of $U_{\frakI}^{(\nu)}$}
 Let us prove that $\GL _n(K)\cap \Stab _{\nu}'=U_{\frakI}^{(\nu)}$
 and that $U_{\frakI}^{(\nu)}$ acts on $\bbA _{\overline{k}}^{n, \text{perf}}$ as in the assertion \ref{item:actionofGL}.
 For $g=(g_{i, j})_{1\leq i, j\leq n}=1+(a_{i, j})_{1\leq i, j\leq n}\in \GL _n(K)\cap G^0$,
 the equation \eqref{eq:actiononYi} reads
   \begin{align} 
   g^{\ast}\bY _i  
    &=(\Nil ^{\flat}_{H_0})\sum _{1\leq j\leq n} [g_{j, i}]_{\Nil ^{\flat}_{H_0}}(\bY _j)+_{\Nil ^{\flat}_{H_0}} \left[ \sum _{1\leq j\leq n}g_{j, i}\varphi _D^{i-j}\right]_{\Nil ^{\flat}_{H_0}}(\bxi _i) 
   -_{\Nil ^{\flat}_{H_0}} \bxi _i \nonumber \\
    &=(\Nil ^{\flat}_{H_0})\sum _{1\leq j\leq n} [g_{j, i}]_{\Nil ^{\flat}_{H_0}}(\bY _j)+_{\Nil ^{\flat}_{H_0}} \left[ \sum _{1\leq j\leq n} a_{j, i}\varphi _D^{i-j}\right] _{\Nil ^{\flat}_{H_0}}(\bxi _i)  \label{eq:GLY}, 
   \end{align}
   and thus
   \begin{align} 
   g^{\ast}\bZ 
    &=(\Nil ^{\flat}_{H_0})\sum _{1\leq i\leq n} [\varphi _D^{i-1}]_{\Nil ^{\flat}_{H_0}} (g^{\ast}\bY _i) \nonumber \\
    &=(\Nil ^{\flat}_{H_0})\sum _{1\leq i, j\leq n} [\varphi _D^{i-1}g_{j, i}]_{\Nil ^{\flat}_{H_0}} (\bY _j) \nonumber \\
    &\qquad +_{\Nil ^{\flat}_{H_0}} (\Nil ^{\flat}_{H_0})\sum _{1\leq i\leq n} \left[ \varphi _D^{i-1}(\sum _{1\leq j\leq n}a_{j, i}\varphi _D^{i-j})\right] _{\Nil ^{\flat}_{H_0}} (\bxi _i) \nonumber \\
    &=(\Nil ^{\flat}_{H_0})\sum _{1\leq i, j\leq n} [\varphi _D^{i-1}g_{j, i}]_{\Nil ^{\flat}_{H_0}} (\bY _j)
       +_{\Nil ^{\flat}_{H_0}} \left[ \sum _{1\leq i, j\leq n}a_{j, i}\varphi _D^{i-j}\right] _{\Nil ^{\flat}_{H_0}} (\bxi _1)
    \label{eq:GLR}.
   \end{align}
 We write 
 \begin{equation} \label{eq:cnu}
 c_{\nu}=\begin{cases}
 q^{\mu}(q+1)/2 &\text{ if $\nu =2\mu +1$ is odd} \\
 q^{\mu} &\text{ if $\nu =2\mu$ is even}.
 \end{cases}
 \end{equation}
 
 Assume that $g$ stabilizes $\calX _{\nu}$.
 We consider the condition $\xi ^{(g, 1, 1)} \in \calX _{\nu}$ (here and in the rest of the proof, we simply write $\xi$ for the image of $\xi \in \calZ _{\nu}$ in $\calX _{\nu}$).
 By (\ref{eq:GLY}),
 we must have
 \[
 \lv 
 \left[ \sum _{1\leq j\leq n} a_{j, i}\varphi _D^{i-j}\right] _{\Nil ^{\flat}_{H_0}}(\bxi _i)
 \rv
 \leq \lv \bxi _i \rv ^{c_{\nu}}
 \]
 for all $1\leq i\leq n$.
 By Lemma \ref{lem:appD} \ref{htnnilaction}, this is equivalent to
 \[
 v_D\left( \sum _{1\leq j\leq n} a_{j, i}\varphi _D^{i-j}\right) \geq \lfloor (\nu +1)/2\rfloor,
 \]
 for all $1\leq i\leq n$, 
 which in turn is equivalent to $g\in U_{\frakI}^{\lfloor (\nu +1)/2\rfloor}$ by Proposition \ref{Prop:GL(nu)explicit} \ref{UforGL}.
 
 Similarly,
 by (\ref{eq:GLR}),
 we must have
 \[
 \lv
 \left[ \sum _{1\leq i, j\leq n}a_{j, i}\varphi _D^{i-j}\right] _{\Nil ^{\flat}_{H_0}} (\bxi _1)
 \rv
   \leq \lv \bxi _1\rv ^{q^{\nu}}.
 \]
 Again by Lemma \ref{lem:appD} \ref{htnnilaction}, this is equivalent to
 \[
 v_D\left( \sum _{1\leq i, j\leq n}a_{j, i}\varphi _D^{i-j}\right) \geq \nu.
 \]
 Noting that $\sum _{1\leq i, j\leq n}a_{j, i}\varphi _D^{i-j}=np_1(g-1)$,
 we have $\Stab _{\nu}' \cap \GL _n(K)\subset U_{\frakI}^{(\nu )}$ 
 by Proposition \ref{Prop:GL(nu)explicit} \ref{UforGL}.
 
 Conversely,
 let $g=(g_{i, j})_{1\leq i, j\leq n}=1+(a_{i, j})_{1\leq i, j\leq n}\in U_{\frakI}^{(\nu )}$.
 We have, 
 by (\ref{eq:GLY}) and \eqref{eq:GLR}, 
 \begin{align*}
   g^{\ast}\bY _i 
     &=\bY _i+_{\Nil ^{\flat}_{H_0}}(\Nil ^{\flat}_{H_0})\sum _{1\leq j\leq n} [a_{j, i}]_{\Nil ^{\flat}_{H_0}}(\bY _j)+_{\Nil ^{\flat}_{H_0}} \left[ \sum _{1\leq j\leq n} a_{j, i}\varphi _D^{i-j}\right] _{\Nil ^{\flat}_{H_0}}(\bxi _i) \\
   g^{\ast}\bZ
     &=\bZ 
      +_{\Nil ^{\flat}_{H_0}} (\Nil ^{\flat}_{H_0})\sum _{1\leq j\leq n} \left[ \sum _{1\leq i\leq n}\varphi _D^{i-1}a_{j, i}\right] _{\Nil ^{\flat}_{H_0}} (\bY _j) \\
       &\qquad +_{\Nil ^{\flat}_{H_0}} \left[ \sum _{1\leq i, j\leq n}a_{j, i}\varphi _D^{i-j}\right] _{\Nil ^{\flat}_{H_0}} (\bxi _1) .
 \end{align*}
 To estimate and approximate these functions we work in a term-by-term manner.
 
 Let $1\leq i\leq n$.
 We have by Lemma \ref{lem:appD} \ref{htnnilaction} 
 and the condition for $g\in U_{\frakI} ^{\lfloor (\nu+1)/2\rfloor}$
 in Proposition \ref{Prop:GL(nu)explicit} \ref{UforGL},
 \begin{equation} \label{eq:nbdaction1}
 \lvert [a_{j, i}]_{\Nil ^{\flat}_{H_0}}(\bY _j)(\tau) \rvert 
 \leq \lvert \bxi _j \rvert ^{c_{\nu}q^{v_D(a_{j, i})}} 
 \leq \lvert \bxi _i \rvert ^{c_{\nu}q^{\lfloor (\nu +1)/2\rfloor}} 
 < \lvert \bxi _i \rvert ^{c_{\nu}}
 \end{equation}
 for $1\leq j\leq n$ and
 \begin{align*}
 \lv \left[ \sum _{1\leq j\leq n} a_{j, i}\varphi _D^{i-j}\right] _{\Nil ^{\flat}_{H_0}}(\bxi _i) \rv 
 \leq \lvert \bxi _i \rvert ^{q^{\lfloor (\nu +1)/2\rfloor}}
 < \lv \bxi _i \rv ^{c_{\nu}} \quad &\text{ if $\nu$ is odd}, \\
 \lv \left[ \sum _{1\leq j\leq n} a_{j, i}\varphi _D^{i-j}\right] _{\Nil ^{\flat}_{H_0}}(\bxi _i) -\zeta _i\bxi _i ^{q^{\lfloor (\nu +1)/2\rfloor}}\rv 
 < \lvert \bxi _i \rvert ^{q^{\lfloor (\nu +1)/2\rfloor}}
 = \lvert \bxi _i \rvert ^{c_{\nu}} \quad &\text{ if $\nu$ is even},
 \end{align*}
 where $\zeta _i \in \mu _{q-1}(K)\cup \{0\}$ is the element satisfying
 $\overline{\zeta} _i=\overline{(\sum _{1\leq j\leq n}a_{j, i}\varphi _D^{i-j})\varphi _D^{-\lfloor (\nu +1)/2\rfloor}}$.
 Therefore, we have by Lemma \ref{lem:app2} 
 \begin{align}
 \lvert (g^{\ast} \bY _i -\bY _i)(\tau) \rvert < \lvert \bxi _i \rvert ^{c_{\nu}} \quad &\text{ if $\nu$ is odd}, \label{eq:GLappr1} \\
 \lvert (g^{\ast} \bY _i -(\bY _i+\zeta _i\bxi _i^{q^{\mu}}))(\tau) \rvert < \lvert \bxi _i \rvert ^{c_{\nu}} \quad &\text{ if $\nu =2\mu$ is even}. \label{eq:GLappr2}
 \end{align}
 
 We turn to $g^{\ast}\bZ$.
 We have by Proposition \ref{Prop:GL(nu)explicit} \ref{UforGL}
 \[
 v_D\left( \sum _{1\leq i\leq n}\varphi _D^{i-1}a_{j, i}\right) \geq \lfloor (\nu +1)/2\rfloor +j-1
 \]
 for each $1\leq j\leq n$.
 Moreover, we have for each $1\leq i\leq n$
 \[
 \overline{a_{j, i}\varphi _D^{i-j}\cdot \varphi _D^{-\lfloor (\nu +1)/2\rfloor}}=
 \begin{cases}
 \overline{\zeta} _i &\text{if $i-j\equiv \lfloor (\nu +1)/2\rfloor \pmod{n}$} \\
 0 &\text{otherwise.}
 \end{cases}
 \]
 Thus we infer that for $1\leq j\leq n$
 \[
 \overline{(\sum _{1\leq i\leq n}\varphi _D^{i-1}a_{j, i})\varphi _D^{-(\lfloor (\nu +1)/2\rfloor +j-1)}}=\overline{\zeta} _{j+\lfloor (\nu +1)/2\rfloor}.
 \]
 Here $\zeta _{j+\lfloor (\nu +1)/2\rfloor}$ is understood to be $\zeta _{j'}$ with $j'\equiv j+\lfloor (\nu +1)/2\rfloor \pmod{n}$.
 Hence, by Lemma \ref{lem:appD} \ref{htnnilaction}
 \begin{align*}
 &\lv \left[ \sum _{1\leq i\leq n}\varphi _D^{i-1}a_{j, i}\right] _{\Nil ^{\flat}_{H_0}} (\bY _j) (\tau)\rv && \\
 &\leq \lvert \bxi _j \rvert ^{c_{\nu}q^{v_D(\sum _{1\leq i\leq n}\varphi _D^{i-1}a_{j, i})}}
 \leq \lvert \bxi _1 \rvert ^{c_{\nu}q^{\lfloor (\nu +1)/2\rfloor}} 
 < \lvert \bxi _1 \rvert ^{q^{\nu}} \quad &&\text{ if $\nu$ is odd}, \\
 &\lv (\left[ \sum _{1\leq i\leq n}\varphi _D^{i-1}a_{j, i}\right] _{\Nil ^{\flat}_{H_0}} (\bY _j) -\zeta _{j+\mu} \bY _j^{q^{\mu +j-1}})(\tau)\rv &&\\
 &< \lvert \bxi _j \rvert ^{c_{\nu}q^{v_D(\sum _{1\leq i\leq n}\varphi _D^{i-1}a_{j, i})}}
 \leq \lvert \bxi _1 \rvert ^{c_{\nu}q^{\lfloor (\nu +1)/2\rfloor}} 
 = \lvert \bxi _1 \rvert ^{q^{\nu}} \quad &&\text{ if $\nu =2\mu$ is even}
 \end{align*}
 for $1\leq j\leq n$.
 By Lemma \ref{lem:appD} \ref{htnnilaction} and the condition for $g\in U_{\frakI}^{(\nu)}$ in Proposition \ref{Prop:GL(nu)explicit} \ref{UforGL},
 \[
 \lv \left[ \sum _{1\leq i, j\leq n}a_{j, i}\varphi _D^{i-j}\right] _{\Nil ^{\flat}_{H_0}} (\bxi _1)-\zeta \bxi _1^{q^{\nu}}\rv 
 < \lvert \bxi _1 \rvert ^{q^{\nu}},
 \]
 where $\zeta \in \mu _{q-1}(K)\cup \{0\}$ is the unique element satisfying 
 \[
 \overline{\zeta}=\overline{(\sum _{1\leq i, j\leq n}a_{j, i}\varphi _D^{i-j})\varphi _D^{-\nu}}
 =\overline{np_1(g-1)\varphi _D^{-\nu}}=\overline{\tr ((g-1)\varphi _D^{-\nu})}.
 \]
 Therefore we have by Lemma \ref{lem:app2}
 \begin{align}
 \lvert (g^{\ast}\bZ -(\bZ +\zeta \bxi _1^{q^{\nu}}))(\tau)\rvert <\lvert \bxi _1 \rvert ^{q^{\nu}} \quad &\text{ if $\nu$ is odd}, \label{eq:GLappr3} \\
 \lvert (g^{\ast}\bZ -(\bZ +\sum _{1\leq j\leq n}\zeta _{j+\mu} \bY _j^{q^{\mu +j-1}}+\zeta \bxi _1^{q^{\nu}}))(\tau) \rvert <\lvert \bxi _1 \rvert ^{q^{\nu}} \quad &\text{ if $\nu =2\mu$ is even}. \label{eq:GLappr4}
 \end{align}
 
 Now \eqref{eq:GLappr1}, \eqref{eq:GLappr2}, \eqref{eq:GLappr3}, \eqref{eq:GLappr4}
 show that $g$ indeed lies in $\Stab _{\nu}'$ and 
 moreover, noting \eqref{eq:capitaltoprime} and \eqref{eq:zyandprime}, one applies Lemma \ref{lem:esttoapp} to deduce that it acts on $\bbA _{\overline{k}}^{n, \text{perf}}$ as in the assertion \ref{item:actionofGL}.
 
 \paragraph{Action of $U_D^{(\nu )}$}
 Let us prove that $D^{\times} \cap \Stab _{\nu}'=U_D^{(\nu)}$
 and that $U_D^{(\nu)}$ acts on $\bbA _{\overline{k}}^{n, \text{perf}}$ as in the assertion \ref{item:actionofD}.
 The argument is analogous to the above.
 Let $d\in D^{\times}\cap G^0$.
 By
 \eqref{eq:actiononYi}, \eqref{eq:actiononZ}
 we have
 \begin{align}
 d^{\ast}\bY _i 
   &=[d^{-1}]_{\Nil ^{\flat}_{H_0}} (\bY _i) +_{\Nil ^{\flat}_{H_0}} [d^{-1}-1]_{\Nil ^{\flat}_{H_0}}(\bxi _i), \label{eq:DY}\\
  d^{\ast}\bZ 
   &=(\Nil ^{\flat}_{H_0})\sum _{1\leq i\leq n} [\varphi _D^{i-1}]_{\Nil ^{\flat}_{H_0}} (d^{\ast}\bY _i) \nonumber \\
   &=(\Nil ^{\flat}_{H_0})\sum _{1\leq i\leq n} [\varphi _D^{i-1}d^{-1}]_{\Nil ^{\flat}_{H_0}} (\bY _i) \nonumber \\
      &\qquad +_{\Nil ^{\flat}_{H_0}} \left[ \sum _{1\leq i\leq n}\varphi _D^{i-1}(d^{-1}-1)\varphi _D^{-(i-1)}\right] _{\Nil ^{\flat}_{H_0}} (\bxi _1) . \label{eq:DR}
 \end{align}
 Let $c_{\nu}$ be as in \eqref{eq:cnu}.
 
 Assume that $d$ stabilizes $\calX _{\nu}$.
 By the condition
 $\xi ^{(1, d, 1)}\in \calX _{\nu}$,
 we necessarily have
 \begin{align*}
 &\lv
 [d^{-1}-1]_{\Nil ^{\flat}_{H_0}}(\bxi _i)
 \rv
 \leq \lv \bxi _i \rv ^{c_{\nu}} \\
 \intertext{for $1\leq i\leq n$ and}
 &\lv
 \left[ \sum _{1\leq i\leq n}\varphi _D^{i-1}(d^{-1}-1)\varphi _D^{-(i-1)}\right] _{\Nil ^{\flat}_{H_0}} (\bxi _1)
 \rv
 \leq \lv \bxi _1 \rv ^{q^{\nu}}.
 \end{align*}
 In view of $\sum _{1\leq i\leq n}\varphi _D^{i-1}a\varphi _D^{-(i-1)}=np_2(a)$ for $a\in D$ (which can be verified by using the realization in Remark \ref{Rem:D(nu)explicit}), 
 these conditions are seen to be equivalent to 
 $d^{-1}\in U_D^{(\nu )}$, i.e.\ $d\in U_D^{(\nu )}$, 
 by Lemma \ref{lem:appD} \ref{htnnilaction} and Remark \ref{Rem:D(nu)explicit} \ref{UforD}.
 
 Conversely,
 let $d\in U_D^{(\nu )}$.
 We have, by (\ref{eq:DY}) \eqref{eq:DR},
 \begin{align*}
 d^{\ast}\bY _i 
   &=\bY _i+_{\Nil ^{\flat}_{H_0}}[d^{-1}-1]_{\Nil ^{\flat}_{H_0}} (\bY _i)+_{\Nil ^{\flat}_{H_0}} [d^{-1}-1]_{\Nil ^{\flat}_{H_0}}(\bxi _i), \\
  d^{\ast}\bZ 
   &=\bZ +_{\Nil ^{\flat}_{H_0}} (\Nil ^{\flat}_{H_0})\sum _{1\leq i\leq n} [\varphi _D^{i-1}(d^{-1}-1)]_{\Nil ^{\flat}_{H_0}} (\bY _i) \\
   &\qquad +_{\Nil ^{\flat}_{H_0}} \left[ \sum _{1\leq i\leq n}\varphi _D^{i-1}(d^{-1}-1)\varphi _D^{-(i-1)}\right] _{\Nil ^{\flat}_{H_0}} (\bxi _1) . 
 \end{align*}
 
 Let $1\leq i\leq n$.
 By Lemma \ref{lem:appD} \ref{htnnilaction}
 and $d\in U_D^{\lfloor (\nu+1)/2\rfloor}$
 we have
 \begin{align}
 \lvert [d^{-1}-1]_{\Nil ^{\flat}_{H_0}} (\bY _i)(\tau)\rvert
 \leq \lvert \bY _i(\tau)\rvert ^{q^{\lfloor (\nu+1)/2\rfloor}}
 <  \lvert \bY _i(\tau)\rvert 
 \leq \lvert \bxi _i\rvert ^{c_{\nu}}, \quad & \label{eq:nbdaction2} \\
 \lvert [d^{-1}-1]_{\Nil ^{\flat}_{H_0}}(\bxi _i) \rvert
 \leq \lvert \bxi _i\rvert ^{q^{\lfloor (\nu+1)/2\rfloor}}
 <\lvert \bxi _i\rvert ^{c_{\nu}} \quad &\text{ if $\nu$ is odd}, \nonumber \\
 \lvert [d^{-1}-1]_{\Nil ^{\flat}_{H_0}}(\bxi _i) -\bzeta \bxi _i^{q^{\lfloor (\nu+1)/2\rfloor}}\rvert
 < \lvert \bxi _i\rvert ^{q^{\lfloor (\nu+1)/2\rfloor}}
 =\lvert \bxi _i\rvert ^{c_{\nu}} \quad &\text{ if $\nu$ is even}, \nonumber
 \end{align}
 for $1\leq i\leq n$,
 where $\zeta \in \mu _{q^n-1}(K_n)\cup \{0\}$ is now the unique element satisfying
 $\overline{\zeta}=\overline{(d^{-1}-1)\varphi _D^{-\lfloor (\nu +1)/2\rfloor}}$
 and $\bzeta =(\zeta ^{q^{-l}})_{l\geq 0}\in \Nil ^{\flat}(\calO _{\widehat{K} ^{\text{ur}}})$
 is the element such that $\zeta ^{q^{-l}}\in \mu _{q^n-1}(K_n)\cup \{0\}$ for all $l\geq 0$.
 As before we have by Lemma \ref{lem:app2}
 \begin{align}
 \lvert (d^{\ast}\bY _i-\bY _i) (\tau) \rvert
 <\lvert \bxi _i\rvert ^{c_{\nu}} \quad &\text{ if $\nu$ is odd}, \label{eq:Dappr1} \\
 \lvert (d^{\ast}\bY _i-(\bY _i+ \bzeta \bxi _i^{q^{\mu}})) (\tau) \rvert
 <\lvert \bxi _i\rvert ^{c_{\nu}} \quad &\text{ if $\nu =2\mu$ is even}. \label{eq:Dappr2}
 \end{align}
 
 Similarly, by Lemma \ref{lem:appD} \ref{htnnilaction}
 and $d\in U_D^{\lfloor (\nu+1)/2\rfloor}$
 \begin{align*}
 &\lvert [\varphi _D^{i-1}(d^{-1}-1)]_{\Nil ^{\flat}_{H_0}} (\bY _i)(\tau)\rvert && \\
 &\leq \lvert \bxi _i\rvert ^{c_{\nu}q^{\lfloor (\nu+1)/2\rfloor+i-1}}
 =\lvert \bxi _1\rvert ^{c_{\nu}q^{\lfloor (\nu+1)/2\rfloor}}
 <\lvert \bxi _1\rvert ^{q^{\nu}} \quad &&\text{ if $\nu$ is odd}, \\
 &\lvert ([\varphi _D^{i-1}(d^{-1}-1)]_{\Nil ^{\flat}_{H_0}} (\bY _i)-\bzeta ^{q^{i-1}}\bY _i^{q^{\lfloor (\nu+1)/2\rfloor+i-1}})(\tau)\rvert && \\
 &< \lvert \bxi _i\rvert ^{c_{\nu}q^{\lfloor (\nu+1)/2\rfloor+i-1}}
 =\lvert \bxi _1\rvert ^{c_{\nu}q^{\lfloor (\nu+1)/2\rfloor}}
 =\lvert \bxi _1\rvert ^{q^{\nu}} \quad &&\text{ if $\nu$ is even}
 \end{align*}
 for $1\leq i\leq n$.
 By Lemma \ref{lem:appD} \ref{htnnilaction}
 and $d\in U_D^{(\nu)}$
 \[
 \lv
 \left[ \sum _{1\leq i\leq n}\varphi _D^{i-1}(d^{-1}-1)\varphi _D^{-(i-1)}\right] _{\Nil ^{\flat}_{H_0}} (\bxi _1)-\zeta' \bxi _1^{q^{\nu}}
 \rv
 < \lv \bxi _1 \rv ^{q^{\nu}},
 \]
 where $\zeta' \in \mu _{q-1}(K)\cup \{0\}$ is the element satisfying 
 \[
 \overline{\zeta'}=\overline{(\sum _{1\leq i\leq n}\varphi _D^{i-1}(d^{-1}-1)\varphi _D^{-(i-1)})\varphi _D^{-\nu}}
 =\overline{np_2(d^{-1}-1)\varphi _D^{-\nu}}=\overline{\Trd ((d^{-1}-1)\varphi _D^{-\nu})}.
 \]
 Hence we have by Lemma \ref{lem:app2}
 \begin{align}
 \lvert (d^{\ast}\bZ -(\bZ +\zeta' \bxi _1^{q^{\nu}}))(\tau)\rvert
 < \lvert \bxi _1\rvert ^{q^{\nu}}  \quad &\text{ if $\nu$ is odd}, \label{eq:Dappr3} \\
 \lvert (d^{\ast}\bZ -(\bZ +\sum _{1\leq i\leq n}\bzeta ^{q^{i-1}}\bY _i^{q^{\mu +i-1}}+\zeta' \bxi _1^{q^{\nu}}))(\tau)\rvert
 < \lvert \bxi _1\rvert ^{q^{\nu}}  \quad &\text{ if $\nu =2\mu$ is even}. \label{eq:Dappr4}
 \end{align}
 
 Now \eqref{eq:Dappr1}, \eqref{eq:Dappr2}, \eqref{eq:Dappr3}, \eqref{eq:Dappr4} show that $d$ indeed lies in $\Stab _{\nu}'$
 and moreover, noting \eqref{eq:capitaltoprime} and \eqref{eq:zyandprime}, one applies Lemma \ref{lem:esttoapp} to deduce that it acts on $\bbA _{\overline{k}}^{n, \text{perf}}$ as in the assertion \ref{item:actionofD}.
 
 \paragraph{Action of $\varphi _G$}
 Let us prove that $\varphi _G\in \calS$ stabilizes $\calX _{\nu}$ 
 and induces the stated action on $\bbA _{\overline{k}}^{n, \text{perf}}$.
 If we write $\varphi=(g_{i, j})_{1\leq i, j\leq n}$
 then \eqref{eq:actiononYi}
 reduces to
 \begin{equation*}
 \varphi _G^{\ast}\bY _i=\begin{cases}
                              \bY _n^{q^{n-1}} +_{\Nil ^{\flat}_{H_0}} 
                                                                   \left[ \sum _{1\leq j\leq n}g_{j, i}\varphi _D^{i-j-1}\right] _{\Nil ^{\flat}_{H_0}}(\bxi _i) -_{\Nil ^{\flat}_{H_0}} \bxi _i &\text{if $i=1$} \\
                              \bY _{i-1}^{q^{-1}} +_{\Nil ^{\flat}_{H_0}} 
                                                                   \left[ \sum _{1\leq j\leq n}g_{j, i}\varphi _D^{i-j-1}\right] _{\Nil ^{\flat}_{H_0}}(\bxi _i) -_{\Nil ^{\flat}_{H_0}} \bxi _i &\text{otherwise}.
                             \end{cases}
 \end{equation*}
 Now 
 we see that $\left[ \sum _{1\leq j\leq n}g_{j, i}\varphi _D^{i-j-1}\right] _{\Nil ^{\flat}_{H_0}}(\bxi _i) =\bxi _i$
 in view of $\varphi _G\in \calS$ or by a simple computation,
 and hence
 \begin{equation*}
 \varphi _G^{\ast}\bY _i=\begin{cases}
                              \bY _n^{q^{n-1}}&\text{if $i=1$} \\
                              \bY _{i-1}^{q^{-1}}&\text{otherwise},
                             \end{cases}
                            \qquad \varphi _G^{\ast}\bZ =(\Nil ^{\flat}_{H_0})\sum _{1\leq i\leq n}[\varphi _D^{i-1}]_{\Nil ^{\flat}_{H_0}}(\varphi _G^{\ast}\bY _i)=\bZ.
 \end{equation*}
 From this we find that $\varphi _G$ stabilizes $\calX _{\nu}$ 
 and acts on $\bbA _{\overline{k}}^{n, \text{perf}}$ in a manner stated in \ref{item:actionofvarphi}.
 
 \paragraph{Action of $\Delta _{\xi}(U_L)$}
 Let $u\in U_L$.
 We claim that
 $\Delta _{\xi}(u)$ lies in $\Stab _{\nu}'$ 
 and that it acts trivially on $\bbA _{\overline{k}}^{n, \text{perf}}$
 as stated in \ref{item:actionofU_L}.
 We may assume that $u$ lies in $U_L^1$
 since $\Delta _{\xi}(K^{\times})$ acts trivially on $\LTp$.
 We simply write $u$ for $i_{\xi}^D(u)$
 and express $i_{\xi}(u)=(g_{i, j})_{1\leq i, j\leq n}\in \GL _n(K)$. 
 We have 
 by \eqref{eq:actiononYi}
 \begin{align*}
 \Delta _{\xi}(u)^{\ast}\bY _i&=(\Nil ^{\flat}_{H_0})\sum _{1\leq j\leq n} [u^{-1}g_{j, i}]_{\Nil ^{\flat}_{H_0}}(\bY _j) \\
 &\qquad +_{\Nil ^{\flat}_{H_0}} (\Nil ^{\flat}_{H_0})\sum _{1\leq j\leq n} [u^{-1}g_{j, i}]_{\Nil ^{\flat}_{H_0}}(\bxi _j)-_{\Nil ^{\flat}_{H_0}} \bxi _i.
 \end{align*}
 Since $\Delta _{\xi}(u)\in \calS$, we see that 
 $(\Nil ^{\flat}_{H_0})\sum _{1\leq j\leq n} [u^{-1}g_{j, i}]_{\Nil ^{\flat}_{H_0}}(\bxi _j)= \bxi _i$.
 Therefore, arguing as in \eqref{eq:nbdaction1} and \eqref{eq:nbdaction2}
 we deduce
 \begin{align*}
 \Delta _{\xi}(u)^{\ast}\bY _i&=(\Nil ^{\flat}_{H_0})\sum _{1\leq j\leq n} [u^{-1}g_{j, i}]_{\Nil ^{\flat}_{H_0}}(\bY _j) \\
                                      &\equiv _1(\Nil ^{\flat}_{H_0})\sum _{1\leq j\leq n} [g_{j, i}]_{\Nil ^{\flat}_{H_0}}(\bY _j) \\
                                      &\equiv _1\bY _i
 \end{align*}
 by Lemma \ref{lem:appD} \ref{htnnilaction} and $u\in U_L^1$.
 For the last line, note that if $(g_{i, j})_{1\leq i, j\leq n}=1+(a_{i, j})_{1\leq i, j\leq n}$,
 then $\sum _{1\leq j\leq n}a_{j, i}\varphi _D^{i-j}=u-1\in \frakp _L$
 and thus $v_D(a_{j, i}\varphi _D^{i-j})>0$ for $1\leq j\leq n$.
 
 Similarly we have $\sum _{1\leq i\leq n}\varphi _D^{i-j}g_{j, i}=u$ for $1\leq j\leq n$
 and thus
 \begin{align*}
    \Delta _{\xi}(u)^{\ast}\bZ&=(\Nil ^{\flat}_{H_0})\sum _{1\leq i\leq n} [\varphi _D^{i-1}]_{\Nil ^{\flat}_{H_0}} (\Delta _{\xi}(u)^{\ast}\bY _i) \\
                                      &=(\Nil ^{\flat}_{H_0})\sum _{1\leq i, j\leq n} [\varphi _D^{i-1}u^{-1}g_{j, i}]_{\Nil ^{\flat}_{H_0}}(\bY _j) \\
                                      &=(\Nil ^{\flat}_{H_0})\sum _{1\leq j\leq n} [\varphi _D^{j-1}u^{-1}]_{\Nil ^{\flat}_{H_0}} \left(
                                          \left[ \sum _{1\leq i\leq n}\varphi _D^{i-j}g_{j, i}\right] _{\Nil ^{\flat}_{H_0}}(\bY _j) \right) \\
                                      &=(\Nil ^{\flat}_{H_0})\sum _{1\leq j\leq n} [\varphi _D^{j-1}]_{\Nil ^{\flat}_{H_0}} (\bY _j) \\
                                      &=\bZ.
 \end{align*}
 By Lemma \ref{lem:esttoapp} we obtain the claim.
 
 \paragraph{Action of $(1, a_{\sigma}^{-1}, \sigma)$}
 Let $\sigma \in W_L$
 and set $a_{\sigma}, n_{\sigma}, u_{\sigma}$
 as in \ref{item:actionofWeilgp}.
 Again by \eqref{eq:actiononxi}, \eqref{eq:actiononYi}, \eqref{eq:actiononZ}
 and by $(1, a_{\sigma}^{-1}, \sigma) \in \calS$,
 we have,
 \begin{align*}
 &(1, a_{\sigma}^{-1}, \sigma)^{\ast}\bxi _i=\sigma(\bxi_i)=[u_{\sigma}]_{\Nil ^{\flat}_{H_0}} (\bxi _i),  \\
 &(1, a_{\sigma}^{-1}, \sigma)^{\ast}\bY _i=[u_{\sigma}]_{\Nil ^{\flat}_{H_0}} (\bY _i), \\
 &(1, a_{\sigma}^{-1}, \sigma)^{\ast}\bZ =(\Nil ^{\flat}_{H_0})\sum _{1\leq i\leq n} [\varphi _D^{i-1}u_{\sigma}]_{\Nil ^{\flat}_{H_0}} (\bY _i)
 =[u_{\sigma}]_{\Nil ^{\flat}_{H_0}}(\bZ).
 \end{align*}
 This shows that $(1, a_{\sigma}^{-1}, \sigma) \in \Stab _{\nu}'$.
 
 Suppose that $\nu =2\mu +1$ is odd.
 Then it follows from the above and (\ref{eq:capitaltoprime}) that
 \begin{align*}
 (1, a_{\sigma}^{-1}, \sigma)^{\ast}\by _i' &=\left( [u_{\sigma}]_{\Nil ^{\flat}_{H_0}} (\bxi _i)
                                                        \right) ^{-q^{\mu}(q+1)/2}
                                                  \cdot \left( [u_{\sigma}]_{\Nil ^{\flat}_{H_0}} (\bY _i) \right) \\
                                                    &=\left( [u_{\sigma}]_{\Nil ^{\flat}_{H_0}} (\bxi _i)
                                                        \right) ^{-q^{\mu}(q+1)/2}
                                                  \cdot \left( [u_{\sigma}]_{\Nil ^{\flat}_{H_0}} (\bxi _i^{q^{\mu}(q+1)/2}\by _i') \right).
 \end{align*}
 By Lemma \ref{lem:appD} \ref{htnnilaction}, Lemma \ref{lem:esttoapp} we see that
 \[
 \overline{(1, a_{\sigma}^{-1}, \sigma)^{\ast}{y_i'}^{q^{-l}}}= \overline{u_{\sigma}^{(q-1)/2}{y_i'}^{q^{-l}}}
 \]
 for $l\geq 0$.
 By (\ref{eq:zyandprime}) we infer that the induced action on $y_i$ 
 is as stated in \ref{item:actionofWeilgp}
 in this case.
 The computations proceed in a similar and easier way
 for the action on $z$ in the odd $\nu$ case
 and for the action on $y_i$ and $z$ in the even $\nu$ case.
 
 \paragraph{The inclusion $\calS \subset \Stab _{\nu}'$}
 Let $(g, d, \sigma) \in \calS$.
 Let us prove that $(g, d, \sigma)$ stabilizes $\calX _{\nu}$
 and induces an action on $\scrX _{\nu}$. 
 We have $\sigma \in W_{L'}$
 by Proposition \ref{Prop:S} \ref{item:defofL'},
 and so $\alpha =\sigma (\varphi _L)/{\varphi _L}\in \mu _{n_q}(K)$.
 We put $g'=\diag (\alpha ^{-1}, \dots, \alpha ^{-n})\in \GL _n(K)$
 and take $d'\in \mu _{q^n-1}(K_n)\subset D^{\times}$ such that $d'^{1-q}=\alpha$.
 Then we have
 \[
 \sigma(x)=g'xg'^{-1}=d'xd'^{-1}
 \]
 for all $x\in L^{\times}$.
 Thus
 by Proposition \ref{Prop:j} 
 there exists $\sigma' \in W_{L'}$ such that $(g', d', \sigma')\in \calS$
 and $\sigma |_L=\sigma'|_L$.
 It suffices to prove the claim for $(g', d', \sigma')$,
 because 
 \[
 (g, d, \sigma)^{-1}(g', d', \sigma')\in \Delta _{\xi}(L^{\times})\cdot \{ (1, a_{w}^{-1}, w)\mid  w \in W_L\} \subset \Stab_{\nu}'.
 \]
 
 As before, we have 
 by \eqref{eq:actiononxi}, \eqref{eq:actiononYi}, \eqref{eq:actiononZ}
 and by $(g', d', \sigma') \in \calS$
 \begin{align*}
 (g', d', \sigma')^{\ast}\bxi _i&=\sigma'(\bxi_i)=[\alpha ^{-i}d'^{-1}]_{\Nil ^{\flat}_{H_0}}(\bxi _i), \\
 (g', d', \sigma')^{\ast}\bY _i&=[\alpha ^{-i}d'^{-1}]_{\Nil ^{\flat}_{H_0}}(\bY _i), \\
 (g', d', \sigma')^{\ast}\bZ&=(\Nil ^{\flat}_{H_0})\sum _{1\leq i\leq n} [\varphi _D^{i-1}\alpha ^{-i}d'^{-1}]_{\Nil ^{\flat}_{H_0}} (\bY _i) \\
 &=(\Nil ^{\flat}_{H_0})\sum _{1\leq i\leq n} [\alpha ^{-i}\cdot \alpha^{i-1}d'^{-1}\varphi _D^{i-1}]_{\Nil ^{\flat}_{H_0}} (\bY _i)
 =[\alpha ^{-1}d'^{-1}]_{\Nil ^{\flat}_{H_0}}(\bZ). 
 \end{align*}
 Since $\alpha ^{-i}d'^{-1}\in U_D$, this is enough to conclude by Lemma \ref{lem:appD} \ref{htnnilaction}.

 \paragraph{The inclusion $\Stab _{\nu}' \subset (U_{\frakI}^{(\nu)}\times U_D^{(\nu)}\times \{ 1\})\cdot \calS$}
 To prove the inclusion 
 $\Stab _{\nu}' \subset (U_{\frakI}^{(\nu)}\times U_D^{(\nu)}\times \{ 1\})\cdot \calS$,
 we take an element in $\Stab _{\nu}'$
 and write it as $(g, d\varphi _D^{-n_\sigma}, \sigma )$
 with $(g, d, 1)\in G^0$ and $\sigma \in W_K$.
 As we have shown $\varphi _G\in \Stab _{\nu}'$,
 we may further assume $(g, 1, 1), (1, d, 1)\in G^0$.
 
 Let us first show $\sigma \in W_{L'}$.
 There exists an element $\zeta _n\in \mu _n(\overline{K})$
 such that $\sigma (\varphi _L)=\zeta _n\varphi _L$.
 We are to prove that $\zeta _n\in \mu _{q-1}(K)$.
 Since $\xi ^{(g, d\varphi _D^{-n_\sigma}, \sigma )}\in \calX _{\nu}$,
 we have by \eqref{eq:actiononYi}
 \begin{equation} \label{eq:Yigeneral}
 \lv \left[ \sum _{1\leq j\leq n}g_{j, i}d^{-1}\varphi _D^{i-j}\right] _{\Nil ^{\flat}_{H_0}}(\bxi _i) 
      -_{\Nil ^{\flat}_{H_0}} \sigma (\bxi _i)
 \rv \leq 
 \lv \bxi _i\rv ^{q^{c_{\nu}}}
 \end{equation}
 for all $1\leq i\leq n$.
 In particular,
 \[
 \lv \left[ \sum _{1\leq j\leq n}g_{j, i}d^{-1}\varphi _D^{i-j}\right] _{\Nil ^{\flat}_{H_0}}(\bxi _i) 
      -_{\Nil ^{\flat}_{H_0}} \sigma (\bxi _i)\rv 
 <
 \lv \bxi _i\rv
 \]
 for $1\leq i\leq n$.
 By Lemma \ref{lem:app2} and Lemma \ref{lem:appD} \ref{htnnilaction}
 this implies
 \[
 \sum _{1\leq j\leq n}g_{j, i}d^{-1}\varphi _D^{i-j} \in U_D \text{ and } \overline{\sum _{1\leq j\leq n}g_{j, i}d^{-1}\varphi _D^{i-j}}=\overline{{\xi _i}^{-1}\sigma (\xi _i)}.
 \]
 In particular we see that $\sum _{1\leq j\leq n}g_{j, i}\varphi _D^{i-j} \in U_L$.
 The second equation further implies
 $\overline{d}^{q-1}\in k^{\times}$
 because $\xi _i^{-1}\sigma (\xi _i)=(\xi _{i+1}^{-1}\sigma (\xi _{i+1}))^q$.
 On the other hand,
 by Proposition \ref{Prop:explicitCM} \ref{item:Weilactiononxi}, 
 we have 
 $\overline{\xi _1^{-1} \sigma (\xi _1)}^{q-1} = \overline{\zeta _n}$.
 Therefore,
 \[
 \overline{\zeta} _n=\overline{\xi _1^{-1} \sigma (\xi _1)}^{q-1}=\overline{\sum _{1\leq j\leq n}g_{j, 1}d^{-1}\varphi _D^{1-j}}^{q-1}=\overline{d}^{-(q-1)}\in k^{\times},
 \]
 which amounts to $\sigma \in W_{L'}$.
 
 Now we may assume $\sigma=1$, so that $(g, 1, 1), (1, d, 1)\in G^0$ and $(g, d, 1)\in \Stab _{\nu}'$.
 We consider \eqref{eq:Yigeneral} for $i=1$
 and apply Lemma \ref{lem:appD} \ref{htnnilaction} again to find that
 \[
 v_D(\sum _{1\leq j\leq n}g_{j, 1}d^{-1}\varphi _D^{1-j}-1)\geq \lfloor (\nu +1)/2\rfloor,
 \]
 which is to say,
 \[
 d\equiv \sum _{1\leq j\leq n}g_{j, 1}\varphi _D^{1-j} \pmod{U_D^{\lfloor (\nu +1)/2\rfloor}},
 \]
 and hence $d\in U_LU_D^{\lfloor (\nu +1)/2\rfloor}=U_LU_D^{(\nu)}$.
 Therefore, we may further assume $d=1$.
 Then $g\in \Stab _{\nu}' \cap \GL _n(K)=U_{\frakI}^{(\nu)}$.
 This completes the proof of the claim
 and also of the theorem.
 
 \subsection{Alternative description of the reductions in terms of algebraic groups and quadratic forms} \label{subsec:alggps}
 In \cite[\S3.4]{BWMax},
 algebraic varieties obtained by the reduction
 of affinoids
 are described in terms of the Lang torsors of 
 certain
 algebraic groups.
 Motivated by their observation, 
 we give here an alternative description of $Z_{\nu}$ 
 using algebraic groups $\calG _{\nu}$ 
 and quadratic forms $Q_{\nu}$
 for $\nu >0$ not divisible by $n$.
 It suffices to treat the cases where $0<\nu <2n$.
 
 Suppose first that $1\leq \nu =2\mu +1< 2n$ is odd. 
 We put $\calG _{\nu} =\bbG _a$,
 considered over $k$.
 We define a quadratic form $Q_{\nu}(y_1, \dots , y_n)\in k[y_1, \dots , y_n]$
 by
 \[
 Q_{\nu}(y_1, \dots , y_n)=\begin{cases}
                                -\sum _{\mu <j-i<n-\mu} y_iy_j         &\text{ if $1\leq \nu <n$} \\
                                \sum _{n-\mu \leq j-i\leq \mu} y_iy_j &\text{ if $n+1\leq \nu <2n$}.
                                \end{cases}
 \]
 If we denote by $F_q$ the $q$-th power Frobenius endomorphism,
 then the Lang torsor $L_{\calG}$ of an algebraic group $\calG$
 over $k$ is defined by
 \[
 L_{\calG} \colon \calG \rightarrow \calG ; \ x\mapsto F_q(x)\cdot x^{-1}.
 \]
 In this case 
 the Lang torsor $L _{\calG _{\nu}}$ of $\calG _{\nu}$
 is nothing but the Artin-Schreier map:
 \[
 L _{\calG _{\nu}}\colon \calG _{\nu}\rightarrow \calG _{\nu}; \ x\mapsto x^q-x.
 \]
 Then it is clear from Theorem \ref{Thm:reduction} \ref{item:red:smallodd}, \ref{item:red:largeodd}
 that $Z_{\nu}$ is isomorphic to the base change to $\overline{k}$ of 
 $Z_{\nu , 0}$
 defined by the following Cartesian diagram:
 \[
 \begin{CD}
 Z_{\nu , 0} @>>> \calG _{\nu} \\
 @VVV         @VV{L_{\calG _{\nu}}}V \\
 V @>>{Q_{\nu}|_{V}}> \calG _{\nu},
 \end{CD}
 \]
 where $V=V(y_1+\dots +y_n)\subset \bbA _{k}^{n}= \Spec k[y_1, \dots , y_n]$
 is a closed subscheme defined by
 $y_1+\dots +y_n=0$
 and $Q_{\nu}$ is considered as a morphism $\bbA _{k}^n\rightarrow \calG_{\nu}$.
 Note that the action of $S_{1, \nu}$ (and also $S_{2, \nu}$)
 agrees with the action of $\calG _{\nu}(k)$
 induced by the Lang torsor.
 
 Suppose that $\nu =2\mu$ is even.
 In this case
 we first define 
 an auxiliary algebraic group $\widetilde{\calG} _{\nu}$.
 We set $\widetilde{\calG} _{\nu}=\bbA _{k}^{n+1}$
 as a scheme
 and define a structure of an algebraic group
 by the same formula as \eqref{eq:S1formula}:
 \[
    (v, (w_i))\cdot (v', (w'_i))
    =(v+v'+\sum _{i}w_i w'_{i+\mu}, (w_i+w_i')) 
 \]
 for any $k$-algebra $R$ and
 for any $(v, (w_i))\in \widetilde{\calG} _{\nu}(R)$.
 We define
 \[
 \calG _{\nu}=\Ker \left( \widetilde{\calG} _{\nu}\rightarrow \bbG _a; \  (v, (w_i))\mapsto \sum _{i}w_i\right).
 \]
 By definition, we have $\calG _{\nu}(k)=S_{1, \nu}$.
 We put
 \[
 Q_{\nu}(y_1, \dots , y_n)=\begin{cases}
                                -\sum _{\mu <j-i<n-\mu} y_iy_j         &\text{ if $1\leq \nu <n$} \\
                                -\sum _{n-\mu < j-i< \mu} y_iy_j &\text{ if $n+1\leq \nu <2n$}
                                \end{cases}
 \]
 and put $f=(Q_{\nu}, \text{id})\colon \bbA _k^n\rightarrow \widetilde{\calG} _{\nu}$.

 If $1\leq \nu <n$, then we define $Z_{\nu ,0}$ by 
 the following Cartesian diagrams:
 \[
 \begin{CD}
 Z_{\nu, 0} @>>> \calG _{\nu} \\
 @VVV         @VV{(\cdot)^{-1}\circ L_{\calG _{\nu}}}V \\
 V @>>>           \calG _{\nu} \\
 @VVV                         @VVV \\
 \bbA _k^n @>>f> \widetilde{\calG} _{\nu}.
 \end{CD}
 \]

 If $n+1\leq \nu <2n$, then we define $Z_{\nu ,0}$ by 
 the following Cartesian diagrams:
 \[
 \begin{CD}
 Z_{\nu, 0} @>>> \calG _{\nu}\\
 @VVV         @VV{L_{\calG _{\nu}}}V \\
 V @>>> \calG _{\nu} \\
 @VVV                         @VVV \\
 \bbA _k^n @>>f> \widetilde{\calG} _{\nu}.
 \end{CD}
 \]
 Note that the Lang torsor (or its composite with $(\cdot)^{-1}$)
 induces an action of $S_{1, \nu}=\calG _{\nu}(k)$
 on $Z_{\nu, 0}$
 in each case.

 \begin{Prop}
 There exists a natural isomorphism
 between $Z_{\nu}$ and 
 the base change to $\overline{k}$ of $Z_{\nu, 0}$
 which respects the actions of $S_{1, \nu}$.
 \end{Prop}
 
 \begin{proof}
  This can be verified by a computation.
  Note that if we set $\wp (x)=x^q-x$
  and $\nu =2\mu$,
  then we have
  \begin{align*}
   L_{\calG_{\nu}}(v, (w_i))&=\left( \wp (v)-\sum _{i}\wp (w_{i-\mu})w_i, \wp (w_1), \dots , \wp (w_n) \right) \\
   L_{\calG_{\nu}}(v, (w_i))^{-1}&=\left( -\wp (v)+\sum _{i}\wp (w_{i-\mu})w_i^q, -\wp (w_1), \dots , -\wp (w_n) \right)
  \end{align*}
  on valued points.
 \end{proof}
 
 \begin{Rem} \label{Rem:Lang}
 It is interesting that the complicated defining equation of $Z_{\nu}$
 simplifies  with the introduction of $\calG _{\nu}$
 when $\nu$ is even.
 While
 this description
 is not used 
 in the computation of the cohomology in the case where $\nu$ and $n$ are coprime
 (Proposition \ref{Prop:evencoh}),
 if $n$ and $\nu$ are not coprime
 then one can conceptually deduce a product decomposition of $Z_{\nu}$
 using a product decomposition of $\calG _{\nu}$
 and thereby compute the cohomology;
 details may appear elsewhere.
 \end{Rem}

\section{Cohomology of the reductions} \label{sec:cohofred}
 \subsection{Quadratic forms and $\ell$-adic cohomology} \label{subsec:quadformsandcoh}
 Let $V$ be a $k$-vector space of dimension $m$
 and $Q=Q(x_1, \dots , x_m)$ a quadratic form on $V$.

 Suppose that $p\neq 2$.
 We define the associated symmetric bilinear form $b_Q\colon V\times V \rightarrow k$
 by
 \[
 b_Q(v_1, v_2)=2^{-1}(Q(v_1+v_2)-Q(v_1)-Q(v_2)).
 \]
 Then $Q$ is non-degenerate if and only if
 $b_Q$ is non-degenerate.
 In this case,
 we put $\det Q=\det b_Q \pmod{k^{\times 2}}$.
 In general,
 we have an orthogonal decomposition
 $(V, Q)=(V_{\text{nd}}, Q_{\text{nd}}) \oplus (V_{\text{null}}, Q_{\text{null}})$,
 where $Q_{\text{nd}}$ is non-degenerate
 and $Q_{\text{null}}$ is zero.
 We call $r=\dim V_{\text{nd}}$ the rank of $Q$. 
 We simply write $Q=Q_{\text{nd}}\oplus Q_{\text{null}}$.
 Let $\psi $ be a non-trivial character of $k$.
 To treat quadratic exponential sums associated to 
 quadratic forms
 by reducing to the standard one,
 we put 
 \[
 \frak{g} (\psi)=\sum _{x\in k}\psi (x^2)=\sum _{x\in k}{x \overwithdelims ()k}\psi (x),
 \]
 where ${\cdot \overwithdelims ()k}$ is the quadratic residue symbol of $k$.
 
 Suppose that $p=2$.
 We define the associated alternating bilinear form $a_Q\colon V\times V\rightarrow k$
 by
 \[
 a_Q(v_1, v_2)=Q(v_1+v_2)-Q(v_1)-Q(v_2).
 \]
 For $a, b\in k$
 we define a quadratic form $Q_{a, b}$
 on $k^2$
 by $Q_{a, b}(x, y)=ax^2+xy+by^2$.
 For $c\in k$
 we define a quadratic form $Q_{c}$
 on $k$
 by $Q_{c}(z)=cz^2$. 
 It is well-known that
 there exists an orthogonal decomposition
 \begin{equation} \label{eq:quasi-diagonalization}
 (V, Q)\simeq \bigoplus _{1\leq i\leq r'}(k^2, Q_{a_i, b_i})
                   \oplus (k, Q_c)^{\oplus \varepsilon}
                   \oplus (k, Q_0)^{\oplus s},
 \end{equation} 
 where $r'\geq 0$, $\varepsilon \in \{ 0, 1\}$, $s\geq 0$,
 $a_i, b_i\in k$ and $c\in k^{\times}$.
 This is called a quasi-diagonalization of $Q$.
 Here $r'$, $\varepsilon$, $s$ do not depend on the choice of decomposition
 and we call $r=2r'$ the rank of $Q$.
 We denote by $(V_{\text{nd}}, Q_{\text{nd}})$ (resp.\ $(V_{\text{ql}}, Q_{\text{ql}})$)
 the quadratic space given by the non-degenerate subspace $\bigoplus _{1\leq i\leq r'}(k^2, Q_{a_i, b_i})$
 (resp.\ the ``quasi-linear'' subspace
 $(k, Q_c)^{\oplus \varepsilon}
 \oplus (k, Q_0)^{\oplus s}$) 
 and simply write $Q=Q_{\text{nd}}\oplus Q_{\text{ql}}$.
 We put 
 $\Arf (Q_{\text{nd}})=\sum _{1\leq i\leq r'}a_ib_i \pmod{\wp(k)}$,
 where $\wp(k)=\{ x^2+x\mid x\in k\}$.
 This is also an invariant of $Q$, called the Arf invariant.
 
 The following can be found, for instance, in \cite[Definition 10.2]{LicKnot}:
 \begin{Prop} \label{Prop:Arfbycount}
 Let $(V, Q)$ be as above.
 Suppose that $Q$ is non-degenerate (so that $m=2r'$ is even)
 and $k=\bbF _2$.
 Then the number of elements of the fiber $Q^{-1}(1)$
 is either $2^{m-1}-2^{r'-1}$ or $2^{m-1}+2^{r'-1}$.
 We have $\Arf (Q)=0$ in the first case, 
 and  $\Arf (Q)=1$ in the second case.
 \end{Prop}
 
 We record an elementary computation of the cohomology
 of certain varieties associated to quadratic forms.
 \begin{Prop} \label{Prop:quadandcoh}
  Let $Q$ and $k$ be as above.
  Let $X_0 $ be the algebraic variety defined by the Cartesian diagram;
  \[
  \begin{CD}
   X_0 @>>> \bbA _k^1 \\
   @VVV     @VV\wp _k V \\
   \bbA _k^m @>>Q> \bbA _k^1,
  \end{CD} 
  \]
  where $\wp _k$ is the Artin-Schreier map $\wp _k(x)=x^q-x$
  and $Q$ is considered as a morphism.
  Denote by $X$ the base change of $X_0$ to the algebraic closure $\overline{k}$.
  Take a prime number $\ell \neq p$
  and put
  \[
  H_c^i=H_c^i(X, \overline{\bbQ} _{\ell}),
  \]
  which carries the actions of the additive group $k$ and $\Omega =\Gal (\overline{k} /k)$.
  Then the following assertions hold.
  \begin{enumerate}[(1)]
   \item Suppose that $p\neq 2$.
   Let $r$ be the rank of $Q$ 
   and express $Q$ as an orthogonal sum $Q=Q_{\emph{nd}}\oplus Q_{\text{null}}$ 
   as before.
   Suppose that $r>0$.
   Then we have
   \[
   H_c^i\simeq \begin{cases}
                     \bigoplus _{\psi \in k^{\vee} \setminus \{ 1\} } V_{\psi} &\text{ if $i=2m-r$} \\
                     \overline{\bbQ} _{\ell}(-m) &\text{ if $i=2m$} \\
                     0 &\text{ otherwise},
                    \end{cases}
   \]   
   where $V_{\psi}$ is a one-dimensional vector space 
   on which $k$ acts via $\psi$ and
   the $q$-th power geometric Frobenius element $\Frob _q$ acts as multiplication by the scalar
   \[
   (-1)^{2m-r}{{\det Q_{\emph{nd}}}\overwithdelims ()k} \frak{g} (\psi)^r q^{m-r},
   \]
   and $\overline{\bbQ} _{\ell}(-m)$ is 
   a one-dimensional vector space 
   on which $k$ acts trivially and
   $\Frob _q$ acts as multiplication by the scalar
   $q^m.$
   \item Suppose that $p=2$.
   Let $r$ be the 
   rank of $Q$ 
   and express $Q$ as an orthogonal sum $Q=Q_{\emph{nd}}\oplus Q_{\emph{ql}}$ 
   as before.
   Suppose that $r>0$.
   Let $\eps \in \{ 0, 1\}$ and $s\geq 0$ be as in \eqref{eq:quasi-diagonalization}.
   We denote by $\psi _0$ 
   the unique non-trivial character of $\bbF _2$.
   Then we have
   \[
   H_c^i\simeq \begin{cases}
                     \bigoplus _{\psi \in k^{\vee} \setminus \{ 1\} } V_{\psi} &\text{ if $i=r+2s(=2m-r-2\eps )$} \\
                     \overline{\bbQ} _{\ell}(-m) &\text{ if $i=2m$} \\
                     0 &\text{ otherwise},
                    \end{cases}
   \]   
   where $V_{\psi}$ is a one-dimensional vector space 
   on which $k$ acts via $\psi$ and
   $\Frob _q$ acts as multiplication by the scalar
   \[
   (-1)^{2m-r-2\eps }\psi _0\left( \Tr _{k/{\bbF _2}} \Arf \left( Q_{\emph{nd}}\right) \right) q^{m-r/2-\eps }
   \]
   $($which is independent of $\psi)$,
   and $\overline{\bbQ} _{\ell}(-m)$ is 
   a one-dimensional vector space 
   on which $k$ acts trivially and
   $\Frob _q$ acts as multiplication by the scalar
   $q^m.$
  \end{enumerate}
 \end{Prop}
 \begin{proof}
 For any additive character $\psi$ of $k$,
 let $\scrL_{\psi}$ 
 denote the Artin-Schreier 
 $\overline{\bbQ} _{\ell}$-sheaf on 
 $\bbA_k^1$ associated to $\psi$, 
 which is equal to 
 $\frakF (\psi)$ in the notation of 
 \cite[Sommes trig.\ 1.8 (i)]{DelCoet}. 
 No matter whether $p=2$ or not, we have an isomorphism
 \[
 H_c^i\simeq \bigoplus _{\psi \in k^{\vee}} H_c^i(\bbA _{\overline{k}}^m, Q^{\ast}\scrL_{\psi})
 \]
 as representations of $k\times \Omega$.
 
 Suppose that $p\neq 2$.
 The assertion in this case is well-known.
 Diagonalizing $Q$ and applying the K\"unneth formula, 
 we are reduced to computing
 $H_c^i(\bbA _{\overline{k}}^m, Q^{\ast}\scrL_{\psi})$
 for $m=1$ and $Q(x)=ax^2, (a\in k)$.
 As the pull-back of $\scrL_{\psi}$ by the zero map is the constant sheaf, 
 giving rise to $\overline{\bbQ} _{\ell}(-1)[-2]$,
 it suffices to show the proposition in $r=m=1$ case, 
 which is done by the Grothendieck-Ogg-Shafarevich formula and
 the Grothendieck-Lefschetz trace formula.
 
 Suppose now that $p=2$.
 Again, by quasi-diagonalizing $Q$, 
 we are reduced to computing
 $H_c^i(\bbA _{\overline{k}}^m, Q^{\ast}\scrL_{\psi})$
 for either $m=1$ and $Q=Q_a$
 or $m=2$ and $Q=Q_{a, b}$.
 As the computation is easier 
 if $\psi$ is trivial,
 we assume that $\psi$ is non-trivial.
 In the first case,
 if $a=0$, then the cohomology is 
 $\overline{\bbQ} _{\ell}(-1)[-2]$
 as above
 and if $a\neq 0$,
 then $Q$ is a morphism of 
 additive group schemes $\bbG _a\simeq \bbA _k^1$,
 which implies that the cohomology vanishes 
 in every degree
 by \cite[Sommes trig.\ Th\'eor\`eme $2.7^{\ast}$]{DelCoet}.
 In the remaining case
 we may assume that $\psi$ is non-trivial.
 Recall the isomorphism
 $k\xrightarrow{\sim} k^{\vee}; \ x\mapsto \psi _{0, x},$
 where $\psi _{0, x}(y)=\psi _0\left( \Tr _{k/{\bbF _2}}\left( xy\right) \right),$
 and take $c\in k^{\times}$ such that $\psi =\psi _{0, c}.$ 
 Then we see that
 $\psi _{0, c}(ax^2)=\psi _{0, d}(x)$ for all $x\in k$ 
 with $d=(ca)^{1/2}\in k.$
 Now we can turn 
 an elementary manipulation
 \begin{align*}
 \sum _{x, y\in k} \psi _{0, c}(Q(x, y))  &=\sum _{x, y} \psi _{0, c}(ax^2+xy+by^2) \\
                                                 &=\sum _y \psi _{0, c}(by^2)\sum _x \psi _{0, c}(ax^2+xy) \\
                                                 &=\sum _y \psi _{0, c}(by^2)\sum _x \psi _{0, d+cy}(x) \\
                                                 &=\psi _{0, c}\left( b(-d/c)^2\right)q \\
                                                 &=\psi _0\left( \Tr _{k/{\bbF _2}} \Arf (Q)\right)q
 \end{align*}
 into the desired cohomological statement
 as in the proof of \cite[Proposition 2.10]{BoyDLconst}.
 \end{proof}
 
 \subsection{Representations of a cyclic group in finite classical groups} \label{subsec:repofcycfgps}
 In \cite{BFGdiv} and \cite{BHetLLCII},
 one is naturally led to consider
 orthogonal and symplectic representations of a cyclic group over a finite field
 in order to compute subtle invariants of certain representations.
 We use the theory in our analysis of the cohomology of $Z_{\nu}$. 
 Thus we summarize parts of \cite[\S4]{BHetLLCII}
 in this subsection. 
 
 We put $\Omega =\Gal (\overline{k} /k)$
 and $\Gamma ={\bbZ}/{n\bbZ}$, 
 where $n$ is assumed to be coprime to $p$
 as always. 
 
 Let $\Omega$ act on $\widehat{\Gamma} =\Hom (\Gamma , \overline{k} ^{\times})$
 via its natural action on the target.
 For $\chi \in \widehat{\Gamma}$, 
 we define a $k[\Gamma]$-module $V_{\chi}$ 
 in the following way:
 the underlying vector space is 
 the field
 $k[\chi]\subset \overline{k}$
 generated by the values of $\chi$
 and 
 $\Gamma$ acts 
 via the character 
 $\chi \colon \Gamma \rightarrow k[\chi]^{\times}$.
 \begin{Prop}
 The $k[\Gamma]$-module $V_{\chi}$ is simple
 and its isomorphism class depends only on the $\Omega$-orbit of $\chi$.
 Moreover, we have a bijection
 between the set of $\Omega$-orbits of $\widehat{\Gamma}$ and 
 the set of isomorphism classes of simple $k[\Gamma]$-modules
 induced by $\chi \mapsto V_{\chi}$.
 In particular, the following decomposition holds;
 \[
 k[\Gamma]=\bigoplus _{\chi \in \Omega \backslash \widehat{\Gamma}} V_{\chi}.
 \]
 \end{Prop}
 
 \begin{Rem} \label{Rem:action and orthogonality}
 Let $V_1$ and $V_2$ be $k[\Gamma]$-modules and $Q$ a possibly degenerate $\Gamma$-invariant quadratic form
 on $V=V_1\oplus V_2$.
 It is easy to see that if $V_1$ and the contragredient $V_2^{\vee}$ of $V_2$
 do not have any common simple factors,
 then $V_1$ and $V_2$ are orthogonal with respect to $Q$.
 \end{Rem}
 As we need to treat $p=2$ case,
 we define \emph{an orthogonal representation} $(V, Q)$ of $\Gamma$ over $k$ 
 to be a $k[\Gamma]$-module $V$ 
 endowed with a non-degenerate $\Gamma$-invariant quadratic form $Q$.
 If $p\neq 2$,
 non-degenerate quadratic forms $Q$ correspond to 
 non-degenerate symmetric bilinear forms $b_Q$
 and thus this notion coincides with the usual one.
 
 \begin{Prop} \label{Prop:orthrep} 
 Assume that $p\neq 2$. 
 Let $(V, Q)$ be an orthogonal representation 
 of $\Gamma$
 over $k$.
 \begin{enumerate}[(1)]
 \item Suppose that $(V, Q)$ is indecomposable. 
 Then exactly one of the following holds.
  \begin{enumerate}[(i)]
   \item The underlying $k[\Gamma]$-module $V$ is simple 
   and isomorphic to $V_{\chi}$ with $\chi ^2=1$.
   \item $V$ is isomorphic to $U\oplus U^{\vee}$,
   where $U$ is a simple $k[\Gamma]$-module 
   which is not isomorphic to its contragredient $U^{\vee}$.
   \item $V=V_{\chi}$ is simple, isomorphic to its contragredient,
   but $\chi ^2\neq 1$.
  \end{enumerate}
  Moreover, the isometry class of $(V, Q)$
  as an orthogonal representation
  is determined by the isomorphism class of $V$
  in the last two cases.
 \item There exists a decomposition of $(V, Q)$ 
 into an orthogonal sum of indecomposable orthogonal representations.
 In particular, 
 $\det Q$ is determined by 
 the isomorphism class of the underlying $k[\Gamma]$-module $V$ and
 the restriction of $Q$ 
 to the largest subspace of $V$ fixed by $\{ \gamma ^2\mid \gamma \in \Gamma \}$.
 \end{enumerate}
 \end{Prop}
 
 \begin{ex} \label{stdex}
 Let $a$ be a positive divisor of $n$
 and let $\Gamma ^a\subset \Gamma$ denote
 the unique subgroup of order $n/a$.
 The regular $k[\Gamma]$-module $k[\Gamma]$
 has canonical $\Gamma$-submodules
 \[
 k[{\Gamma}/{\Gamma ^a}]=k[\Gamma]^{\Gamma ^a}
                                   =\bigoplus _{\chi \in \Omega \backslash \widehat{\Gamma}, \chi ^a=1}
                                    V_{\chi}, \quad 
 I_k(\Gamma ; a, n)=\bigoplus _{\chi \in \Omega \backslash \widehat{\Gamma}, \chi ^a\neq1}
                                    V_{\chi}.
 \]
 Then $I_k(\Gamma ; a, n)$ is the unique complement of $k[{\Gamma}/{\Gamma ^a}]$ in $k[\Gamma]$
 as a $\Gamma$-submodule.
 Similarly, if $a\mid b\mid n$, we define $I_k(\Gamma ; a, b)$ by
 \[
 I_k(\Gamma ; a, n)=I_k(\Gamma ; b, n)\oplus I_k(\Gamma ; a, b).
 \]
 If $Q$ is a (possibly degenerate) $\Gamma$-invariant quadratic form on $I_k(\Gamma ; a, n)$,
 then this decomposition is orthogonal with respect to $Q$ by Remark \ref{Rem:action and orthogonality}.
 
 Suppose now that $p\neq 2$.
 Let $\varepsilon \colon k[\Gamma]\rightarrow k$ be the $k$-linear map
 sending $1\in \Gamma$ to $1\in k$ and 
 $1\neq \gamma \in \Gamma$ to $0\in k$,
 and let $x\mapsto \overline{x}$ be the standard $k$-linear involution
 on $k[\Gamma]$
 such that $\overline{\gamma}=\gamma ^{-1}$ for $\gamma \in \Gamma$.
 Then $Q_{\Gamma}(x)=\eps (x\overline{x})$ $(x\in k[\Gamma])$
 defines a $\Gamma$-invariant 
 non-degenerate quadratic form on $k[\Gamma]$
 such that $\det Q_{\Gamma}=1$.
 We see that
 \[
 Q_{\Gamma}|_{k[{\Gamma}/{\Gamma ^a}]}=a^{-1}nQ_{{\Gamma}/{\Gamma ^a}}
 \]
 and hence
 \[
 \det (Q_{\Gamma}|_{k[{\Gamma}/{\Gamma ^a}]})=(a^{-1}n)^a, \quad
 \det (Q_{\Gamma}|_{I_k(\Gamma ; a, n)})=(a^{-1}n)^a \pmod{k^{\times 2}}.
 \]
 \end{ex}
 
 As usual,
 by \emph{a symplectic representation}
 of $\Gamma$ over $k$
 we mean a pair $(V, b)$ consisting of 
 a $k[\Gamma]$-module $V$ 
 and a non-degenerate $\Gamma$-invariant 
 alternating form $b$.
 
  \begin{Prop}\footnote{In \cite{BFGdiv}, \cite{BHetLLCII}, 
                              this proposition is stated under the assumption
                              that $k=\bbF _p$.
                              However, it plays a role only in the discussion
                              of the computation of trace invariants
                              and this proposition remains true without the assumption.}
                 \label{Prop:symprep} 
 Let $(V, b)$ be a symplectic representation 
 of $\Gamma$
 over $k$.
 \begin{enumerate}[(1)]
 \item Suppose that $(V, b)$ is indecomposable. 
 Then exactly one of the following holds.
  \begin{enumerate}[(i)]
   \item The underlying $k[\Gamma]$-module $V$ 
   is isomorphic to $U\oplus U^{\vee}$,
   where $U$ is either isomorphic to $V_{\chi}$
   for some $\chi$ with $\chi ^2=1$,
   or is a simple $k[\Gamma]$-module 
   which is not isomorphic to its contragredient.
   \item $V=V_{\chi}$ is simple, isomorphic to its contragredient,
   but $\chi ^2\neq 1$.
  \end{enumerate}
  Moreover, the isometry class of $(V, b)$
  as a symplectic representation
  is determined by the isomorphism class of $V$
  in both the cases.
 \item There exists a decomposition of $(V, b)$ 
 into an orthogonal sum of indecomposable symplectic representations.
 \end{enumerate}
 \end{Prop}

 \begin{Rem} \label{Rem:char2orth}
 Suppose that $p=2$. 
 Let $(V, Q)$ be an orthogonal representation of $\Gamma$ over $k$.
 Considering the orthogonal decomposition of 
 the alternating bilinear form $a_Q$
 associated to $Q$,
 we may separately study
 the isotypic components of $V$ appearing in 
 Proposition \ref{Prop:symprep}
 to compute the Arf invariant of $Q$.
 \end{Rem}
 
 \subsection{Certain Heisenberg groups and their representations} \label{subsec:Heisenberg}
 Let $H$ be a finite group
 and $Z\subset H$ its center.
 We say that a finite group $H$ is
 a Heisenberg group
 if it is not abelian
 and $Q=H/Z$ is abelian.
 If $H$ is a Heisenberg group,
 the map $H\times H\rightarrow Z; \ (x, y)\mapsto [x, y]=xyx^{-1}y^{-1}$ 
 induces an alternating bilinear map $[\cdot ]\colon Q\times Q\rightarrow Z$.
 
 The following is well-known; see for instance \cite[Exercises 4.1.4-4.1.7]{BumAuto}.
 \begin{Prop} \label{Prop:Heisenberg}
 Let $H, Z, Q$ be as above.
 Let $\psi \colon Z\rightarrow \overline{\bbQ} _{\ell}^{\times}$ be a character of $Z$.
 Assume that $\psi \circ [\cdot ]\colon Q\times Q\rightarrow \overline{\bbQ} _{\ell}^{\times}$
 is non-degenerate.
 Then there exists an irreducible representation $\rho _{\psi}$ of $H$,
 unique up to isomorphism,
 whose central character is $\psi$. 
 Moreover, $\dim \rho _{\psi}=\sqrt{c}$,
 where $c$ is the order of $Q$.
 \end{Prop}
 
 The following proposition can be verified easily.
 \begin{Prop} \label{Prop:ourHeisenberg}
  Let $1< \nu =2\mu <2n$ be an even integer.
  Let $S_{1, \nu}$ (resp.\ $S_{2, \nu}$) be 
  the group defined in Proposition \ref{Prop:gpS1}
  (resp.\ in Proposition \ref{Prop:gpS2}).
  Assume that $n$ and $\nu$ are coprime.
  \begin{enumerate}[(1)]
   \item The group $S_{1, \nu}$ is a Heisenberg group.
   In the notation of Proposition \ref{Prop:gpS1} \ref{item:evenS1},
   the center $Z(S_{1, \nu})$ is
   \[
   Z(S_{1, \nu})=k=\{ (v, (w_i))\in k\times k^{{\bbZ}/{n\bbZ}} \mid w_i=0 \text{ for all $i$} \}
   \subset S_{1, \nu}.
   \]
   Moreover, a character $\psi$ of $k$
   induces a non-degenerate alternating form
   $\psi \circ [\cdot ]\colon S_{1, \nu}/{Z(S_{1, \nu})}\times S_{1, \nu}/{Z(S_{1, \nu})}\rightarrow \overline{\bbQ} _{\ell}^{\times}$
   if and only if $\psi$ is non-trivial.
   
   \item The group $S_{2, \nu}$ is a Heisenberg group.
   In the notation of Proposition \ref{Prop:gpS2} \ref{item:evenS2},
   the center $Z(S_{2, \nu})$ is
   \[
   Z(S_{2, \nu})=k=\{ (v, w)\in k\times k_n \mid w=0 \}
   \subset S_{2, \nu}.
   \]
   Moreover, a character $\psi$ of $k$
   induces a non-degenerate alternating form
   $\psi \circ [\cdot ]\colon S_{2, \nu}/{Z(S_{2, \nu})}\times S_{2, \nu}/{Z(S_{2, \nu})}\rightarrow \overline{\bbQ} _{\ell}^{\times}$
   if and only if $\psi$ is non-trivial.
  \end{enumerate}
 \end{Prop}
 If $n$ and $\nu$ are coprime, then by Propositions \ref{Prop:Heisenberg} and \ref{Prop:ourHeisenberg},
 there exists a unique irreducible representation $\rho _{1, \psi}$ (resp.\ $\rho _{2, \psi}$)
 of $S_{1, \nu}$ (resp.\ of $S_{2, \nu}$)
 with the central character $\psi$,
 for any non-trivial character $\psi $ of $k$.

 \subsection{Cohomology of the reductions} \label{subsec:cohofreds}
 \begin{Prop}\footnote{Although  
                             this proposition allows us to compute the cohomology of $Z_{\nu}$
                             for any $\nu$ not divisible by $n$,
                             only the cases where $n$ and $\nu$ are coprime are relevant to Main Theorem;
                             we find the result interesting nonetheless.}
                             \label{Prop:oddcoh}
 Let $\nu$ be an integer satisfying $1\leq \nu <n$ or $n+1\leq \nu <2n$. 
 Suppose that $\nu$ is odd
  and write $\nu =2\mu +1.$
 Put $d=\gcd(n, \nu).$
 Let $Z_{\nu}$ be the algebraic variety defined in Theorem \ref{Thm:reduction},
 which we regard as the base change of $Z_{\nu, 0}$ as in Subsection \ref{subsec:alggps}.
 Put 
 \[
 H_c^i=H_c^i(Z_{\nu}, \overline{\bbQ} _{\ell}),
 \]
 which carries the action of $k$ and $\Omega$.
 We denote by ${\cdot \overwithdelims ()m}$ the Jacobi symbol 
 for any positive odd integer $m$.
 \begin{enumerate}[(1)]
         \item \label{item:oddvanish}
                 We have
                 $H_c^i=0$
                 unless $i=2(n-1), n+d-2.$
         \item \label{item:oddnondeg}
                 If $d\neq n$, then
                 \[
                 H_c^i\simeq \begin{cases}
                                  \bigoplus _{\psi \in k^{\vee}\setminus \{ 1\}} W_{\psi} &\text{if $i=n+d-2$} \\
                                  \overline{\bbQ} _{\ell}(-(n-1)) &\text{if $i=2(n-1)$},
                                  \end{cases}
                 \]       
                 where $W_{\psi}$ is a one-dimensional vector space
                 on which $k$ acts via $\psi$ and
                 $\Frob _q$ acts as multiplication by the scalar
                 \[
                 \begin{cases}
                 {q\overwithdelims () {n/d}}q^{\frac{n+d-2}{2}} &\text{ if $n$ is odd and $p=2$} \\ 
                 {{n/d}\overwithdelims ()k}\frakg(\psi)^{n-d}q^{d-1} &\text{ if $n$ is odd and $p\neq 2$} \\
                 -{-1\overwithdelims ()k}^{\mu}{2\overwithdelims ()k}{{n/d}\overwithdelims ()k}\frakg(\psi)^{n-d}q^{d-1} 
                 &\text{ if $n$ is even (and hence $p\neq 2$),}
                 \end{cases}
                 \]
                 and $k$ acts trivially on $\overline{\bbQ} _{\ell}(-(n-1))$.
         \item \label{item:odddeg}
                 If $d=n$, then
                 \[
                 H_c^{2(n-1)}\simeq \overline{\bbQ} _{\ell}[k]\boxtimes \overline{\bbQ} _{\ell}(-(n-1))
                 \] 
                 as a representation of $k\times \Omega$.
         \item \label{item:oddcyclic}
                 Consider the following natural action of the standard generators $\gamma \in \Gamma ={\bbZ}/{n\bbZ}$ 
                 and $\beta \in {\bbZ}/{2\bbZ}$ 
                 on $Z_{\nu}$;
                 \begin{gather*}
                 (z, y_1, y_2, \dots , y_n)\mapsto (z, y_n, y_1, \dots , y_{n-1}), \\
                 (z, y_1, \dots , y_n)\mapsto (z, -y_1, \dots , -y_n)
                 \end{gather*}
                 respectively (cf. Subsection \ref{subsec:stabandaction}).
                 Let $j$ be an integer coprime to $n$.
                 Let $\psi$ be a character of $k$.
                 Then both the actions of $\gamma^j$ and $\beta$ induce, on the $\psi$-isotypic component of $\bigoplus _{i}H_c^i$, 
                 multiplication by the scalar 
                 \[
                 \begin{cases}
                 (-1)^{n-1}&\text{ if $\psi$ is non-trivial} \\
                 1&\text{ otherwise}.
                 \end{cases}
                 \]                      
 \end{enumerate}
 \end{Prop}
 \begin{proof}
 We first treat \ref{item:oddvanish} to \ref{item:odddeg}.
 Let $Q_{\nu}$ be the quadratic form defined in Subsection \ref{subsec:alggps}.
 In view of Proposition \ref{Prop:quadandcoh}, 
 we need to compute various invariants 
 of the quadratic form $Q_{\nu}$
 defining $Z_{\nu}$.
 For this, we follow the approach of \cite[\S8.3]{BHetLLCII}
 and exploit the $\Gamma$-invariance of $Q_{\nu}$.
 We regard $Q_{\nu}=Q_{\nu}(y_1, \dots , y_n)$ 
 as a quadratic form on
 the group algebra
 $k[\Gamma]=\left\{ \sum _{1\leq i\leq n} y_i\gamma ^i \ \mid \ y_i\in k\right\}.$
 
 In the notation of Example \ref{stdex}
 we are to study the restriction of $Q_{\nu}$
 on $I_k(\Gamma ; 1, n)\subset k[\Gamma]$.
 We have an orthogonal decomposition
 \begin{equation*}
 I_k(\Gamma ; 1, n)=I_k(\Gamma ; d, n)\oplus I_k(\Gamma ; 1, d).
 \end{equation*} 
 First we claim, 
 with no assumption on the parity of $p$,
 that
 (the restriction of) $Q_{\nu}$ is non-degenerate (resp.\ zero) 
 on $I_k(\Gamma ; d, n)$ (resp.\ on $I_k(\Gamma ; 1, d)$).

 Suppose that $1\leq \nu < n$.
 Then, for $x\in k[\Gamma],$ 
 \[
 Q_{\nu}(x)=\begin{cases}
               -\eps \left( \sum _{\mu +1\leq i\leq (n-1)/2} x\overline{(\gamma ^{-i}x)}\right) &\text{ if $n$ is odd} \\
               -\eps \left( \sum _{\mu +1\leq i\leq n/2-1} x\overline{(\gamma ^{-i}x)}
                            +2^{-1}x\overline{(\gamma ^{-n/2}x)}\right) &\text{ if $n$ is even}. 
               \end{cases}
 \]
 By Remark \ref{Rem:action and orthogonality}, Proposition \ref{Prop:orthrep} and Remark \ref{Rem:char2orth}
 we may separately consider each isotypic component
 underlying some indecomposable orthogonal (or symplectic, if $p=2$) representation of $\Gamma$.
 Also, to prove non-degeneracy or triviality of $Q_{\nu}$, 
 we may assume that 
 all characters of $\Gamma$ take values in $k^{\times}$.
 
 For a character $\chi$ of $\Gamma$, let
 \[
 e_{\chi}=n^{-1}\sum _{1\leq i\leq n}\chi (\gamma ^{i})\gamma ^{-i}
 \]
 be the corresponding idempotent,
 so that 
 $V_{\chi}=e_{\chi}k[\Gamma].$
 First let $\chi$ be a character 
 such that $\chi ^2\neq 1$
 and consider $Q_{\nu}|_{V_{\chi}\oplus V_{\chi ^{-1}}}.$
 Put $\alpha =\chi (\gamma).$
 If $n$ is odd, 
 then we have, for $y, z\in k,$
 \begin{align*}
 Q_{\nu}(ye_{\chi}+ze_{\chi ^{-1}})&=-\eps \left( 
                                                            \sum _{\mu +1\leq i\leq (n-1)/2} 
                                                            \left(
                                                                  ye_{\chi}+ze_{\chi ^{-1}}
                                                            \right)
                                                            (
                                                             \alpha ^{-i}ye_{\chi ^{-1}}+
                                                             \alpha ^{i}ze_{\chi}
                                                             )
                                                           \right) \\
                                           &=-n^{-1}yz\left( \sum _{\mu +1\leq i\leq (n-1)/2}(\alpha ^i+\alpha ^{-i})\right) \\
                                           &=-n^{-1}yz\left( \alpha^{n-\mu}-\alpha ^{\mu +1}\right)/ \left( \alpha -1\right). 
 \end{align*}
 Similarly, if $n$ is even, 
 we see that
 \begin{align*}
 Q_{\nu}(ye_{\chi}+ze_{\chi ^{-1}})&=-n^{-1}yz\left( \sum _{\mu +1\leq i\leq n/2-1} 
                                                               (\alpha ^i+\alpha ^{-i})
                                                               +\alpha ^{n/2}
                                                         \right) \\
                                         &=-n^{-1}yz\left( \alpha^{n-\mu}-\alpha ^{\mu +1}\right)/ \left( \alpha -1\right). 
 \end{align*}
 Thus, $Q_{\nu}|_{V_{\chi}\oplus V_{\chi ^{-1}}}$ is trivial 
 if $\chi ^d=\id$
 and is non-degenerate
 otherwise.

 Next let $\chi$ be of order two
 (so that $n$ is even)
 and consider the restriction of $Q_{\nu}$ 
 on $V_{\chi}\subset I_k(\Gamma ; d, n)$.
 Then
 \begin{equation} \label{eq:detoftsd1}
 Q_{\nu}(e_{\chi})=-n^{-1}\left( \sum _{\mu +1\leq i\leq n/2-1} (-1)^i+2^{-1}(-1)^{n/2}\right) 
                     =(2n)^{-1}(-1)^{\mu}                  
 \end{equation}
 and $Q_{\nu}$ is non-degenerate on $V_{\chi}$.
 Hence we have proved the above claim 
 for $1\leq \nu <n.$
 The case where $n+1\leq \nu <2n$
 can be reduced to the above case 
 by noting that 
 $Q_{\nu}=-Q_{2n-\nu}.$
 In particular, we find that
 also in this case
 \begin{equation} \label{eq:detoftsd2}
 Q_{\nu}(e_{\chi})=(2n)^{-1}(-1)^{\mu} 
 \end{equation}
 for the character $\chi$ of order two, 
 if $n$ is even.
 
 Now suppose that 
 $p\neq 2$.
 Then we need to show that
 \[
 \det Q_{\nu}|_{I_k(\Gamma ; d, n)}=\begin{cases}
                                            n/d &\text{ if $n$ is odd} \\
                                            (-1)^{\mu}2n/d &\text{ if $n$ is even}  
                                            \end{cases}
 \pmod{k^{\times 2}}.
 \]
 We compare $Q_{\nu}$
 with the standard quadratic form $Q_{\Gamma}$ on $k[\Gamma].$
 If $n$ is even 
 and $\chi$ is the character of order two, 
 then the determinant 
 of $Q_{\Gamma}|_{V_{\chi}}$
 is $1/n.$
 By Proposition \ref{Prop:orthrep},
 (\ref{eq:detoftsd1}), (\ref{eq:detoftsd2}),
 we infer that
 the determinants of $Q_{\nu}|_{I_k(\Gamma ; d, n)}$
 and $Q_{\Gamma}|_{I_k(\Gamma ; d, n)}$
 differ by a factor of $(-1)^{\mu}2$
 if $n$ is even
 and coincide 
 if $n$ is odd.
 Now the determinant of $Q_{\Gamma}|_{I_k(\Gamma ; d, n)}$
 is indeed $n/d$, 
 as is seen from Example \ref{stdex}.
 
 Suppose that $p=2$ (and hence $n$ is odd).
 Then we are to prove that
 \[
 \psi _0\left( \Tr _{k/{\bbF _2}} \Arf \left( Q_{\nu}|_{I_k(\Gamma ; d, n)}\right) \right)={q\overwithdelims () {n/d}}.
 \]
 However, as $Q_{\nu}$ is clearly defined over $\bbF _2$
 the computation is reduced to $k=\bbF _2$ case.
 
 Assume therefore $k=\bbF _2$.
 Let $m>1$ be a divisor of $n$
 and set
 \[
 V_m=\bigoplus V_{\chi},
 \] 
 where the sum is taken over 
 all the orbits $\Omega \chi \in \Omega \backslash \Gamma ^{\vee}$ 
 of the characters of order $m$.
 Note that $\dim V_m=\varphi (m)$,
 where $\varphi$ is Euler's totient function.
 Take a prime divisor $l$ of $m$.
 Suppose that a non-degenerate quadratic form $Q$ on $V_m$ 
 is invariant under the action of $\Gamma$.
 Now every $\Gamma$-orbit of $V_m$
 except for $\{ 0\}$ 
 is of length $m$.
 Thus if $\Arf (Q)=0$ (or $=1$), 
 then $2^{\varphi (m)/2-1}+2^{\varphi (m)-1} \equiv 1\bmod{m}$
 (resp.\ $\equiv 0\bmod{m}$)
 by Proposition \ref{Prop:Arfbycount}.
 Since the two congruences in the latter condition never occur together 
 and $\Arf (Q)\in \{ 0, 1\}$,
 the two conditions are in fact equivalent.
 The congruences are further equivalent to
 $2^{\varphi (m)/2} \equiv 1\bmod{m}$
 (resp.\ $\equiv -1\bmod{m}$).
 We find that 
 these in turn are equivalent to 
 the same congruence $\bmod{l}$,
 again by observing that 
 $2^{\varphi (m)/2}$ can only be congruent to
 $1$ or $-1 \bmod{m}$.
 Now an elementary calculation shows that
 $2^{\varphi (m)/2} \equiv 2^{(l-1)/2}\bmod{l}$
 if $m$ is a prime power and 
 $2^{\varphi (m)/2} \equiv 1\bmod{l}$
 otherwise.
 Therefore,
 \[
 \psi _0\left( \Arf (Q)\right)=\begin{cases}
                                      {2\overwithdelims ()l_m} &\text{ if $m$ is a prime power} \\
                                      1 &\text{otherwise},
                                      \end{cases}
 \]
 where $l_m$ is the unique prime factor of $m$,
 from which we conclude 
 \begin{align*}
 \psi _0\left( \Arf \left( Q_{\nu}|_{I_k(\Gamma ; d, n)}\right) \right) 
 &=\prod_{\substack{ m\mid n \\ m\nmid d}} \psi _0\left( \Arf \left( Q_{\nu}|_{V_m}\right) \right) \\
 &=\prod_{\substack{ m\mid n \\m\nmid d \\ \text{power of a prime $l_m$}}} {2\overwithdelims () {l_m}}
 ={2\overwithdelims () {n/d}}
 \end{align*}
 as desired.
 
 Finally, let us prove \ref{item:oddcyclic}.
 As in \cite[\S4.4]{BWGeom} we apply \cite[Theorem 3.2]{DeLu}
 to deduce 
 \[
 \sum _{i}(-1)^i\tr \left( \gamma^j x \mid H_c^i
                       \right)
 =\sum _{i}(-1)^i\tr \left( x \mid H_c^i(Z_{\nu}^{\gamma^j}, \overline{\bbQ} _{\ell})\right),
 \]
 where $x\in k$ and 
 $Z_{\nu}^{\gamma^j}$ denotes the fixed point variety 
 with respect to the action of $\gamma^j \in \Gamma$.
 Since $Z_{\nu}^{\gamma^j}$ is clearly a discrete set of points indexed by $k$, 
 the right-hand side equals the trace of the regular representation of $k$.
 Now we find the required action of $\gamma^j$ 
 by applying the idempotents of the group ring
 corresponding to each characters of $k$.
 The action of $\beta \in {\bbZ}/{2\bbZ}$ is treated in exactly the same way
 because the fixed point variety remains the same 
 (unless $p=2$, in which case the statement is trivial).  
 \end{proof}
 
 \begin{Prop} \label{Prop:evencoh}
 Let $1\leq \nu < 2n$ be an integer. 
 Suppose that $\nu$ is even
 and write $\nu =2\mu.$
 Assume that $\nu$ and $n$ are coprime.
 Let $Z_{\nu}$ be the algebraic variety defined in Theorem \ref{Thm:reduction},
 which we regard as the base change of $Z_{\nu, 0}$ as in Subsection \ref{subsec:alggps}.
 Put 
 \[
 H_c^i=H_c^i(Z_{\nu}, \overline{\bbQ} _{\ell}),
 \]
 which carries the actions of $S_{1, \nu}\times S_{2, \nu}$
 (defined in Subsection \ref{subsec:stabandaction}) and $\Omega$.
 \begin{enumerate}[(1)]
         \item \label{item:evenvanish}
         We have
                 $H_c^i=0$
                 unless $i=2(n-1), n-1.$
         \item \label{item:evennondeg}
                 We have
                 \[
                 H_c^{n-1}\simeq \bigoplus _{\psi \in k^{\vee}\setminus \{ 1\}} 
                                       \rho _{1, \psi}\boxtimes \rho _{2, \psi}
                 \]       
                 as a representation of $S_{1, \nu}\times S_{2, \nu}$,
                 where $\rho _{1, \psi}$ $($resp.\ $\rho _{2, \psi})$ is 
                 the unique irreducible representation of 
                 $S_{1, \nu}$ $($resp.\ of $S_{2, \nu})$ 
                 with the central character $\psi$
                 (cf.\ Proposition \ref{Prop:ourHeisenberg}),
                 and
                 \[
                 \tr \left( \Frob _q \mid 
                              \rho _{1, \psi}\boxtimes \rho _{2, \psi}
                      \right) 
                 =q^{n-1}.       
                 \]
         \item \label{item:evencyclic}
                 Consider the following natural action of the standard generator $\gamma \in \Gamma ={\bbZ}/{n\bbZ}$ 
                 on $Z_{\nu};$
                 \[
                 (z, y_1, y_2, \dots , y_n)\mapsto (z, y_n, y_1, \dots , y_{n-1})
                 \]
                 Let $\psi$ be a character of $k$
                 and $H_{c, \psi}^{\bullet}$ the $\psi$-isotypic component of $\bigoplus _{i}H_c^i$.
                 Then we have
                 \[
                 \tr \left( \gamma ^{j}\mid H_{c, \psi}^{\bullet}\right)=1
                 \]
                 for any $j$ coprime to $n$.                      
 \end{enumerate}
 \end{Prop}
 
 \begin{proof}
 Let us prove the assertions \ref{item:evenvanish}, \ref{item:evennondeg}.
 As the case where $n<\nu <2n$ is settled in exactly the same way,
 we only treat the case $\nu <n$.
 We denote by $\widetilde{P}_{\nu}(y_1, \dots , y_n)$
 the polynomial appearing 
 in the right-hand side of the second equation
 in Theorem \ref{Thm:reduction} \ref{item:red:smalleven}.
 We put $P_{\nu}(y_2, \dots , y_n)=\widetilde{P} _{\nu}(-(y_2+\dots +y_n), y_2, \dots , y_n)$.
 Then for any $i$ we have the following decomposition:
 \[
 H_c^i\simeq \bigoplus _{\psi \in k^{\vee}} H_c^i(\bbA _{\overline{k}}^{n-1}, P_{\nu}^{\ast}\scrL_{\psi})
 \]
 as a representation of $S_{1, \nu}\times S_{2, \nu}\times \Omega$. 
 It suffices to prove 
 \[
 \dim H_c^i(\bbA _{\overline{k}}^{n-1}, P_{\nu}^{\ast}\scrL_{\psi})=\begin{cases}
                                                                                     q^{n-1}&\text{if $i=n-1$} \\
                                                                                     0&\text{otherwise}
                                                                                    \end{cases} 
 \]
 for any non-trivial $\psi \in k^{\vee}$.
 Indeed, 
 $\psi =1$ non-trivially contributes to the above decomposition 
 only if $i=2(n-1)$,
 and 
 we have
 $\dim \rho _{1, \psi}=\dim \rho _{2, \psi}=q^{(n-1)/2}$
 if $\psi$ is non-trivial
 by Proposition \ref{Prop:Heisenberg}.
 Also, as $\psi (x^q-x)=1$ for any $x\in k$ and any $\psi \in k^{\vee}$,
 the statement for the Frobenius trace immediately follows from 
 the Grothendieck-Lefschetz trace formula
 and the above vanishing.
 
 Our basic strategy is to apply \cite[(3.7.2.3)]{DelWeilII}\footnote{In fact, it also asserts that the cohomology in degree $m$ is pure of weight $m$.
                                                         However, we only need the dimension assertions in what follows.}:
 \begin{quote}
 Let $P\in k[T_1, \dots , T_m]$ be a polynomial of degree $d$.
 Suppose that $d$ is coprime to $p$
 and that the homogeneous part $P^{(d)}$ of degree $d$ of $P$
 defines a smooth hypersurface in $\bbP _k^{m-1}$.
 Then
 \[
 \dim H_c^i(\bbA _{\overline{k}}^m, P^{\ast}\scrL_{\psi})=\begin{cases}
                                                                         (d-1)^m&\text{if $i=m$} \\
                                                                         0&\text{otherwise}. 
                                                                        \end{cases}
 \]                                                                       
 \end{quote}
 
 Although the polynomial $P_{\nu}$ is of degree $2q$,
 we may replace each monomial of the form $y_i^qy_j^q$
 with $y_iy_j$,
 because $f^{\ast}\scrL_{\psi}$ is a constant sheaf
 if $f=g^q-g$ for some polynomial $g$.
 We denote by $P' _{\nu}\in k[y_2, \dots , y_n]$
 the polynomial obtained by applying the above procedures
 to all monomials of the form $y_i^qy_j^q$.
 We similarly denote by $\widetilde{P}'_{\nu}\in k[y_1, \dots , y_n]$
 the polynomial obtained from $\widetilde{P} _{\nu}$ in the same way.
 Then we have $\deg P' _{\nu}=\deg \widetilde{P}'_{\nu}=q+1$
 and $P'_{\nu}(y_2, \dots , y_n)=\widetilde{P}' _{\nu}(-(y_2+\dots +y_n), y_2, \dots , y_n)$.
 Thus it suffices to show that 
 ${P'}^{(q+1)}_{\nu}$ defines a smooth hypersurface in $\bbP _{\overline{k}}^{n-2}$.
 As this hypersurface is isomorphic to 
 the projective variety $V$ defined by $\widetilde{P}'^{(q+1)} _{\nu}$ and $y_1+\dots +y_n$
 in $\bbP _{\overline{k}}^{n-1}$,
 we are reduced to proving that the Jacobian
 \[
 \begin{pmatrix}
 1&1&\dots &1 \\
 \frac{\partial}{\partial y_1}\widetilde{P}'^{(q+1)} _{\nu}& \frac{\partial}{\partial y_2}\widetilde{P}'^{(q+1)} _{\nu}&\dots &  \frac{\partial}{\partial y_n}\widetilde{P}'^{(q+1)} _{\nu}
 \end{pmatrix}
 \]
 is of rank two at every $\overline{k}$-valued point $[Y_1: \dots :Y_n]$ of $V$.
 Regarding $\{ y_i\}$ as indexed by ${\bbZ}/{n\bbZ}$
 we easily verify
 that 
 \[
 \widetilde{P}'^{(q+1)} _{\nu}(y_1, \dots , y_n)=-\sum _{1\leq i\leq n}y_i\sum_{\mu \leq d<n-\mu}y_{i+d}^q
 \]
 and hence 
 \[
 \frac{\partial}{\partial y_i}\widetilde{P}'^{(q+1)} _{\nu}(y_1, \dots , y_n)=-\sum _{\mu \leq d<n-\mu}y_{i+d}^q.
 \]
 The rank of the Jacobian is not maximal
 if and only if these partial derivatives are all equal,
 that is,
 $Y_{i}^q=Y_{j}^q$
 for all $1\leq i, j\leq n$,
 by the assumption that $n$ and $\nu=2\mu$ are coprime.
 Together with $Y_1+\dots +Y_n=0$,
 this implies $Y_1=\dots =Y_n=0$, as required.
 
 Now the assertion \ref{item:evencyclic} follows from \cite[Theorem 3.2]{DeLu}
 exactly as in Proposition \ref{Prop:oddcoh} \ref{item:oddcyclic}. 
 \end{proof}
 \begin{Rem}
 Although we omit details here
 the author computed the cohomology groups $H_c^{i}(Z_{\nu}, \overline{\bbQ} _{\ell})$
 even when $\nu$ is even and not coprime to $n$
 (cf.\ Remark \ref{Rem:Lang}).
 In particular, it can be shown that $H_c^{n-1}(Z_{\nu}, \overline{\bbQ} _{\ell})$
 is non-trivial if and only if $n$ and $\nu$ are coprime,
 which is in parallel with the situation for odd $\nu$ cases.
 \end{Rem}
 
 Given the preceding propositions,
 the cohomology of the reduction $\overline{\scrZ}_{\nu}$ is computed 
 by exploiting the periodicity of $Z_{\nu}$ with respect to $\nu$ and 
 the following proposition.
 
 \begin{Prop} \label{Prop:redcoh}
 Let $\nu >0$ be an integer,
 not divisible by $n$.
 Then we have the following isomorphism
 \[
 H_c^{i}(\overline{\scrZ}_{\nu}, \overline{\bbQ} _{\ell})
 \simeq
 \bigoplus _{\chi \in \left( U_K^{\lceil \nu /n\rceil} \right)^{\vee}} 
 H_c^{i}(Z_{\nu}, \overline{\bbQ} _{\ell})\otimes (\chi \circ N_G).
 \]
 of representations of $\Stab _{\nu}$
 for any $i$.
 Here $\left( U_K^{\lceil \nu /n\rceil} \right)^{\vee}$ denotes the group of smooth characters of $U_K^{\lceil \nu /n\rceil}$.
 \end{Prop}
 \begin{proof}
 The proposition follows from 
 Theorem \ref{Thm:stabandaction} \ref{item:stabinducesaction}, \ref{item:actiononcomps}
 in the same way as 
 \cite[Corollary 3.6.2]{BWMax}.
 \end{proof}

\section{Realization of the correspondences} \label{sec:real}
 \subsection{Special cases of the essentially tame local Langlands and Jacquet-Langlands correspondences} \label{subsec:etc}
 Let us allow $n$ to be divisible by $p$ 
 \emph{only in this subsection}.
 
 \begin{Def} (\cite[1, 2]{BHetLLCI})
 Let $\rho$ be an $n$-dimensional irreducible smooth representation of $W_K$. 
 Let $t(\rho)$ be the number of unramified characters $\chi$ of $K^{\times}$
 such that $\chi \otimes \rho \simeq \rho$.
 Then $t(\rho)$ divides $n$ and 
 $\rho$ is said to be \emph{essentially tame}
 if $p$ does not divide $n/{t(\rho)}$.
 
 We denote by $\calG _n^{\text{et}}(K)$ the set of isomorphism classes
 of $n$-dimensional (irreducible) essentially tame representations of $W_K$.
 
 Similarly, let $\pi$ be an irreducible supercuspidal representation of $\GL _n(K)$
 or an irreducible smooth representation of $D^{\times}$.
 We say that $\pi$ is \emph{essentially tame}
 if $p$ does not divide $n/{t(\pi)}$,
 where $t(\pi)$ is the number of unramified characters $\chi$ of $K^{\times}$
 such that $\chi \pi \simeq \pi$. 
 \end{Def}
 
 \begin{Def} (\cite[p.\ 437]{HoTame}; see also \cite[3.\ Definition]{BHetLLCI})
 \emph{An admissible pair} (of degree $n$)
 is a pair $(F/K, \xi)$
 in which $F/K$ is a tamely ramified extension of degree $n$
 and $\xi$ is a character of $F^{\times}$
 such that \begin{enumerate}[(1)]
              \item if $\xi$ factors through the norm map $N_{F/E}\colon F^{\times}\rightarrow E^{\times}$
              for a subextension $K\subset E\subset F$,
              then $F=E$;
              \item if $\xi |_{U_F^1}$ factors through $N_{F/E}\colon F^{\times}\rightarrow E^{\times}$ 
              for a subextension $K\subset E\subset F$,
              then $F/E$ is unramified.
              \end{enumerate}
 Two admissible pairs $({F_1}/K, \xi _1)$, $({F_2}/K, \xi _2)$
 are said to be $K$-isomorphic
 if there exists a $K$-isomorphism $i: F_1\xrightarrow{\sim} F_2$ 
 such that $\xi _1=\xi _2\circ i$.
 
 We denote by $P_n(K)$ the set 
 of $K$-isomorphism classes of admissible pairs of degree $n$.
 \end{Def}
 
 The proof of the following is found in \cite[A.3 Theorem]{BHetLLCI}.
 \begin{Prop} 
 The following map is a bijection$:$
 \[
 P_n(K)\rightarrow \calG _n^{\emph{et}}(K); \ (F/K, \xi)\mapsto \Ind _{F/K}\xi =\Ind _{W_F}^{W_K}\xi. 
 \]
 \end{Prop}

 In \cite{BHetLLCI}, \cite{BHetLLCII}, \cite{BHetLLCIII},
 a canonical bijection $(F/K, \xi) \mapsto \pi _{\xi}$
 between $P_n(K)$ and the set of isomorphism classes of
 essentially tame representations of $\GL _n(K)$
 is constructed, 
 and the existence and an explicit description of 
 tamely ramified characters ${}_{K}\mu _{\xi}$ of $F^{\times}$
 are established 
 such that $\Ind _{F/K}\xi \mapsto \pi _{{}_{K}\mu _{\xi}\xi}$
 is the local Langlands correspondence.
 Likewise, a special case of the main result of \cite{BHetJL}
 yields a canonical injection $(F/K, \xi)\mapsto \pi _{\xi}^D$ 
 from $P_n(K)$ to the set of isomorphism classes of
 essentially tame representations of $D^{\times}$,
 and the description of tamely ramified characters ${}_{D}\iota _{\xi}$ of $F^{\times}$
 such that $\pi _{{}_{D}\iota _{\xi}\xi}\mapsto \pi _{\xi}^D$
 is the local Jacquet-Langlands correspondence.
 
 In what follows,
 we review the construction of $\pi _{\xi}$ and $\pi _{\xi}^D$
 and the description of ${}_{K}\mu _{\xi}$ and ${}_{D}\iota _{\xi}$
 for certain admissible pairs $(F/K, \xi)$
 that are relevant to our results.
 
 \paragraph{Minimal pairs}
 \begin{Def} \label{def:minimalpair}
 Let $i\geq 0$ be an integer.
 An admissible pair $(F/K, \xi)$ 
 is said to be \emph{minimal with the jump at $i$}
 if $\xi |_{U_F^{i+1}}$ factors through the norm map $N_{F/K}$
 and $\xi |_{U_F^{i}}$ does not factor through $N_{F/E}$
 for any subextension $K\subset E\subsetneq F$.
 
 We say that an admissible pair is \emph{minimal}\footnote{
                                                                             Note that some authors further impose 
                                                                             the triviality of $\xi |_{U_F^{i+1}}$
                                                                             in the definition of minimality.
                                                                             This definition is
                                                                             taken from \cite[\S2.2]{BHetLLCI}
                                                                             (except that $i$ is assumed to be positive there).
                                                                             They also discuss jumps of 
                                                                             possibly not minimal pairs.
                                                                             }
 if it is minimal with the jump at $i$
 for some $i\geq 0$.
 \end{Def}
 
 \begin{Rem} \label{Rem:minimalpairs}
 \begin{enumerate}[(1)]
  \item If $n$ is a prime, then any admissible pairs are minimal.
  \item If $(F/K, \xi)$ is a minimal pair with the jump at $i$,
  then there exists a decomposition
  \begin{equation} \label{eq:decompofadmpair}
  \xi =\left( \varphi \otimes (\chi \circ N_{F/K}) \right),
  \end{equation}
  where $\varphi$ is a character of $F^{\times}$ trivial on $U_F^{i+1}$
  and $\chi$ is a character of $K^{\times}$.
  \item If $(F/K, \xi)$ is a minimal pair with the jump at $i$, 
  then the ramification index of $F/K$ is coprime to $i$.
  \item In this paper
  we are interested in 
  minimal pairs $(F/K, \xi)$ with the jump at $\nu$
  in which $F/K$ is totally ramified 
  (hence, $p\nmid n$ and $\nu$ is coprime to $n$).
  We will see that the cohomology of 
  each reduction $\overline{\scrZ} _{\nu}$, with $\nu$ coprime to $n$, realizes 
  the local Langlands and Jacquet-Langlands correspondences
  for representations parametrized by such minimal pairs
  (for a specific $F$).
  \item \label{item:precedingresults}
  Many of the preceding results 
  treat the representations parametrized by minimal pairs;
  \begin{itemize}
   \item The cohomology of each reduction in \cite{BWMax} 
   deals with minimal pairs $(F/K, \xi)$ with the jump at some $i\geq 1$
   in which $F/K$ is unramified.
   \item That of \cite{ITepitame}
   deals with minimal pairs $(F/K, \xi)$ with the jump at $1$
   in which $F/K$ is totally ramified.
   These representations are exactly 
   character twists of
   \emph{simple supercuspidal representations}
   if $p$ does not divide $n$.
   In \cite{ITsimpwild} Imai and Tsushima 
   also treated the cases where $p$ divides $n$.
   In these cases 
   simple supercuspidal representations 
   are not essentially tame.
   \item In \cite{WeSemi},
   it is assumed that $n=2$ and $p\neq 2$,
   in which case all representations involved are
   parametrized by minimal pairs.
  \end{itemize} 
  \end{enumerate}
 \end{Rem}
 
 \paragraph{Construction of $\pi _{\xi}$ and $\pi _{\xi}^D$ in special cases}
 In the rest of this subsection
 we assume that $(F/K, \xi)$ is a minimal pair with the jump at $i$
 and $F/K$ is totally ramified.
 
 Fix a character $\psi$ of $K$ \footnote{Thus, we change notation here;
                                                    in Section \ref{sec:cohofred}
                                                    $\psi$ generally denotes characters of $k$.} 
 which is trivial on $\frakp$, but not on $\calO _K$.
 For $\alpha \in F$, define a function $\psi _{\alpha}^F$ by 
 $\psi _{\alpha}^F(u)=\psi (\Tr _{F/K}\alpha (u-1))$,  $(u\in F).$
 Then by the tame ramification assumption
 we have
 \begin{equation} \label{eq:paramofchars}
 {\frakp _F^{-s}}/{\frakp _F^{-r}}\xrightarrow{\sim} ({U_F^{r+1}}/{U_F^{s+1}})^{\vee}; \ 
 \alpha +\frakp _F^{-r}\mapsto \psi _{\alpha}^F
 \end{equation} 
 for any integers $r, s$ such that $0\leq r<s\leq 2r+1$.
 Similarly, for $\beta \in M_n(K)$ and $\gamma \in D$,
 set
 \[
 \psi _{\beta}(g)=\psi (\tr \beta (g-1)), \ (g\in M_n(K)) \text{ and }
 \psi _{\gamma}^D(d)=\psi (\Trd \gamma (d-1)), \ (d\in D).
 \]
 
 Let us construct $\pi _{\xi}$ and $\pi _{\xi}^D$.
 
 First suppose that $\xi$ is trivial on $U_F^{i+1}$.
 Then there exists an $\alpha \in F$ with $v_F(\alpha)=-i$ 
 such that 
 $\xi |_{U_F^{\lfloor i/2\rfloor +1}}=\psi _{\alpha}^F$.
 Take a $K$-embedding $F\rightarrow M_n(K)$ (resp.\ $F\rightarrow D$)
 and regard $F$ as a $K$-subalgebra $F\subset M_n(K)$ (resp.\ $F\subset D$). 
 Let $\frakI=\frakI _{\xi}\subset M_n(K)$ be the unique hereditary $\calO _K$-order 
 normalized by $F^{\times}$
 (later we will always arrange it to be the standard Iwahori order;
 see Remark \ref{Rem:connwiththegeom}).
 Denote by $\frakP _{\frakI}=\rad \frakI$ the Jacobson radical of $\frakI$
 and set $U_{\frakI}=\frakI ^{\times}$, 
 $U_{\frakI}^i=1+\frakP _{\frakI}^i$ for $i\geq 1$
 as usual.
 Similarly let $U_D^i$ be the subgroup defined after Proposition \ref{Prop:gpS1} for $i\geq 1$.
 Define a character $\theta _{\xi}$ (resp.\ $\theta _{\xi}^D$)
 of $H_{\xi}^1=U_F^1U_{\frakI}^{\lfloor i/2\rfloor +1}$
 (resp.\ of $H_{\xi}^{1, D}=U_F^1U_D^{\lfloor i/2\rfloor +1}$)
 by
 \begin{align*}
 &\theta _{\xi}|_{U_F^1}=\xi |_{U_F^1}, \quad 
  \theta _{\xi}|_{U_{\frakI}^{\lfloor i/2\rfloor +1}}=\psi _{\alpha}|_{U_{\frakI}^{\lfloor i/2\rfloor +1}}, \\
 &\theta _{\xi}^D|_{U_F^1}=\xi |_{U_F^1}, \quad
  \theta _{\xi}^D|_{U_D^{\lfloor i/2\rfloor +1}}=\psi _{\alpha}^D|_{U_D^{\lfloor i/2\rfloor +1}}.
 \end{align*}
 Set $J_{\xi}^1=U_F^1U_{\frakI}^{\lfloor (i+1)/2\rfloor}$, $J_{\xi}^{1, D}=U_F^1U_D^{\lfloor (i+1)/2\rfloor}$,
 $J_{\xi}=F^{\times}J_{\xi}^1$ and $J_{\xi}^D=F^{\times}J_{\xi}^{1, D}$.
 To construct an irreducible smooth representation $\Lambda _{\xi}$ (resp.\ $\Lambda _{\xi}^D$)
 of $J_{\xi}$ (resp.\ of $J_{\xi}^D$),
 we use the following lemma. 

 \begin{lem} \label{lem:constoflambda}
 Let $\theta =\theta _{\xi}$ or $\theta=\theta _{\xi}^D$.
 Accordingly, set $H^1=H_{\xi}^1$ (resp.\ $H_{\xi}^{1, D}$),
 $J^1=J_{\xi}^1$ (resp.\ $J_{\xi}^{1, D}$)
 and $J=J_{\xi}$ (resp.\ $J_{\xi}^D$).
 \begin{enumerate}[(1)]
 \item The conjugation by $F^{\times}$ 
 stabilizes $\theta$.
 Thus the cyclic group $\Gamma ={F^{\times}}/{K^{\times}U_F^1}$
 acts on the finite $p$-group $\calQ ={J^1}/{\Ker \theta}$.
 The center $\calZ$ of $\calQ$ is the cyclic group $\calZ =H^1/{\Ker \theta}$,
 which is also the $\Gamma$-fixed part $\calZ =\calQ ^{\Gamma}$.
 \item There exists a unique irreducible smooth representation $\eta$ of $\calQ$
 whose central character is $\theta$.
 \item There exists a unique irreducible smooth representation $\tilde{\eta}$ of $\Gamma \ltimes \calQ$
 such that $\tilde{\eta}|_{\calQ}\simeq \eta$ and $\det \tilde{\eta} |_{\Gamma}=1$.
 \item There exists a constant $\epsilon \in \{ \pm 1\}$ 
 such that $\tr \tilde{\eta} (\gamma u)=\epsilon \xi (u)$ 
 for any generator $\gamma \in \Gamma$ and 
 $u\in U_F^1$.
 \item There exists a unique irreducible smooth representation $\Lambda$ of $J$
 such that $\Lambda |_{J^1}\simeq \eta$ and
 \begin{equation} \label{eq:charreloflambda}
 \tr \Lambda (h)=\epsilon \xi (h)
 \end{equation}
 for any $h\in F^{\times}$ whose image in $\Gamma$ is a generator.
 \item \label{item:isotypic}
 The restriction $\Lambda|_{K^{\times}U_F}$ is $\xi|_{K^{\times}U_F}$-isotypic.
 \end{enumerate}
 \end{lem}
 
 \begin{proof}
 All the assertions except for \ref{item:isotypic} follow from \cite[(4.1.1) and Lemma 4.1]{BHetLLCI} 
 and \cite[\S5.2 Lemma 1]{BHetJL},
 where the construction of $\Lambda$ using $\tilde{\eta}$ is given.
 The assertion \ref{item:isotypic} is a consequence of \eqref{eq:charreloflambda};
 since $K^{\times}U_F=K^{\times}U_F^1\subset K^{\times}H^1$,
 the restriction $\Lambda|_{K^{\times}U_F}$ is $\xi'$-isotypic for some character $\xi'$
 and one can show $\xi'=\xi|_{K^{\times}U_F}$ by comparing \eqref{eq:charreloflambda} for $h$ and $uh$ with $u\in K^{\times}U_F$.
 
 Note that the statements are trivial if $i$ is odd,
 in which case $H^1=J^1$.
 In fact, then $\Lambda$ is one-dimensional,
 $\Lambda |_{F^{\times}}=\xi$ and $\epsilon =1$.
 Note also that the existence of $\eta$ (if $i$ is even)
 is a consequence of Proposition \ref{Prop:Heisenberg}.
 \end{proof}
 
 According to whether $\theta =\theta _{\xi}$ or $\theta =\theta _{\xi}^D$,
 we denote the sign $\epsilon$ appearing in the proposition
 by $\epsilon _{\xi}$ (resp.\ $\epsilon _{\xi}^D$)
 and similarly denote the representation $\Lambda$
 by $\Lambda _{\xi}$ (resp.\ $\Lambda _{\xi}^D$).
 We set
 \[
 \pi _{\xi}=\cInd _{J_{\xi}}^{\GL _n(K)}\Lambda _{\xi}, \quad \pi _{\xi}^D=\Ind _{J_{\xi}^D}^{D^{\times}}\Lambda _{\xi}^D.
 \]
 
 Finally, if $\xi$ is not trivial on $U_F^{i+1}$,
 we take a decomposition of $\xi$ as in (\ref{eq:decompofadmpair})
 and put
 \[
 \pi _{\xi}=\chi\pi _{\varphi}, \quad \pi _{\xi}^D=\chi\pi _{\varphi}^D.
 \]
 The representation $\pi _{\xi}$ (resp.\ $\pi _{\xi}^D$) is irreducible and supercuspidal (resp.\ irreducible).
 The isomorphism classes of $\pi _{\xi}$ and $\pi _{\xi}^D$ 
 only depend on the $K$-isomorphism class of $(F/K, \xi)$. 

 We need an explicit description of the signs $\epsilon _{\xi}$ and $\epsilon _{\xi}^D$.
 
 \begin{Prop} \label{Prop:sympsign}
 Suppose that $i$ is even.
 In the situation of Lemma \ref{lem:constoflambda}
 the sign $\epsilon$ equals the Jacobi symbol
 \[
 \epsilon ={q\overwithdelims ()n}.
 \]
 \end{Prop}
 
 \begin{proof}
 By \cite[(8.6.1)]{BFGdiv}, 
 the sign $\epsilon$ is
 determined by 
 the symplectic representation
 $(\calQ /\calZ, h_{\theta})$
 of $\Gamma$
 induced by $h_{\theta} \colon (x, y)\mapsto \theta (xyx^{-1}y^{-1})$
 and hence by the $k[\Gamma]$-module $\calQ /\calZ$ (cf.\ Proposition \ref{Prop:symprep}).
 Also it is multiplicative with respect to orthogonal sums of symplectic representations.
 We have 
 $\calQ /\calZ \simeq I_k(\Gamma ; 1, n)\simeq I_{\bbF _p}(\Gamma ; 1, n)\otimes _{\bbF _p}k$,
 no matter whether $\theta =\theta _{\xi}$ or $\theta =\theta _{\xi}^D$.
 Therefore, the assertion is reduced to $k=\bbF _p$ case,
 which is treated in \cite[(9.3.5)]{BFGdiv}
 \end{proof}
  
 \begin{Rem} \label{Rem:connwiththegeom}
  Let us temporarily return 
  to the situation of Theorem \ref{Thm:stabandaction}.
  There $L=K(\varphi _L)$ is a totally tamely ramified extension of $K$
  of degree $n$
  and it is considered as a $K$-subalgebra of
  $M_n(K)$ (resp.\ $D$)
  via a fixed embedding $\varphi _L\mapsto \varphi$ 
  (resp.\ of $\varphi _L\mapsto \varphi _D$)
  arising from the fixed CM point.
  It is easily seen that $L^{\times}$
  does normalize the standard Iwahori order,
  which is denoted there
  again by $\frakI \subset M_n(K)$.
  We apply the preceding constructions
  with respect to this field, these embeddings and this order.
  Note also the equalities $L^{\times}U_{\frakI}^{(\nu)}=L^{\times}U_{\frakI}^{\lfloor (\nu +1)/2\rfloor}$
  and $L^{\times}\cap U_{\frakI}^{(\nu)}=U_{L}^{\nu}$,
  and the analogous equalities for $U_D^{(\nu)}$.
 \end{Rem}
 
 \paragraph{Description of ${}_{K}\mu _{\xi}$ and ${}_{D}\iota _{\xi}$ in special cases}
 Let us first define some invariants 
 attached to $F/K$, $\psi$ and $\xi$.
 We set
 \[
 R_{F/K}=\Ind _{F/K}1_{K}, \quad \delta _{F/K}=\det R_{F/K},
 \]
 where $1_{K}$ denotes the trivial representation of $W_K$.
 We define the Langlands constant $\lambda _{F/K}(\psi )$ by
 \[
 \lambda _{F/K}(\psi )=\frac{\eps (R_{F/K}, 1/2, \psi)}{\eps (1 _{F}, 1/2, \psi \circ \Tr _{F/K})},
 \]
 where the denominator and the numerator denote the Langlands-Deligne local constants
 (see \cite[\S30]{BHLLCGL2} for these two constants).
 
 Take $\alpha =\alpha (\xi)\in F$ as in the construction of $\pi _{\xi}$,
 so that $v_F(\alpha)=-i$
 and 
 $\varphi |_{U_F^{\lfloor i/2\rfloor +1}}=\psi _{\alpha}^F$
 where $\varphi$ is as in the decomposition (\ref{eq:decompofadmpair}).
 Note that $\alpha (\xi) U_F^1$ only depends on $\xi$ 
 by (\ref{eq:paramofchars}).
 For any uniformizer $\varpi _F\in F$,
 we define $\zeta (\varpi _F,\xi)\in F$ as the unique root of unity satisfying
 \[
 \zeta (\varpi _F,\xi)\equiv \varpi _F^i\alpha (\xi) \pmod{U_F^1}.
 \]
 
 In our case,
 \cite[Theorem 2.1]{BHetLLCII}
 reads as follows.
 \begin{Thm} \label{Thm:GLrect}
 Let $(F/K, \xi)$ be an admissible pair as above,
 i.e. it is minimal with the jump at $i$ 
 and $F/K$ is totally ramified. 
 Then the image of $\Ind _{F/K}\xi$
 under the local Langlands correspondence
 is $\pi _{{}_{K}\mu _{\xi}\xi}$,
 where ${}_{K}\mu _{\xi}$ is 
 a character
 of $F^{\times}$ defined below.
 \begin{enumerate}[(1)]
 \item If $n$ is odd, 
 then ${}_{K}\mu _{\xi}$ is unramified and, for any uniformizer $\varpi _F\in F$, 
 \[
 {}_{K}\mu _{\xi}(\varpi _F)=\lambda _{F/K}(\psi ).
 \]
 \item If $n$ is even,
 then ${}_{K}\mu _{\xi}$ is determined by the following conditions
 \begin{align*}
 &{}_{K}\mu _{\xi}|_{U_F^1}=1, \quad {}_{K}\mu _{\xi}|_{K^{\times}}=\delta _{F/K}, \\
 &{}_{K}\mu _{\xi} (\varpi _F)={\zeta (\varpi _F,\xi) \overwithdelims ()k}
                                     {-1 \overwithdelims ()k}^{(i-1)/2}
                                     \lambda _{F/K}(\psi )
 \end{align*}
 for any uniformizer $\varpi _F\in F$.
 \end{enumerate}
 \end{Thm}
 \begin{Rem}
 In both cases, 
 ${}_{K}\mu _{\xi}$ does not depend on the choice of $\psi$
 (see \cite[Remark 2.1.3]{BHetLLCII}).
 \end{Rem}
 
 Similarly, we need the following special case of \cite[Theorem 5.3]{BHetJL}. 
 \begin{Thm} \label{Thm:Drect}
 Let $(F/K, \xi)$ be as above.
 Then the image of $\pi _{\xi}^D$ 
 under the local Jacquet-Langlands correspondence is 
 $\pi _{{}_{D}\iota _{\xi}\xi}$,
 where $\iota ={}_{D}\iota _{\xi}$ is the unramified character of $F^{\times}$
 sending uniformizers to $(-1)^{n-1}$.
 \end{Thm}
 
 We record here some of the explicit values
 used later.
 \begin{Prop} \label{Prop:valuesforinvs}
 Let the notation be as above.
  \begin{enumerate}[(1)]
   \item \label{item:valueforoddlambda}
   Suppose that $n$ is odd.
   Then 
   $
   \lambda _{F/K}(\psi)={q\overwithdelims ()n}.
   $
   \item \label{item:valuefordelta}
   Suppose that $n$ is even.
   Then
   $
   \delta _{F/K}(u)={{\overline{u}} \overwithdelims ()k}
   $
   for $u\in U_K$.
  \end{enumerate}
 \end{Prop}
 \begin{proof}
 The assertion \ref{item:valueforoddlambda} (resp.\ \ref{item:valuefordelta})
 is part of \cite[Lemma 1.5 (2)]{BHetLLCII}
 (resp.\ part of \cite[Lemma 5.3]{ITepitame}).
 \end{proof}
    
 \subsection{Realization of the correspondences} \label{subsec:real}
 In this subsection we finally prove Main Theorem in the introduction. The argument is inspired by that in \cite[\S5.2]{ITepitame}.
 
 Let $\psi$ be the additive character of $K$
 fixed in the previous subsection
 and denote by $\overline{\psi}$
 the non-trivial additive character of $k$ 
 obtained as the reduction of $\psi |_{\calO _K}$.
 We also denote by $\overline{\psi} _{\zeta}$
 the character
 $\overline{\psi} _{\zeta} \colon k\rightarrow \overline{\bbQ} _{\ell}^{\times}; \ x\mapsto \overline{\psi} (\overline{\zeta}x)$
 for any $\zeta \in \mu _{q-1}(K)$.
 
 Let $\nu >0$ be an integer
 and assume that it is coprime to $n$.
 We return to our analysis of the cohomology 
 in Section \ref{sec:cohofred}.
 Put $H_{\nu}=H_c^{n-1}(Z_{\nu}, \overline{\bbQ} _{\ell})((n-1)/2)$ 
 and $\Pi _{\nu}=H_c^{n-1}(\overline{\scrZ}_{\nu}, \overline{\bbQ} _{\ell})((n-1)/2)$.
 As in Proposition \ref{Prop:oddcoh} (2)
 (resp.\ Proposition \ref{Prop:evencoh} (2))
 we consider the standard action of $k$
 (resp.\ the action of the center $k$ of $S_{1, \nu}$)
 if $\nu$ is odd
 (resp.\ if $\nu$ is even)
 and
 denote by $H_{\nu, \zeta}$ 
 the $\overline{\psi} _{\zeta}$-isotypic component
 of $H_{\nu}$
 and set
 \[
 \Pi _{\nu, \zeta}=\bigoplus _{\chi \in \left( U_K^{\lceil \nu /n\rceil} \right)^{\vee}}
                      H_{\nu, \zeta}\otimes (\chi \circ N_G).
 \]
 Then we have 
 $H_{\nu}=\bigoplus _{\zeta \in \mu _{q-1}(K)} H_{\nu, \zeta}$ and
 $\Pi _{\nu}=\bigoplus _{\zeta \in \mu _{q-1}(K)} \Pi_{\nu, \zeta}$
 by  
 Proposition \ref{Prop:oddcoh} \ref{item:oddnondeg}, 
 Proposition \ref{Prop:evencoh} \ref{item:evennondeg}
 and Proposition \ref{Prop:redcoh}.
  
 \begin{lem} \label{lem:homfromind}
 Let $\pi$ be an irreducible smooth representation 
 of $GL_n(K)$.
 Set $G_1=GL_n(K)$ and $G_2=D^{\times}\times W_K$.
 Denote by $\overline{\Stab} _{\nu} \subset G_2$ 
 the image of $\Stab _{\nu}$
 under the projection $G\rightarrow G_2$.
 Then we have
 a canonical isomorphism
 \[
 \Hom _{G_1}
         \left( \cInd _{\Stab _{\nu}}^G
         \Pi _{\nu},
         \pi
         \right)
 \simeq
 \Ind _{\overline{\Stab} _{\nu}}^{G_2}
 \Hom _{U_{\frakI}^{(\nu)}}
         \left( \Pi _{\nu},
         \pi
         \right)
 \]
 of representations of $G_2$,
 where the action of  
 $(d, \sigma) \in \overline{\Stab} _{\nu}$
 on 
 $\Hom _{U_{\frakI}^{(\nu)}}
         \left( \Pi _{\nu},
         \pi
         \right)$
 is given by the composition
 of the action of $(g, d, \sigma)^{-1}\in \Stab _{\nu}$
 on the source
 and that of $g\in G_1$ on the target
 for some lift $(g, d, \sigma) \in \Stab _{\nu}$ 
 of $(d, \sigma) \in \overline{\Stab} _{\nu}$.         
 \end{lem}
 \begin{proof} 
 This is straightforward; 
 one only needs to check the action of $G_2$
 on the right-hand side of the following isomorphism
 \[
 \Hom _{G_1}
         \left( \cInd _{\Stab _{\nu}}^G
         \Pi _{\nu},
         \pi
         \right)
 \simeq
 \bigoplus _{{\overline{\Stab} _{\nu}} \backslash G_2}
 \Hom _{U_{\frakI}^{(\nu)}}
         \left( \Pi _{\nu},
         \pi
         \right)         
 \]
 induced by the Mackey decomposition
 \begin{align*}
 (\cInd _{\Stab _{\nu}}^G \Pi _{\nu})|_{G_1}&\simeq \bigoplus _{\Stab _{\nu}gG_1\in {\Stab _{\nu}} \backslash G/{G_1}}
                                                                \cInd _{\Stab _{\nu}^g\cap G_1}^{G_1} \Pi _{\nu}^g \\
                                                     &\simeq \bigoplus _{\overline{\Stab} _{\nu} \backslash {G_2}}
                                                                \cInd _{\Stab _{\nu} \cap G_1}^{G_1} \Pi _{\nu}  
 \end{align*}
 and the Frobenius reciprocity.
 \end{proof}
 \begin{Rem} \label{Rem:Stabastorsor}
 	Recall that 
 	we have $\Stab _{\nu}=(U_{\frakI}^{(\nu)}\times U_D^{(\nu)}\times \{ 1\})\cdot \calS$ by Theorem \ref{Thm:stabandaction} \ref{item:stabinducesaction}
 	and the stabilizer $\calS$ of the CM point is expressed by Proposition \ref{Prop:S} and \eqref{eq:stabofxi}.
 	In particular, 
 	\begin{itemize}
 		\item We have a natural short exact sequence:
 		\[
 		1 \to \Delta_{\xi}(L^{\times})\to \calS \to W_{L'}\to 1.
 		\]
 	    \item The pull-back of $W_L\subset W_{L'}$ is
 	    $(\Delta_{\xi}(L^{\times})\times \{ 1\}) 
 	    \cdot \{ (1, a_{\sigma}^{-1}, \sigma) \mid \sigma \in W_L\}$,
 	    where $a_{\sigma}=\Art_L^{-1}(\sigma)$ for $\sigma \in W_L$ as before. 
    \end{itemize}
 	Consequently,
 	\begin{itemize}
 		\item We have natural short exact sequences:
 		\begin{gather*}
 		1 \to (U_{\frakI}^{(\nu)}\times U_D^{(\nu)})\cdot \Delta_{\xi}(L^{\times})
 		\to \Stab _{\nu}\to W_{L'}\to 1, \\
 		1\to L^{\times}U_D^{(\nu)}=L^{\times}U_D^{\lfloor (\nu +1)/2\rfloor}\to \overline{\Stab} _{\nu}\to W_{L'}\to 1.
 		\end{gather*}
 		\item The pull-back of $W_{L}\subset W_{L'}$ to $\Stab_{\nu}$ (resp.\ $\overline{\Stab} _{\nu}$) is
 		$(U_{\frakI}^{(\nu)}\times U_D^{(\nu)}\times \{ 1\})
 		\cdot (\Delta_{\xi}(L^{\times})\times \{ 1\}) 
 		\cdot \{ (1, a_{\sigma}^{-1}, \sigma) \mid \sigma \in W_L\}$
 		(resp.\ $L^{\times}U_D^{\lfloor (\nu +1)/2\rfloor} \times W_{L}$).
 	\end{itemize}
    The action of the pull-back of $W_{L}\subset W_{L'}$ to $\Stab_{\nu}$ on $H_{\nu}$ is completely explicit (Theorem \ref{Thm:stabandaction} and Propositions \ref{Prop:oddcoh}, \ref{Prop:evencoh}) and accordingly we will first study the action of 
    $L^{\times}U_D^{\lfloor (\nu +1)/2\rfloor} \times W_{L}\subset \overline{\Stab} _{\nu}$ in Proposition \ref{Prop:homasrep} below.
 \end{Rem}
 \begin{Prop} \label{Prop:occurrencecond}
 Let $(\pi, V)$ be an irreducible smooth representation 
 of $GL_n(K)$.
         For $\zeta \in \mu _{q-1}(K)$ 
         we have
         \[
         \Hom _{U_{\frakI}^{(\nu)}}
         \left( H _{\nu, \zeta},
         \pi
         \right)
         \neq 0
         \]
         if and only if
         $\pi$ is an essentially tame (supercuspidal) representation
         parametrized by
         a minimal admissible pair 
         $(L/K, \xi)$
         such that $\xi |_{U_L^{\nu}}=\psi _{\zeta \varphi _L^{-\nu}}^L$.
         Moreover, 
         if this space is non-zero,
         then we have
         \[
         \dim \Hom _{U_{\frakI}^{(\nu)}} \left( 
                                                 H _{\nu, \zeta}, \pi
                                                 \right)=\begin{cases}
                                                           1 &\text{ if $\nu$ is odd} \\
                                                           q^{(n-1)/2} &\text{ if $\nu$ is even}.
                                                           \end{cases}
         \]
 \end{Prop}
 
 \begin{proof}
 Define irreducible representations 
 $\rho _1$ and $\rho _2$ of $S_{1, \nu}$ and $S_{2, \nu}$ by expressing
 $H_{\nu, \zeta}\simeq \rho _1\boxtimes \rho _2$ 
 as a representation of $S_{1, \nu}\times S_{2, \nu}$.
 Thus, if $\nu$ is even, $\rho _1=\rho _{1, \overline{\psi} _{\zeta}}$ in the notation of
 Proposition \ref{Prop:evencoh} (2)
 and similarly for $\rho _2$,
 so that 
 \[
 \dim \rho _1=\dim \rho _2=\begin{cases}
                                    1& \text{ if $\nu$ is odd} \\
                                    q^{(n-1)/2} &\text{ if $\nu$ is even.}
                                    \end{cases}
 \]
 Then we need to determine 
 the condition for $\rho _1$ to occur in $\pi$
 and prove that 
 the multiplicity is (at most) one. 
 
 Suppose $\rho _1$ occurs in $\pi$.
 Note first that 
 $L^{\times}$ normalizes
 $C_1$ (appearing in the definition of $U_{\frakI}^{(\nu)}$)
 and the central character of $\rho _1$,
 which in turn implies 
 that it normalizes
 $(U_{\frakI}^{(\nu)}, \rho _1)$.
 Thus $L^{\times}U_{\frakI}^{(\nu)}=L^{\times}U_{\frakI}^{\lfloor (\nu +1)/2\rfloor}$
 acts on the $\rho _1$-isotypic part $V^{\rho _1}$
 of $\pi$.
 Take an irreducible subrepresentation $\Lambda \subset V^{\rho _1}$.
 Let us show that $\Lambda \simeq \Lambda _{\xi}$
 for some $\xi$ such that $\Lambda|_{K^{\times}U_{L}^1}$ is $\xi|_{K^{\times}U_{L}^1}$-isotypic.
 We use the following commutative diagrams
 \[
 \xymatrix{
 	\frakP_{\frakI}^{-\nu}/\frakP_{\frakI}^{-\lfloor \nu/2\rfloor}\ar[r]^-{\sim} \ar[d]_-{\text{can}} &(U_{\frakI}^{\lfloor \nu/2\rfloor +1}/U_{\frakI}^{\nu+1})^{\vee} \ar[d]^-{\text{restriction}} 
 	&\frakP_{\frakI}^{-\nu}/\frakP_{\frakI}^{-\lfloor \nu/2\rfloor}\ar[r]^-{\sim} \ar[d]_-{p_1} &(U_{\frakI}^{\lfloor \nu/2\rfloor +1}/U_{\frakI}^{\nu+1})^{\vee} \ar[d]^-{\text{restriction}} \\
 	\frakP_{\frakI}^{-\nu}/\frakP_{\frakI}^{1-\nu}\ar[r]^-{\sim} &(U_{\frakI}^{\nu}/U_{\frakI}^{\nu+1})^{\vee},
 	&\frakp_{L}^{-\nu}/\frakp_{L}^{-\lfloor \nu/2\rfloor}\ar[r]^-{\sim}
 	&(U_{L}^{\lfloor \nu/2\rfloor +1}/U_{L}^{\nu+1})^{\vee},
 }
 \]
 where the horizontal morphisms are from \eqref{eq:paramofchars} for $L$ and analogously defined for $\frakI$, 
 the morphism ``can" is the canonical projection 
 and $p_1$ is from Proposition \ref{Prop:GL(nu)explicit}.
As $\rho_1|_{U_{\frakI}^{\nu}}$ is $\psi _{\zeta \varphi _L^{-\nu}}$-isotypic and in particular trivial on $U_{\frakI}^{\nu+1}$,
the restriction $\Lambda|_{U_{\frakI}^{\lfloor \nu/2\rfloor +1}}$ decomposes into characters.
Take such a character of $U_{\frakI}^{\lfloor \nu/2\rfloor +1}$ occurring in $\Lambda$
and write $\psi_{\alpha}$ with $\alpha \in \frakP_{\frakI}^{-\nu}$,
so that $\alpha+\frakP_{\frakI}^{1-\nu}=\zeta \varphi_L^{-\nu}+\frakP_{\frakI}^{1-\nu}$ by the first diagram.
Now $\rho_1$ is moreover trivial on $U_{\frakI}^{(\nu+1)}$,
which implies that $\alpha$ lies in fact in $\frakp_{L}^{-\nu}+\frakP_{\frakI}^{-\lfloor \nu/2\rfloor}$. 
Therefore we may assume that $\alpha \in \frakp_{L}^{-\nu}$.
The discussions in \cite[\S5.6, 5.7]{CarCusp},
which treat more general cases,
yield in this case 
the classification of
irreducible representations of 
the normalizer of $(U_{\frakI}^{\lfloor \nu/2\rfloor +1}, \psi_{\alpha})$ in $\varphi^{\bbZ}U_{\frakI}=L^{\times}U_{\frakI}$
containing $\psi _{\alpha}$
when restricted to $U_{\frakI}^{\lfloor \nu/2\rfloor+1}$.
Furthermore, by \cite[Proposition 3.6]{CarCusp} 
the normalizer is 
$K(\alpha)^{\times}U_{\frakI}^{\lfloor (\nu +1)/2\rfloor}
=L^{\times}U_{\frakI}^{\lfloor (\nu +1)/2\rfloor}$.
Hence $\Lambda$ arises in the following way:
\begin{itemize}
	\item Extend $\psi _{\alpha}$ to a character $\Theta$ of $K^{\times}U_L^1U_{\frakI}^{\lfloor \nu /2\rfloor +1}$.
	\item Take an irreducible representation $\Upsilon$ of $K^{\times}U_L^1U_{\frakI}^{\lfloor (\nu +1)/2\rfloor}$ containing $\Theta$, which is in fact unique up to isomorphism and $\Theta$-isotypic.
	\item Extend $\Upsilon$ to an irreducible representation $\Lambda$ of $L^{\times}U_{\frakI}^{\lfloor (\nu +1)/2\rfloor}$.
\end{itemize}
Thus if we set $\xi_0=\Theta|_{K^{\times}U_L^1}$, then
the restriction of $\Lambda$ to $K^{\times}U_L^1U_{\frakI}^{\lfloor \nu /2\rfloor +1}$
(resp.\ to $K^{\times}U_L^1$) 
is $\Theta$-isotypic (resp.\ $\xi_0$-isotypic).
We take an extension $\xi$ of $\xi_0$ to $L^{\times}$ and form the character $\theta_{\xi}$ of $U_L^1U_{\frakI}^{\lfloor \nu/2\rfloor +1}$ as defined before Lemma \ref{lem:constoflambda}.
It is independent of the choice of $\xi$ and we denote it by $\theta_{\xi_0}$. 
By the definition of $\theta_{\xi_0}$ and the second diagram we have 
$\theta_{\xi_0}|_{U_{\frakI}^{\lfloor \nu/2\rfloor +1}}=\psi_{p_1(\alpha)}=\psi_{\alpha}$.
Therefore, $\theta_{\xi_0}$ is equal to the restriction of $\Theta$.
It easily follows that $\eta$ constructed from $\theta_{\xi_0}$ is isomorphic to the restriction of $\Upsilon$,
and the possible $n$ extensions of $\xi_0$ to $L^{\times}$ corresponds precisely to the possible $n$ extensions of $\Upsilon$ to $L^{\times}U_{\frakI}^{\lfloor (\nu+1)/2\rfloor}$.
We have shown that $\Lambda \simeq \Lambda_{\xi}$ 
for some $\xi$ extending $\xi_0$ as claimed, and in particular
$\xi|_{U_{L}^{\nu}}=\psi_{\zeta\varphi_L^{-\nu}}^L$.
Hence we obtain a homomorphism
$\pi _{\xi}=\cInd \Lambda _{\xi} \rightarrow \pi$,
which is an isomorphism
by the irreducibility of $\pi$.
The converse being easy,
we deduce the desired condition for the occurrence of $\rho _1$.
  
 Since $\dim \Lambda _{\xi} =\dim \rho _1$ as above,
 the claim about the multiplicity
 is reduced to certain
 multiplicity one statement 
 in the theory of types.
 We can argue as follows.
 Take once again 
 a subrepresentation $\Lambda '$ in $V^{\rho _1}$.
 By the above argument
 we have
 $\Lambda '\simeq \Lambda _{\xi '}$
 for some $\xi '$
 and $\pi _{\xi}\simeq \pi _{\xi '}$.
 Then $\xi ^{\sigma}=\xi '$ 
 for some $\sigma \in \Aut (L/K)$
 by the injectivity of the parametrization,
 which implies $\zeta \varphi _L^{-\nu}=\zeta (\varphi _L^{-\nu})^{\sigma}$
 and hence $\sigma =\mathrm{id}$.
 This shows that $V^{\rho _1}$ is $\Lambda _{\xi}$-isotypic.
 Therefore, 
 \[
 \dim \Hom _{U_{\frakI}^{(\nu)}} (\rho _1, \pi)
 =\dim \Hom _{L^{\times}U_{\frakI}^{\lfloor (\nu +1)/2\rfloor}} (\Lambda _{\xi}, \pi)
 =\dim \End _{\GL _n(K)} (\pi)=1
 \]
 as desired
 (see also \cite[\S15.7 Proposition (3)]{BHLLCGL2}).
 \end{proof}

 \begin{Prop} \label{Prop:homasrep}
         Let $\pi$ be as above.
         Suppose that 
         $\Hom _{U_{\frakI}^{(\nu)}}
         \left( H _{\nu},
         \pi
         \right)
         \neq 0$,
         so that
         $\pi \simeq \pi _{\xi}$
         for some minimal admissible pair 
         $(L/K, \xi)$
         by Proposition \ref{Prop:occurrencecond}.
         Then this space contains
         \[
         \Lambda _{\iota \xi}^{D} \boxtimes {}_{K}\mu _{\xi}^{-1}\xi
         \]
         as a representation of
         $L^{\times}U_D^{\lfloor (\nu +1)/2\rfloor} \times W_L \subset \overline{\Stab} _{\nu}$.
 \end{Prop}
 \begin{proof}
 By Proposition \ref{Prop:occurrencecond}
 we have
 $\Hom _{U_{\frakI}^{(\nu)}}
  \left( H_{\nu, \zeta},
         \pi
         \right)
  \neq 0$
 for some element $\zeta \in \mu _{q-1}(K)$
 such that
 $\xi |_{U_L^{\nu}}=\psi _{\zeta \varphi _L^{-\nu}}^L$.  
 We claim that this subspace 
 $\Hom _{U_{\frakI}^{(\nu)}}
  \left( H_{\nu, \zeta},
         \pi
         \right)
  \subset 
  \Hom _{U_{\frakI}^{(\nu)}}
  \left( H_{\nu},
         \pi
         \right)$
 is isomorphic to the representation
 $\Lambda _{\iota \xi}^{D} \boxtimes {}_{K}\mu _{\xi}^{-1}\xi$ 
 appearing in the assertion.
 
 First 
 it is indeed stable under the action of
 $L^{\times}U_D^{\lfloor (\nu +1)/2\rfloor} \times W_L 
 \subset \overline{\Stab} _{\nu}$
 because
 the action of the pull-back of 
 $L^{\times}U_D^{\lfloor (\nu +1)/2\rfloor}\times W_L$
 in $\Stab _{\nu}$ 
 (cf.\ Remark \ref{Rem:Stabastorsor})
 on 
 $H_{\nu}=\bigoplus _{\zeta '}H_{\nu, \zeta'}$
 preserves each summand.
 
 By the proof of Proposition \ref{Prop:occurrencecond},
  $\Hom _{U_{\frakI}^{(\nu)}}
  (H _{\nu, \zeta}, \pi _{\xi})
  =\Hom _{U_{\frakI}^{(\nu)}}
  (H _{\nu, \zeta}, \Lambda _{\xi})$
 is isomorphic to
 $\rho _{2}^{\vee}$ 
 (inflated via $U_D^{(\nu)}\rightarrow S_{2, \nu}$)  
 as a representation of 
 $U_D^{(\nu)}$.
 Therefore we may set
 $\Hom _{U_{\frakI}^{(\nu)}}
  \left( H_{\nu, \zeta},
         \pi
         \right)
  =\Hom _{U_{\frakI}^{(\nu)}}
  (H _{\nu, \zeta}, \Lambda _{\xi})
  =\Lambda '\boxtimes \xi '$
  with some irreducible smooth representation 
  $\Lambda '$ of $L^{\times}U_D^{\lfloor (\nu +1)/2\rfloor}$
  whose restriction to $U_D^{(\nu)}$ 
  is isomorphic to $\rho _{2}^{\vee}$
  and some smooth character $\xi '$ of
  $W_L$.
  
 Let us first show 
 $\Lambda '\simeq \Lambda _{\iota \xi}^{D}$.
 The action of $x\in L^{\times}\subset L^{\times}U_D^{\lfloor (\nu +1)/2\rfloor}$
 is given by the composition of the action of $(x, x, 1)^{-1}\in \Stab _{\nu}$ on the source
 and that of $x\in L^{\times}U_{\frakI}^{\lfloor (\nu +1)/2\rfloor}$ on the target.
 As $(x, x, 1)\in \Stab _{\nu} \ (x\in U_L)$ acts trivially 
 by Theorem \ref{Thm:stabandaction} \ref{item:actionofU_L} 
 and $\Lambda _{\xi} |_{U_L}$ is 
 $\xi|_{U_L}$-isotypic
 by Lemma \ref{lem:constoflambda} \ref{item:isotypic},
 we find that $U_L$ acts via the character $\xi$
 on ${\Lambda '}$, as desired.
 Thus,
 to conclude 
 $\Lambda '\simeq \Lambda _{\iota \xi}^{D}$, 
 we need to show
 \begin{equation} \label{eq:trforD}
 \tr \Lambda '(\varphi _D^j)=\epsilon _{\iota \xi}^D(\iota \xi) (\varphi _L^j)
 \end{equation}
 for any $j$ coprime to $n$
 by (\ref{eq:charreloflambda}).
 Let us express the left-hand side in terms of the traces of 
 $H_{\nu, \zeta}$ and $\Lambda_{\xi}$.
 Recall $S_{1, \nu}={U_{\frakI}^{(\nu)}}/{U_{\frakI}^{(\nu +1)}}$
 and write $W$ for the representation space of $\Lambda_{\xi}$.
 We view $\Hom _{U_{\frakI}^{(\nu)}}
 (H _{\nu, \zeta}, \Lambda _{\xi})=\Hom _{U_{\frakI}^{(\nu)}}
 (H _{\nu, \zeta}, W)=\Hom _{S_{1, \nu}}
 (H _{\nu, \zeta}, W)$
 as the $S_{1, \nu}$-fixed part of $\Hom _{\overline{\bbQ}_{\ell}}
 (H _{\nu, \zeta}, W)=H _{\nu, \zeta}^{\vee}\otimes W$
 and write $\mathrm{pr}_{S_{1, \nu}}$ for the projection.
 The automorphisms $(\varphi^j, \varphi_{D}^j, 1)$ 
 and $\varphi^j$ of $H _{\nu, \zeta}$
 and $W$
 (as $\overline{\bbQ}_{\ell}$-vector spaces) 
 together induce
 an automorphism $(\varphi^j, \varphi_{D}^j, 1)\otimes \varphi^j $ of $H _{\nu, \zeta}^{\vee}\otimes W$
 which restricts to the automorphism of the $S_{1, \nu}$-fixed part 
 whose trace we need to compute.
 Hence we have
 \begin{align*}
 \tr \Lambda '(\varphi _D^j)
 &=\tr \left( \left( (\varphi^j, \varphi_{D}^j, 1)\otimes \varphi^j \right)
 \circ \mathrm{pr}_{S_{1, \nu}} \middle|
  H _{\nu, \zeta}^{\vee}\otimes W \right) \\
 &=|S_{1, \nu}|^{-1}
 \sum _{x\in S_{1, \nu}}
 \tr \left( \left( \varphi ^jx, \varphi _{D}^j, 1\right) \middle| H_{\nu, \zeta}^{\vee} \right)
 \tr \Lambda _{\xi} (\varphi ^jx) \\
 &=|S_{1, \nu}|^{-1}
 \sum _{x\in S_{1, \nu}}
 \tr \left( \left( \varphi ^jx, \varphi _{D}^j, 1\right)^{-1}\middle| H_{\nu, \zeta} \right)
 \tr \Lambda _{\xi} (\varphi ^jx).
 \end{align*}
 A key to further computation is the fact that one can take a convenient conjugate of $\varphi^jx$ for each $x\in S_{1, \nu}$. 
 This follows from general results in \cite[Appendix 1]{BHtameliftII}, but
 we extract an argument for our simple situation.\footnote{The conclusion below is trivial if $\nu$ is odd and thus $Z=S_{1, \nu}$; one can take $y=1, z=x$.}
 Let $Z$ be the center of $S_{1, \nu}$. 
 Let $\gamma$ be the automorphism of the $\bbF_p$-vector space $Q=S_{1, \nu}/Z$ induced by
 the conjugation action of $\varphi$ on $S_{1, \nu}$.
 As $j$ and $n$ are coprime, $\gamma^{-j}$ has the trivial fixed point and 
 hence $\gamma^{-j}-\mathrm{id}$ 
 is also invertible.
 By applying $(\gamma^{-j}-\mathrm{id})^{-1}$ to the image of $x$ in $Q$,
 we see that there exist elements $y\in S_{1, \nu}$ and $z\in Z$
 such that $x=\varphi^{-j}y\varphi^{j}y^{-1}z$,
 namely
 $\varphi^{j}x=y\varphi^{j}zy^{-1}$.
 As $H_{\nu, \zeta}|_{S_{1, \nu}}$ and $\Lambda_{\xi}|_{S_{1, \nu}}$ share the central characters,  
 \begin{align*}
 &\tr \left( \left( \varphi ^jx, \varphi _{D}^j, 1\right)^{-1}\middle| H_{\nu, \zeta} \right)
 \tr \Lambda _{\xi} (\varphi ^jx) \\
 &=\tr \left( \left( y\varphi^{j}zy^{-1}, \varphi _{D}^j, 1\right)^{-1}\middle| H_{\nu, \zeta} \right)
 \tr \Lambda _{\xi} (y\varphi^{j}zy^{-1}) \\
 &=\tr \left( \left( \varphi^{j}, \varphi _{D}^j, 1\right)^{-1}\middle| H_{\nu, \zeta} \right)
 \tr \Lambda _{\xi} (\varphi^{j}) 
 =(-1)^{n-1}\epsilon _{\xi}\xi (\varphi _L^j),
 \end{align*}
 where
 we use Theorem \ref{Thm:stabandaction} \ref{item:actionofvarphi}, Proposition \ref{Prop:oddcoh} \ref{item:oddcyclic},
 Proposition \ref{Prop:evencoh} \ref{item:evencyclic} and 
 (\ref{eq:charreloflambda})
 in the last equality.
 This shows $\tr \Lambda '(\varphi _D^j)=(-1)^{n-1}\epsilon _{\xi}\xi (\varphi _L^j)$.
 Finally the equality (\ref{eq:trforD}) follows 
 since $\epsilon _{\iota \xi}^D=\epsilon _{\xi}$
 for any $\nu$
 by Proposition \ref{Prop:sympsign}
 and one can check $\iota(\varphi_L^j)=(-1)^{j(n-1)}=(-1)^{n-1}$.
 
 To prove $\xi '= {}_{K}\mu _{\xi}^{-1}\xi$
 we check 
 \begin{equation} \label{eq:trforWeil}
 \tr \left( (a_{\sigma}^{-1}, \sigma) \middle| 
            \Hom _{U_{\frakI}^{(\nu)}}
            \left( H_{\nu, \zeta},
            \Lambda _{\xi}
            \right) 
      \right)
  =\tr \left( \Lambda _{\iota \xi}^{D} \boxtimes {}_{K}\mu _{\xi}^{-1}\xi \right)
        (a_{\sigma}^{-1}, \sigma)
 \end{equation}
 for any $\sigma \in W_L$ with $n_{\sigma}=v_L(a_{\sigma})=-1$;
 this is sufficient,
 as these traces will turn out to be non-zero
 and a character of $W_L$ is clearly determined by 
 its restriction to $\{ \sigma \in W_L\mid n_{\sigma}=-1\}$.
 
 Now we proceed by cases.
 First suppose that $\nu$ is even.
 Noting the twist and the multiplicity of $\Lambda _{\xi}$ in $H_{\nu, \zeta}$,
 we see 
 \[
 \tr \left( (a_{\sigma}^{-1}, \sigma) \middle| 
            \Hom _{U_{\frakI}^{(\nu)}}
            \left( H_{\nu, \zeta},
            \Lambda _{\xi}
            \right) 
      \right)=1
 \] 
 by Theorem \ref{Thm:stabandaction} \ref{item:actionofWeilgp}
 and Proposition \ref{Prop:evencoh} \ref{item:evennondeg}.
 By (\ref{eq:charreloflambda}), Theorems \ref{Thm:GLrect} and \ref{Thm:Drect}, 
 we have
 \[
 \tr \left( \Lambda _{\iota \xi}^{D} \boxtimes {}_{K}\mu _{\xi}^{-1}\xi \right)
        (a_{\sigma}^{-1}, \sigma)
 =(-1)^{n-1}\epsilon _{\xi}^D 
  \lambda _{L/K}(\psi).
 \] 
 As $n$ is odd,
 the equality (\ref{eq:trforWeil}) follows from 
 Proposition \ref{Prop:sympsign} 
 and Proposition \ref{Prop:valuesforinvs} \ref{item:valueforoddlambda}.
 
 Suppose next that $\nu$ is odd.
 Put $u_{\sigma}=a_{\sigma}\varphi _D\in U_L$
 and $\frakm (\overline{\psi} _{\zeta})=q^{-1/2}\frakg (\overline{\psi} _{\zeta})
 ={\overline{\zeta} \overwithdelims ()k}\frakm (\overline{\psi})$.
 We have
 \begin{align*}
 &\tr \left( (a_{\sigma}^{-1}, \sigma) \middle| 
            \Hom _{U_{\frakI}^{(\nu)}}
            \left( H_{\nu, \zeta},
            \Lambda _{\xi}
            \right) 
      \right) \\
 &\quad =\begin{cases}
                   {q \overwithdelims ()n}
                     &\text{if $p=2$ and $n$ is odd} \\
                   {n \overwithdelims ()k} \frakm (\overline{\psi} _{\zeta})^{n-1}
                     &\text{if $p\neq 2$ and $n$ is odd} \\
                   -{{\overline{u} _{\sigma}} \overwithdelims ()k}^{n-1}
                     {{-1} \overwithdelims ()k}^{(\nu -1)/2}
                     {2 \overwithdelims ()k}
                     {n \overwithdelims ()k} 
                     \frakm (\overline{\psi} _{\zeta})^{n-1}
                     &\text{if $p\neq 2$ and $n$ is even},   
             \end{cases}
 \end{align*}
 by Theorem \ref{Thm:stabandaction} \ref{item:actionofWeilgp} and
 Proposition \ref{Prop:oddcoh} \ref{item:oddnondeg}, \ref{item:oddcyclic}.
 On the other hand, 
 we have
 \begin{align*}
 \tr \left( \Lambda _{\iota \xi}^{D} \boxtimes {}_{K}\mu _{\xi}^{-1}\xi \right)
        (a_{\sigma}^{-1}, \sigma)
 &=(-1)^{n-1}{}_{K}\mu _{\xi}^{-1}(a_{\sigma}) \\
 &=\begin{cases}
     \lambda _{L/K}(\psi ) 
      &\text{if $n$ is odd} \\
     -\delta _{L/K}(u_{\sigma}^{-1}) 
      {\zeta \overwithdelims ()k}
      {{-1} \overwithdelims ()k}^{(\nu -1)/2}
      \lambda _{L/K}(\psi )
      &\text{if $n$ is even},
    \end{cases}.
 \end{align*}
 by Theorems \ref{Thm:GLrect}, \ref{Thm:Drect} 
 and Lemma \ref{lem:constoflambda}.
 Now the equality (\ref{eq:trforWeil}) follows 
 from Proposition \ref{Prop:valuesforinvs} \ref{item:valueforoddlambda}
 if $p=2$,
 and from Proposition \ref{Prop:valuesforinvs} \ref{item:valuefordelta}
 and the equalities
 \[
 \lambda _{L/K}(\psi )=\begin{cases}
                               {n \overwithdelims ()q} 
                               \frakm (\overline{\psi} )^{n-1}
                                &\text{if $n$ is odd} \\
                               {2 \overwithdelims ()k}
                               {n \overwithdelims ()k} 
                               \frakm (\overline{\psi} )^{n-1} 
                                &\text{if $n$ is even},
                             \end{cases}
 \]
 which appear in \cite[(5.22)]{ITepitame},
 if $p\neq 2$. 
 \end{proof}
 
 Let $\LJ (\pi)$ (resp.\ $\LL (\pi)$)
 be the image of $\pi$ under the local Jacquet-Langlands correspondence
 (resp.\ the local Langlands correspondence).
 \begin{Thm} \label{Thm:realizationinPi}
 Let $\pi$ be as above.
 We have
         \[
         \Hom _{\GL _n(K)}
         \left( \cInd _{\Stab _{\nu}}^G\Pi _{\nu},
         \pi
         \right)
         \neq 0
         \]
         if and only if
         $\pi$ is parametrized by
         a minimal admissible pair 
         $(L/K, \xi)$
         with the jump at $\nu$.
 Moreover, 
 if non-zero,
 this space is isomorphic to
 $\LJ (\pi) \boxtimes \LL (\pi)$
 as a representation of 
 $D^{\times}\times W_K$.        
 \end{Thm}
 \begin{proof}
  The first assertion follows immediately 
  from Lemma \ref{lem:homfromind},
  Proposition \ref{Prop:redcoh},
  Proposition \ref{Prop:occurrencecond}
  and the Frobenius reciprocity.
  
  To prove the second assertion
  let $\pi \simeq \pi _{\xi}$ occur in $\cInd _{\Stab _{\nu}}^G\Pi _{\nu}$.
  By Theorems \ref{Thm:GLrect}, \ref{Thm:Drect}
  it suffices to show that the following morphism, induced by the Frobenius reciprocity, 
  is an isomorphism;
  \[
  \Ind _{L^{\times}U_D^{\lfloor (\nu +1)/2\rfloor} \times W_L}^{\overline{\Stab} _{\nu}}
  \left( \Lambda _{\iota \xi}^{D} \boxtimes {}_{K}\mu _{\xi}^{-1}\xi \right)
  \rightarrow \Hom _{U_{\frakI}^{(\nu)}}(\Pi _{\nu}, \pi).
  \]
  Since the source is irreducible,
  we only need to show the equality of the dimensions.
  By Remark \ref{Rem:Stabastorsor} we have 
  \begin{align*}
  \dim \Ind _{L^{\times}U_D^{\lfloor (\nu +1)/2\rfloor} \times W_L}^{\overline{\Stab} _{\nu}}
  \left( \Lambda _{\iota \xi}^{D} \boxtimes {}_{K}\mu _{\xi}^{-1}\xi \right)
  &=\dim \Lambda _{\iota \xi}^{D}\cdot [\overline{\Stab} _{\nu} : L^{\times}U_D^{\lfloor (\nu +1)/2\rfloor} \times W_L] \\
  &=\dim \Lambda _{\iota \xi}^{D}\cdot [W_{L'} : W_L].
  \end{align*}
  By Proposition \ref{Prop:occurrencecond}
  we are reduced to showing
  that $[W_{L'} : W_L]$ 
  equals the number of $\zeta \in \mu _{q-1}(K)$
  such that $\Hom _{U_{\frakI}^{(\nu)}}(\Pi _{\nu, \zeta}, \pi)\neq 0$.
  This can be done readily
  by the injectivity of the paramatrization of essentially tame representations.
 \end{proof}
 
 Now recalling that a totally ramified extension $L/K$ is 
 arbitrarily given after Lemma \ref{lem:Lpi},
 we complete the proof of Main Theorem in the introduction.

\end{document}